\documentclass[11pt,letterpaper]{amsart}
\usepackage{verbatim}
\usepackage{lscape}
\usepackage{tabularx}
\usepackage{appendix}
\usepackage{amsmath}
\usepackage{thmtools}

\usepackage{upgreek}
\usepackage{tikz-cd}
\usepackage{enumitem}
\usepackage{amsthm}
\usepackage{extarrows}
\usepackage{xypic}
\usepackage{stmaryrd}
\usepackage{graphicx}
\usepackage{float}
\usepackage{graphics,amssymb}
\usepackage{amsfonts}
\usepackage{latexsym}
\usepackage{epsf}
\usepackage{todonotes}
\usepackage{mathrsfs}
\usepackage{bbm}
\usepackage{hyperref}
\hypersetup{
    colorlinks=true,
    linkcolor=blue,
    urlcolor=blue, 
    citecolor= blue}
\usepackage[noabbrev, capitalize, nameinlink]{cleveref}
\usepackage[framemethod=tikz]{mdframed}
\usepackage{ulem}
\usepackage{cancel}
\usepackage{tikz}
\usetikzlibrary{matrix,arrows}
\usepackage{mathtools}
\usepackage{url}

\usepackage{fancybox}
\usepackage{dsfont}
\usepackage{bbding}
\usepackage{setspace}
\usepackage{booktabs}
\usepackage{spectralsequences}
\usepackage{array}
\usetikzlibrary{fit}

\newtheorem{theoremA}{Theorem}

\crefname{theoremA}{Theorem}{Theorems}
\crefname{propositionA}{Proposition}{Propositions}

\newtheorem{propositionA}[theoremA]{Proposition}
\newtheorem{theorem}{Theorem}[section]
\newtheorem{lemma}[theorem]{Lemma}

\newtheorem{lem/def}[theorem]{Lemma/Definition}
\newtheorem{proposition}[theorem]{Proposition}
\newtheorem{corollary}[theorem]{Corollary}

\newtheorem*{theorem*}{Theorem}
\newtheorem*{corollary*}{Corollary}
\newtheorem*{proposition*}{Proposition}

\theoremstyle{definition}
\newtheorem{remark}[theorem]{Remark}
\newtheorem*{remark*}{Remark}
\newtheorem{definition}[theorem]{Definition}
\newtheorem*{definition*}{Definition}

\newcommand{\TT}{\mathbb{T}} 

\newcommand{\Gal}[2]{\mathrm{Gal}(#1/#2)}

\newcommand{\Z}{\mathbb{Z}}
\newcommand{\Q}{\mathbb{Q}}

\newcommand{\Aa}{\mathfrak{a}}
\newcommand{\cc}{\mathfrak{c}}

\newcommand{\mm}{\mathfrak{m}}

\newcommand{\OK}{\mathcal{O}}
\newcommand{\m}{\mathfrak{m}}
\newcommand{\n}{\mathfrak{n}}
\newcommand{\pp}{\mathfrak{p}}
\newcommand{\F}{\mathbb{F}}

\newcommand{\lb}{\llbracket}
\newcommand{\rb}{\rrbracket}
\newcommand{\zetadelta}{\zeta_{\Delta}}

\DeclareMathOperator{\diff}{\mathfrak{d}}
\DeclareMathOperator{\Div}{\mathrm{Div}}

\DeclareMathOperator{\End}{\mathrm{End}}

\DeclareMathOperator{\Frob}{\mathrm{Frob}}
\DeclareMathOperator{\GL}{\mathrm{GL}}
\DeclareMathOperator{\Hom}{\mathrm{Hom}}

\DeclareMathOperator{\Ind}{\mathrm{Ind}}

\DeclareMathOperator{\ord}{\mathrm{ord}}
\DeclareMathOperator{\PGL}{\mathrm{PGL}}

\DeclareMathOperator{\rk}{\mathrm{rk}}

\DeclareMathOperator{\SL}{SL}

\DeclareMathOperator{\tr}{\mathrm{tr}}

\newcommand{\II}{\mathfrak{I}}

\newcommand{\Res}{\mathrm{Res}}

\newcommand{\p}{\mathfrak{p}}

\title[The rank of Mazur's Eisenstein Hecke algebra]{A new perspective on the rank of Mazur's Eisenstein Hecke algebra}
\author{Jaclyn Lang, Katharina Müller, Bharathwaj Palvannan}
\allowdisplaybreaks

\subjclass[2020]{Primary 11F33; Secondary 11F80, 11F85}

\begin{document}
\begin{abstract}
Let $N, p \geq 5$ be primes such that $N \equiv 1 \bmod p$. We study the rank $r$ of the Hecke algebra that parametrizes modular forms of weight 2 and level $N$ that are Eisenstein modulo $p$.  When $r$ is $2$ or $3$, we prove that $r-1$ equals the order of vanishing of the mod-$p$ reduction of a zeta element that interpolates Dirichlet $L$-values at $-1$, thereby recovering results of Merel and Lecouturier. This equality can fail in some cases when $r \geq 4$, and we provide a heuristic explanation of this failure. Our approach handles all of these cases uniformly by studying the analogous Hecke algebra in level $N^2$. When exactly one of $r-1$ or the order of vanishing equals $3$, we provide precise information about Galois orbits of cuspidal newforms in level $N^2$ that are Eisenstein modulo $p$. 
\end{abstract}

\maketitle

\section{Introduction}\label{sec:introduction}
In his seminal paper on the Eisenstein ideal, Mazur studies congruences modulo a prime $p$ between cusp forms and the unique Eisenstein series of weight $2$ and prime-level $\Gamma_0(N)$ \cite{Mazur}.  When $N, p \geq 5$, he shows that such congruences exist if and only if $N \equiv 1 \bmod p$; we assume these conditions on the primes $N$ and $p$ throughout.  Moreover, the associated $p$-adic Eisenstein Hecke algebra $\TT(N)$ that parametrizes such congruences is a monogenic $\Z_p$-algebra that is free of finite rank \cite{Mazur}; let $r$ denote its $\Z_p$-rank. Mazur asked, ``\textit{Is there anything general that can be said about $\ldots \ r$?}''   \cite[Page 140]{Mazur}.  This question  and related arithmetic applications were considered by Merel \cite{Merel}, Calegari and Emerton \cite{CE2005}, Wake and Wang-Erickson \cite{WWE}, and Lecouturier \cite{Lecouturier}.  

In the context of determining $r$, it is possible to unify the results of Merel \cite{Merel} and Lecouturier \cite{Lecouturier} and frame them in terms of zeta values.  Such a formulation, which we now explain, is due to Calegari and Venkatesh \cite[\S 1.5]{WWE}. Let
\[
\zeta \coloneqq \frac{-N}{2}\sum_{i=1}^{N-1} B_2(i/N)[i + N\Z] \in \Z_p[(\Z/N\Z)^\times],
\]
where $B_2(x)$ denotes the second Bernoulli polynomial $x^2-x+1/6$. Write $\bar{\zeta} \in \F_p[(\Z/N\Z)^\times]$ for the reduction of $\zeta$ modulo $p$. Let $\mathcal{I} \subset \F_p[(\Z/N\Z)^\times]$ be the augmentation ideal, which is easily shown to contain $\bar{\zeta}$. The \textit{order of vanishing} of $x \in \mathcal{I}$, denoted $\ord(x)$, is the largest $m$ such that $x \in \mathcal{I}^m$. 

We give new, uniform, and transparent proofs of the results of Merel \cite[Théorème 2]{Merel} and Lecouturier \cite[Theorem 1.2]{Lecouturier} that determine $r$ in terms of $\ord(\bar{\zeta})$, namely the following theorems.

\begin{theoremA}\label{thmA:merel}  
We have $r = 2$ if and only if $\ord(\bar{\zeta}) = 1$.
\end{theoremA}

\begin{theoremA}\label{thmA:lecouturier} 
We have $r = 3$ if and only if $\ord(\bar{\zeta}) = 2$.
\end{theoremA}

\cref{thmA:merel,thmA:lecouturier} were first proved by Merel \cite{Merel} and Lecouturier \cite{Lecouturier}, respectively.  Therefore we refer to \cref{thmA:merel} as \textit{Merel's criterion} and to \cref{thmA:lecouturier} as \textit{Lecouturier's criterion}.  In fact, Lecouturier gave a uniform proof of \cref{thmA:merel,thmA:lecouturier} using the higher Eisenstein elements.

The naive generalization $\ord(\bar{\zeta}) = r-1$ is known to fail in some situations when either $\ord(\bar{\zeta})$ or $r-1$ is at least $3$ \cite[Table 1]{WWE}. Our methods provide a heuristic explanation of this failure using what we term \textit{spoiler coefficients}; see \cref{sub:higherrank}.  To the best of our knowledge, providing an insight into this failure is novel. 
Furthermore, in the first degenerate case when exactly one of $\ord(\bar{\zeta})$ and $r-1$ equals $3$, we provide precise information about the number of Galois orbits of cuspidal newforms of level $\Gamma_0(N^2)$ that are \textit{Eisenstein modulo $p$}\footnote{In the introduction, a new cusp form $f$ of level $N^2$ is \textit{Eisenstein modulo} $p$ if for all primes $\ell \neq N$, the $\ell$-th Fourier coefficient of $f$ is congruent to $\ell + 1$ modulo a prime above $p$.} and the degrees of their $p$-adic Hecke fields.  If $f$ is a cuspidal newform of level $\Gamma_0(N^2)$ that is Eisenstein modulo $p$, then the local-at-$N$ component of its associated automorphic representation is ramified principal series.  Let $p^s$ be the exact power of $p$ dividing $N-1$.  By local Langlands for $\GL_2$, the restriction of the $p$-adic Galois representation of $f$ to the inertia group at $N$, denoted $I_N$, has order $p^t$ for some $1 \leq t \leq s$.

\begin{theoremA}\label{thmA:higherrank}
Suppose that exactly one of $\ord(\bar{\zeta})$ and $r-1$ is equal to $3$.  There are $s$ Galois orbits of normalized eigenforms that are new of level $\Gamma_0(N^2)$ and Eisenstein modulo $p$. Two such forms belong to the same orbit if and only if the restrictions of their Galois representations to $I_N$ have the same order, say $p^t$.  The $p$-adic Hecke field of a form in that orbit is totally ramified of degree $\frac{1}{2}(p-1)p^{t-1}(r-2)$ over $\Q_p$.
\end{theoremA}

A key insight of our approach to Mazur's question is that it is related to Eisenstein-cuspidal congruences in level $\Gamma_0(N^2)$. It is for this reason that one obtains information about cuspidal newforms of level $\Gamma_0(N^2)$ in \cref{thmA:higherrank}.  This is natural since, due to work of Carayol, all new Eisenstein-cuspidal congruences of $N$-power level occur in level dividing $N^2$ \cite[\S 1.1]{Carayol89}. 

In moving to level $\Gamma_0(N^2)$ we make crucial use of a recent result by the first author with Pollack and Wake.  They show, using the modular representation theory of $\PGL_2(\F_N)$, that the analogue of Mazur's Hecke algebra in level $\Gamma_0(N^2)$, denoted $\TT$, is free of rank $r$ over a certain inertia-at-$N$ pseudodeformation ring, which we describe below \cite{LangPollackWake}.  It gives us three key tools: a formula for the $\Z_p$-rank of $\TT$ in terms of $r$, a fairly explicit presentation of $\TT$, and an equidistribution result among inertia-at-$N$ types for cuspidal newforms of level $\Gamma_0(N^2)$ that are Eisenstein modulo $p$. 

We now describe our methods for proving \cref{thmA:merel,thmA:lecouturier,thmA:higherrank}, pointing out a few results we obtain along the way that may be of independent interest.

\subsection{Summary of proof techniques}
Our strategy is to find the two pieces of data we want to relate --- $\ord(\bar{\zeta})$ and $r$ --- in descriptions of the Hecke algebra $\TT$ of level $\Gamma_0(N^2)$.  In the case of $r$, this is done through the freeness result of the first author with Pollack and Wake.  We obtain a presentation of $\TT$ with two variables and two relations that encodes $r$.  On the other hand, we show that $\zeta$ is closely related to a generator of the congruence ideal in level $\Gamma_0(N^2)$.  This gives a second description of $\TT$ as a fiber product of its cuspidal and Eisenstein parts.  By combining these two descriptions of $\TT$ and considering the map to the cuspidal quotient, one arrives at a combinatorial game that allows us to relate $\ord(\bar{\zeta})$ and $r$ when these quantities are small.

\subsubsection{Ring theoretic properties of $\TT$} \label{subsubsec:ringtheoreticpropTintro}
 In this section, we state the freeness result of the first author with Pollack and Wake.  We explain how this gives a presentation of $\TT$ that sees $r$.  Moreover, we deduce several ring theoretic properties of $\TT$, some of which may be of independent interest. 
 
 We have noted that the first author with Pollack and Wake show that $\TT$ is a free module of rank $r$ over a certain inertia-at-$N$ pseudodeformation ring.  Rather than defining the pseudodeformation ring here, we give an explicit description of it in our case.  Let $\Delta$ denote the maximal $p$-power quotient of $(\Z/N\Z)^\times$.  The inertia-at-$N$ pseudodeformation ring is isomorphic to the subring $\Z_p[\Delta]^+$ of the group algebra $\Z_p[\Delta]$ that is fixed by the involution that inverts group-like elements.  Thus the freeness theorem is as follows.

\begin{theorem*}[Freeness Theorem]\label{thm*:freeness}\cite{LangPollackWake}
The ring $\TT$ is a $\Z_p[\Delta]^+$-algebra, and it is free of rank $r$ as a $\Z_p[\Delta]^+$-module.  Hence $\rk_{\Z_p} \TT = \frac{1}{2}(p^s+1)r$
\end{theorem*}

By counting the Eisenstein series at level $\Gamma_0(N^2)$ and forms of level $\Gamma_0(N)$, one easily arrives at a formula for the number of normalized cuspidal newforms of level $\Gamma_0(N^2)$ that are Eisenstein modulo $p$.  One can strengthen this to the following equidistribution result; see \cref{prop:equidistributionofcharacters}.  It plays an important role when $r \geq 3$, namely in the proofs of \cref{thmA:lecouturier,thmA:higherrank}.

\begin{corollary*}[Equidistribution Result]\label{cor:equidistribution}
There are $\frac{1}{2}(p^s-1)(r-2)$ normalized new cusp forms of level $\Gamma_0(N^2)$ that are Eisenstein modulo $p$. They are equidistributed among the inertia-at-$N$ types, with $r-2$ of each type.
\end{corollary*}
 
We now use the \nameref{thm*:freeness} to describe a presentation of $\TT$ that sees $r$.  Using the modularity theorem of Wake--Wang-Erickson \cite[Corollary 7.1.3]{WWE}, we show that one recovers Mazur's Hecke algebra $\TT(N)$ from $\TT$ by reducing modulo the augmentation ideal $\II^+$ of $\Z_p[\Delta]^+$. That is, 
\begin{align}\label{eq:fullmoduloaug}
\TT / \II^+ \TT \cong \TT(N).
\end{align}
Using Mazur's result that $\TT(N)$ is monogenic, any lift of a generator of $x \in \TT(N)$ to $\tilde{x} \in \TT$ generates $\TT$ as a $\Z_p[\Delta]^+$-algebra.  Note that $x$ and $\tilde{x}$ can be taken to be $T_\ell - \ell - 1$ for any good prime $\ell$ in the sense of Mazur \cite[\S II.16]{Mazur}, though this explicit description never plays a role in our work. Combining this with an explicit monogenic presentation for $\Z_p[\Delta]^+$ and the above \nameref{thm*:freeness} gives a presentation of the form
\begin{equation}\label{eq:presentation}
\TT \cong \Z_p[X, Y]/(G(Y), F(X,Y)).   
\end{equation}
The key features of this presentation are 
\begin{itemize}
\item $X$ maps to $\tilde{x}$;
\item $Y$ maps to $[\delta] + [\delta^{-1}] - 2 \in \Z_p[\Delta]^+$ for a generator $\delta$ of $\Delta$;
\item $G$ is a distinguished polynomial of degree $\frac{1}{2}(p^s + 1)$;
\item $F(x,0)$ is the minimal polynomial for $x \in \TT(N)$.
\end{itemize}
In particular, $r$ is the $X$-degree of $F$.  This leads to \cref{thmA:ringprop}, which gives a new criterion for $r = 2$ in terms of the monogenicity of $\TT$.  In its proof we encounter some ideas used in the proofs of \cref{thmA:merel,thmA:lecouturier,thmA:higherrank}, particularly the use of valuations of the images of $X$ and $Y$ under certain ring maps.

\begin{theoremA}\label{thmA:ringprop}
The ring $\TT$ is a local complete intersection that can be generated by two elements as a $\Z_p$-algebra.  It is monogenic if and only if $r = 2$.
\end{theoremA}

\subsubsection{A generator of the congruence ideal}\label{subsub:gencong}
Having established that $\TT$ carries information about $r$ through its rank, we now show that it also carries information about $\ord(\bar{\zeta})$.  This is done through the congruence ideal. 

Recall that both the space of Eisenstein series and that of cusp forms are stable under the Hecke operators.  Restriction determines a quotient $\TT^0$ of $\TT$ that acts faithfully on the space of cusp forms.  The analogous quotient of $\TT$ acting on Eisenstein series can easily be seen to be the group algebra $\Z_p[\Delta]$.  The congruence module is the largest common quotient of $\TT^0$ and $\Z_p[\Delta]$, and the ideal of $\Z_p[\Delta]$ giving the congruence module is called the congruence ideal.  We show that the congruence ideal is principal with an explicit generator $\Theta \in \Z_p[\Delta]$.  This yields a fiber product description of $\TT$:  
\begin{align}\label{eq:fiberproductdesc}
\TT \cong \TT^0 \times_{\Z_p[\Delta]/(\Theta)}\Z_p[\Delta].
\end{align}

Having seen $\Theta$ in the fiber product description of $\TT$, we now relate $\Theta$ to $\zeta$.  We construct $\Theta$ by interpolating residues of Eisenstein series, which involve the Dirichlet $L$-value at $-1$ of characters of $\Delta$. Since $\zeta$ interpolates the same $L$-values, the two quantities are closely related: for every character $\chi$ of $\Delta$, the ratio  $\chi(\Theta)\chi(\zeta)^{-1}$ is a $p$-adic unit.  Fix a nontrivial character $\chi$ of $\Delta$ for the remainder of the introduction.  Let $\mathrm{val}$ be the valuation on $\Z_p[\chi]$ that is $1$ on a uniformizer.  That the ratio $\chi(\Theta)\chi(\zeta)^{-1}$ is a unit allows us to deduce 
\begin{align}\label{eq:ordzeta-thetaan}
    \text{if }\ord(\bar{\zeta}) \text{ or } \mathrm{val}\left({\chi(\Theta)}\right)\leq 3, \text{ then } \ord(\bar{\zeta}) = \mathrm{val}\left({\chi(\Theta)}\right).
\end{align} 
Thus when $\ord(\bar{\zeta})$ is small, it suffices to relate $\mathrm{val}(\chi(\Theta))$ and $r$.  

\subsubsection{The cuspidal quotient}\label{subsub:cuspquot}
An important theme in our work is that one must pass to the cuspidal quotient $\TT^0$ to relate $r-1$ and $\ord(\bar{\zeta})$.  As such, the kernel of the map $\TT \to \TT^0$ plays a key role in the proofs of \cref{thmA:merel,thmA:lecouturier,thmA:higherrank}.  The fiber product description of $\TT$ in \eqref{eq:fiberproductdesc} shows that this kernel is generated by $(0, \Theta)$.  Using the presentation of $\TT$ in \eqref{eq:presentation} we can lift $(0, \Theta)$ to the unique polynomial of the form
\begin{align}
H(X,Y) = \sum_{i = 0}^{r-1} \sum_{j = 0}^{\frac{|\Delta|-1}{2}} c_{i,j}X^iY^j \in \Z_p[X,Y].
\end{align}
We extend $\chi$ to $\TT$ by precomposing with the projection map $\TT \to \Z_p[\Delta]$ onto the second factor of \eqref{eq:presentation}.  Then $\chi(H(X,Y)) = \chi(\Theta)$.  We show that $\mathrm{val}(\chi(X)) = 1$ and $\mathrm{val}(\chi(Y)) = 2$.  Hence $\mathrm{val}(\chi(\Theta))$, and thus $\ord(\bar{\zeta})$ by \eqref{eq:ordzeta-thetaan}, is determined by the valuations of certain (linear combinations of the) $c_{i,j}$'s.  

The key observation needed to access $r$ from $H$ and its coefficients is that $H(X,0)$ is a distinguished polynomial of degree $r-1$.  That is,
\begin{align}\label{eq:ci0characterizingr}
    c_{i,0} \in \Z_p^\times \iff i=r-1.
\end{align}
The proof of this amounts to showing that, analogous to \eqref{eq:fullmoduloaug}, one recovers the cuspidal quotient $\TT^0(N)$ of Mazur's Hecke algebra by taking the quotient of $\TT^0$ by the augmentation ideal $\II^+$ of $\Z_p[\Delta]^+$.  That is,
\begin{align}\label{eq:cuspidalmoduloaug}
\TT^0 / \II^+\TT^0  \cong \TT^0(N).
\end{align}
The key step in the proof of \eqref{eq:cuspidalmoduloaug} is that the evaluation of $\Theta$ at the trivial character equals $(N^2-1)/12$, which generates the Eisenstein-cuspidal congruence ideal in level $\Gamma_0(N)$; see \cref{prop:Hequalsunittimesf,cor:Tcuspaugideal}. 

\subsubsection{The combinatorial game}\label{subsub:padiccoeff} 
The proofs of \cref{thmA:merel,thmA:lecouturier,thmA:higherrank} now come down to the combinatorial game of finding nonnegative integer solutions to
\begin{align}\label{eq:thegame}
\mathrm{val}(c_{i,j}) + i + 2j = \mathrm{val}(\chi(\Theta)).
\end{align}
Since $\mathrm{val}(\chi(\Theta))$ is small in the cases we consider and $\mathrm{val}(p) \geq p-1$, the equality in \eqref{eq:thegame} forces $c_{i,j} \in \Z_p^\times$ so that $\mathrm{val}(c_{i,j}) = 0$.  

When either $r-1 = 1$ or $\ord(\bar{\zeta}) = \mathrm{val}(\chi(\Theta)) = 1$, there is a unique nonnegative integer solution to \eqref{eq:thegame}, namely $(i,j) = (1,0)$.  Using this, \cref{thmA:merel} follows rather easily from a calculation of $\mathrm{val}(\chi(\Theta))$ using \eqref{eq:ci0characterizingr}.

In the situation of \cref{thmA:lecouturier}, one of $r-1$ and $\ord(\bar{\zeta})$ is equal to $2$ and thus \eqref{eq:thegame} has two nonnegative integer solutions, namely $(i,j) \in \{(2,0), (0,1)\}$.  The problematic situation that would prevent us from relating $\mathrm{val}(\chi(\Theta))$ to $r-1$ is if $c_{0,1} \in \Z_p^\times$.  To rule this out, we use in an essential way the existence of a cuspidal newform $f$ of level $\Gamma_0(N^2)$ that is Eisenstein modulo $p$.  As seen in the \nameref{cor:equidistribution}, this cusp form exists since $r \geq 3$.  Let $\lambda_f \colon \TT \to \overline{\Q}_p$ be the corresponding algebra homomorphism.  Since $f$ is a cusp form, $\lambda_f$ factors through $\TT^0$ and thus $\lambda_f(H) = 0$.  Using an explicit description of the $\Z_p[\Delta]^+$-action on $\TT$, one sees that $\lambda_f(Y) = \xi + \xi^{-1} - 2$ for a nontrivial $p$-power root of unity $\xi$.  One is thus lead to consider 
\[
H(X, \xi + \xi^{-1} - 2) \in \Z_p[\xi + \xi^{-1}][X].
\]
We show that this polynomial is Eisenstein, thus irreducible, if $c_{0,1} \in \Z_p^\times$, but it must be reducible when $r \geq 3$.  Hence we have $c_{0,1} \not\in \Z_p^\times$, which is the essential new ingredient in the proof of \cref{thmA:lecouturier} compared to \cref{thmA:merel}.

Finally, suppose that one of $r-1$ and $\ord(\bar{\zeta})$ is equal to $3$.  Then the two nonnegative integer solutions to \eqref{eq:thegame} are $(i,j) \in \{(3,0), (1,1)\}$.  The problematic situation is if $c_{1,1} \in \Z_p^\times$.  That is, we have the following proposition.
\begin{propositionA}\label{prop:c11sufficientcondition}
Suppose that $\ord(\bar{\zeta}) = 3$ or $r-1 = 3$.  If $c_{1,1} \in p\Z_p$, then $\ord(\bar{\zeta}) = r-1$.
\end{propositionA}
Hence the hypothesis in \cref{thmA:higherrank} gives $c_{1,1} \in \Z_p^\times$.  The proof of \cref{thmA:higherrank} considers the factorization of the polynomial $H(X, \xi + \xi^{-1} - 2)$ when $c_{1,1} \in \Z_p^\times$.  This factorization, which also uses the \nameref{cor:equidistribution}, and Eisenstein polynomials give rise to the description of the Galois orbits of cuspidal newforms of level $\Gamma_0(N^2)$ and their Hecke fields in \cref{thmA:higherrank}.

\subsection{Philosophy on the higher rank case}\label{sub:higherrank}
\subsubsection{Spoiler coefficients in higher rank}
When one of $\ord(\bar{\zeta})$ and $r-1$ is at least $3$, any hope of relating them must account for the fact that they are not always equal.  There is not even a consistent inequality between them, as seen in \cite[Table 1]{WWE}.  Our methods provide a heuristic explanation for the failure of $\ord(\bar{\zeta}) = r-1$ through the coefficients $c_{i,j}$ with $(i,j)$ being a solution to \eqref{eq:thegame} and $j \neq 0$.  There is always at least one such coefficient whenever $\mathrm{val}(\chi(\Theta)) \geq 3$.  Roughly speaking, our method succeeds in showing $\ord(\bar{\zeta}) = r-1$ whenever all such $c_{i,j}$ are nonunits.  When such a $c_{i,j}$ is a unit, we refer to it as a \textit{spoiler coefficient}.  The proof of \cref{thmA:higherrank} shows that spoiler coefficients can be leveraged to provide detailed information about Eisenstein congruences at level $\Gamma_0(N^2)$.  We emphasize that the existence of spoiler coefficients does not necessarily imply that $\ord(\bar{\zeta}) \neq r-1$; it just impedes the ability of our method to establish equality.

Unfortunately for a given $N$ and $p$, it seems difficult to compute the coefficients of $H$ and determine whether they are $p$-adic units.  Even in the simplest case when exactly one of $\ord(\bar{\zeta})$ and $r-1$ is equal to $3$, we do not have a way to easily check whether $c_{1,1}$ is a unit.  Indeed, the conclusion of \cref{thmA:higherrank} suggests that one might need access to the $p$-adic Hecke fields and Galois orbits of cuspidal newforms of level $\Gamma_0(N^2)$ that are Eisenstein modulo $p$ --- a significant computational challenge in practice.

\subsubsection{An example to support \cref{thmA:higherrank}}
One may wonder whether the behavior of the Galois orbits observed in \cref{thmA:higherrank} is truly due to the hypothesis that $\ord(\bar{\zeta}) \neq r-1$.  Perhaps this behavior is always present and simply becomes visible under this hypothesis.  We believe this is not the case.  The following example is one where the hypothesis of \cref{thmA:higherrank} fails in that $\ord(\bar{\zeta}) = 3 = r-1$, and the shape of the Galois orbits does not match the description given in the conclusion of \cref{thmA:higherrank}.

Let $N = 181$ and $p = 5$, in which case $\ord(\bar{\zeta}) = 3 = r-1$ \cite[Table 1]{WWE}.  Consider the modular form $f$ with LMFDB label 32761.2.a.c \cite{lmfdb:32761.2.a.c}, whose Hecke field is $\Q(\sqrt{5})$. We sketch in the next paragraph how one can verify that $f$ is Eisenstein modulo $p$.  According to the \nameref{cor:equidistribution}, there are $\frac{(r-2)(p^s - 1)}{2} = \frac{(4-2)(5 -1)}{2} = 4$ cuspidal newforms of level $\Gamma_0(181^2)$ that are Eisenstein modulo $5$.  Since the $5$-adic Hecke field of $f$ is $\Q_5(\sqrt{5})$, it only accounts for two of these four forms.  The remaining two must lie in a different Galois orbit.  On the other hand, the conclusion of \cref{thmA:higherrank} would predict a unique Galois orbit.

Finally, we sketch a verification that $f$ is Eisenstein modulo $5$.  By considering the LMFDB data about the trace to $\Q$ of the Fourier coefficients of $f$, it is not hard to show that $a_\ell(f) \equiv \ell + 1 \bmod \sqrt{5}$ for all primes $\ell \neq N$ for which the trace of $a_\ell(f)$ exists in LMFDB.  We show that in fact checking this congruence for $\ell \in \{2, 3\}$ suffices to prove that $f$ to be Eisenstein modulo $5$.  Using Magma \cite{MR1484478}, we compute the intersection $E$ of the generalized eigenspaces for the action modulo $5$ of the Hecke operators $T_2-3$ and $T_3-4$ on the space of cuspforms of level $\Gamma_0(181^2)$. We find that $E$, which contains all of the eigenforms that are Eisenstein modulo $5$, has dimension $10$. The contribution to this dimension coming from oldforms of level $\Gamma_0(181)$ that are Eisenstein modulo $5$ equals $2\times3 = 6$.  We saw in the previous paragraph that the \nameref{cor:equidistribution} ensures that there are $4$ cuspidal new forms of level $\Gamma_0(181^2)$ that are Eisenstein modulo $p$, which contribute to $E$.  Thus forms that are Eisenstein modulo $p$ exhaust $E$.  It follows that  $f$ is Eisenstein modulo $5$ if and only if $a_\ell(f) \equiv \ell + 1 \bmod \sqrt{5}$ for $\ell \in \{2,3\}$, as desired.
 
\subsection{A brief comparison with work of Lecouturier} \label{subsec:comparison} 
We close the introduction with a few remarks on the relation between Lecouturier's work and ours. Both provide frameworks that give uniform proofs of \cref{thmA:merel,thmA:lecouturier} that lie on the ``Hecke side", in contrast to the ``Galois side" techniques of Calegari--Emerton \cite{CE2005} and Wake--Wang-Erickson \cite{WWE}.  In particular, we emphasize that the presentation of $\TT$, which plays a crucial role in our proofs, does not use any global deformation theoretic techniques like the numerical criterion, Galois cohomology, or the Taylor--Wiles method.  While we use deformation theory to make $\TT$ a $\Z_p[\Delta]^+$-algebra, this is mild in the sense that deformations are of a procyclic group.  In fact, our use of inertial deformation theory is entirely for expositional clarity; one can give a purely automorphic description of the $\Z_p[\Delta]^+$-action as in \cite{LangPollackWake}.   

Our proof is substantially shorter than Lecouturier's and gives access to information like \cref{thmA:higherrank}, which seems inaccessible via his techniques.  We hope that our techniques are amenable to generalization since the modular representation theory approach to $\Z_p[\Delta]^+$-freeness in \cite{LangPollackWake} is quite flexible.

Lecouturier's strategy involves explicit constructions of higher Eisenstein elements in suitable Hecke modules and intricate interplay between them. He recasts the question of Mazur in terms of the $\F_p$-valued intersection pairing $\mathbin{\scalebox{2.7}{$\boldsymbol{.}$}}$ in homology between the spaces of even and odd modular symbols \cite[Theorem 1.11]{Lecouturier}.  A central role is played by the Shimura class $m_1^{-}$, which is a canonical class in the space of odd modular symbols that comes from the unramified cover $X_1(N) \rightarrow X_0(N)$ of modular curves.  With this perspective, a key step in his proof of \cref{thmA:merel} is to find a formula for the zeroth higher Eisenstein element $\tilde{m}_0^+$ in the space of even modular symbols \cite[Theorem 1.9]{Lecouturier}.  An analogous key step in his proof of \cref{thmA:lecouturier} is to find a formula for the first higher Eisenstein element $\tilde{m}_1^+$ \cite[Theorem 1.10]{Lecouturier}. The constructions of both $\tilde{m}_0^+$  and $\tilde{m}_1^+$ crucially rely on working with the homology of $X_1(N)$. 

Formal properties of pairings of Hecke modules show that $r=2$ if and only if $\tilde{m}_0^+ \mathbin{\scalebox{2.7}{$\boldsymbol{.}$}} m_1^{-} \neq 0$. Using the explicit construction of $\tilde{m}_0^+$, one computes $\tilde{m}_0^+ \mathbin{\scalebox{2.7}{$\boldsymbol{.}$}} m_1^{-}$ and shows that it is nonzero exactly when $\ord(\bar{\zeta})=1$. Similarly, $r=3$ if and only if $\tilde{m}_0^+ \mathbin{\scalebox{2.7}{$\boldsymbol{.}$}} m_1^{-} = 0$ and  $\tilde{m}_1^+ \mathbin{\scalebox{2.7}{$\boldsymbol{.}$}} m_1^{-} \neq 0$.  A similar computation using the constructions of $\tilde{m}_0^+$ and $\tilde{m}_1^+$ implies that this happens exactly when $\ord(\bar{\zeta})=2$. As Lecouturier does not have an analogous formula for the second \textit{higher Eisenstein element} $\tilde{m}_2^+$, he does not directly address the case when $r-1$ or $\ord(\zeta)$ is at least $3$. However, Lecouturier and Wang give an expression for the second higher Eisenstein element $m_2^{-}$, which lives in the space of odd modular symbols, in terms of Steinberg symbols from algebraic $K$-theory \cite[Theorem 1.12, Remark 1.13]{MR4366021}.  

Despite the stark differences in techniques, one could look for commonalities between our work and Lecouturier's. A key motivation in constructing the higher Eisenstein elements $\tilde{m}_0^+$  and $\tilde{m}_1^+$ arises from considering suitable combinations of $\Z_p[\Delta]$-valued Eisenstein series of level $\Gamma_1(N)$ \cite[Section 1.1]{Lecouturier}. In contrast, in level $\Gamma_0(N^2)$ we only need to consider a single $\Z_p[\Delta]$-valued Eisenstein series. With this perspective, one may ask  whether the data of Eisenstein-cuspidal congruences at levels $\Gamma_1(N)$ and $\Gamma_0(N^2)$ are equivalent, and whether one can recover \cref{thmA:merel,thmA:lecouturier} by studying congruences at level $\Gamma_1(N)$ instead. We shed light on the former point in a forthcoming companion paper, where we prove the equivalence of an Eisenstein modularity theorem for $\Gamma_0(N^2)$ and $\Gamma_1(N)$.  We prove the former directly; the latter is an unpublished result of Lecouturier--Wake--Wang-Erickson.

\subsection{Leitfaden}\label{subsec:leitfaden}
The order of the paper approximately mirrors that of the introduction, which also follows the dependence of the content.  \cref{sec:HeckeAlgsDefRings} contains definitions and preliminary results on Hecke algebras and deformation rings.  This is followed in \cref{sec:modrepthy} by a summary of the work of the first author with Pollack and Wake and its consequences.  In particular, the \nameref{thm*:freeness} is \cref{thm:TisFree} and the \nameref{cor:equidistribution} is \cref{prop:equidistributionofcharacters}.  The final preparatory section is \cref{sec:Heckeside}, where we define $\Theta$ and show that it generates the Eisenstein-cuspidal congruence ideal in \cref{cor:Theta_an}.  While the material in \cref{sec:Heckeside} is needed in \cref{sec:ringpropertiesofTT,sec:rankcriteria}, those later sections can be read first if the reader is willing to consult occasionally \cref{sec:Heckeside} for the necessary definitions and facts.  \cref{thmA:ringprop} is proved in \cref{sec:ringpropertiesofTT}, specifically \cref{cor:TisCI} and \cref{thm:monogeniciffr=2}.  The main theorems, \cref{thmA:merel,thmA:lecouturier,thmA:higherrank} are proved in \cref{sec:rankcriteria}; \cref{thmA:merel} is \cref{thm:mereltake2}, \cref{thmA:lecouturier} is \cref{thm:lecouturiertake2}, and \cref{thmA:higherrank} is \cref{cor:Galoisorbits}.

\subsection{Notation}\label{subsec:notation}
Henceforth, fix primes $N, p \geq 5$ such that $N \equiv 1 \bmod p$.  Write $p^s$ for the exact power of $p$ dividing $N-1$.  For an integer $n \geq 1$, write $\zeta_n$ for a primitive $n$-th root of unity.  Let $K/\Q$ be the maximal subfield of $\Q(\zeta_N)$ with $p$-power degree, which is nontrivial since $N \equiv 1 \bmod p$.  Let $\Delta = \Gal{K}{\Q}$, which is canonically  the maximal $p$-power quotient of $(\Z/N\Z)^\times$.  For $\delta \in \Delta$, we use $[\delta]$ to denote the corresponding group-like element in the group ring $\Z_p[\Delta]$.  Write $\hat{\Delta}$ for the group of characters of $\Delta$.

For a positive integer $M$, let $\Gamma_0(M)$ denote the subgroup of $\SL_2(\Z)$ whose lower left entry is divisible by $M$.  For a $\Z_p$-algebra $A$, write $M_2(\Gamma_0(M); A)$ for the space of weight $2$ modular forms of level $\Gamma_0(M)$ with coefficients in $A$ and $S_2(\Gamma_0(M); A)$ for the subspace of cuspforms.  When $A = \Z_p$, we simply write $M_2(\Gamma_0(M))$ and $S_2(\Gamma_0(M))$.  Recall that $M_2(\Gamma_0(M); A) = M_2(\Gamma_0(M)) \otimes A$ and similarly for $S_2(\Gamma_0(M); A)$.   

For each prime $\ell$ we have a Hecke operator $T_\ell$ acting on $M_2(\Gamma_0(M); A)$.  A modular form $f = \sum_{n=0}^\infty a_n(f)q^n$ is \textit{normalized} if $a_1(f) = 1$.  A normalized form $f \in S_2(\Gamma_0(M))$ is \textit{primitive} if it is an eigenform for all Hecke operators away from $M$ and $f$ is a newform of level $M$ \cite[p. 164]{Miyakebook}.  

Let $\overline{\Z}_p$ denote the integral closure of $\Z_p$ in $\overline{\Q}_p$ and $\pp$ its unique maximal ideal.  We say that $f, g \in M_2(\Gamma_0(M); \overline{\Z}_p)$ are \textit{congruent modulo $\pp$} if $a_n(f) \equiv a_n(g) \bmod \pp$ for all $n \neq 0$.  We say that $f$ and $g$ are \textit{congruent modulo $\pp$ away from $N$} if $a_n(f) \equiv a_n(g) \bmod \pp$ for all $n \geq 0$ such that $N \nmid n$.  

For a field $L$ of characteristic zero, fix an algebraic closure $\overline{L}$ of $L$, and let $G_L = \Gal{\overline{L}}{L}$.  For each prime $\ell$ we fix embeddings $\overline{\Q} \hookrightarrow \overline{\Q}_\ell$, which induces a map $G_{\Q_\ell} \to G_\Q$.  Write $I_\ell$ for the inertia subgroup of $G_{\Q_\ell}$, which we view inside $G_\Q$ via the given map.  Set $S \coloneqq \{N, p, \infty\}$, and write $G_{\Q,S}$ for the quotient of $G_\Q$ corresponding to the maximal extension of $\Q$ unramified outside $S$.  For each prime $\ell \not\in S$, fix a choice of a Frobenius element at $\ell$, denoted $\Frob_\ell \in G_{\Q,S}$.  Let $\varepsilon \colon G_{\Q, S} \to \Z_p^\times$ be the $p$-adic cyclotomic character and $\omega \colon G_{\Q, S} \to \F_p^\times$ its reduction modulo $p$.  Given a finite order character $\chi$, write $\Z_p[\chi]$ for the $\Z_p$-algebra generated by the values of $\chi$.

For a normalized form $f \in S_2(\Gamma_0(M); \overline{\Q}_p)$ that is a Hecke eigenform for the operators away from $N$, we write $\rho_f \colon G_{\Q,S} \to \GL_2(\overline{\Z}_p)$ for the $\pp$-adic Galois representation associated to $f$ satisfying $\tr \rho_f(\Frob_\ell) = a_\ell(f)$ and $\det \rho_f(\Frob_\ell) = \varepsilon(\Frob_\ell) = \ell$ for all primes $\ell \nmid Np$.

If $M$ is a $\Z_p$-module (or in particular an $\Z_p$-algebra) and $\OK$ is a finite extension of $\Z_p$, we sometimes write $M_\OK$ for the base change $M \otimes_{\Z_p} \OK$.

\section{Hecke algebras and deformation rings}\label{sec:HeckeAlgsDefRings}
\cref{subsec:genHeckealgs} gives some generalities about Hecke algebras; nothing here is specific to the Eisenstein situation.  In \cref{subsec:EisHeckealgs} we recall the Eisenstein series we need and use this to define the Hecke algebra $\TT$, referenced in the introduction, and Mazur's Eisenstein Hecke algebra $\TT(N)$.  

\subsection{Generalities about Hecke algebras}\label{subsec:genHeckealgs}
In this section we collect the definitions and facts about Hecke algebras that we need that are not specific to the Eisenstein situation.  Everything is in weight $2$ and either prime level or prime-square level.  The fundamental technicality is that while the anemic Hecke algebra satisfies duality in prime level (since there are no forms of weight 2 and level 1), one must add the $T_N$-operator in prime-square level to obtain a duality theorem.  However, we require a natural map from the full prime-square level Hecke algebra to the anemic Hecke algebra in prime level.  This is achieved in \cref{cor:TTtoTT(N)}.
\subsubsection{Hecke algebras and duality}
We begin by introducing notation for the full and anemic Hecke algebras in levels $\Gamma_0(N)$ and $\Gamma_0(N^2)$.  
\begin{definition}
Let $j \in \{1,2\}$.
\begin{itemize}
    \item Let $H(N^j)$ be the $\Z_p$-subalgebra of $\End_{\Z_p}(M_2(\Gamma_0(N^j)))$ generated by $\{T_\ell \colon \ell \text{ prime}\}$. 
    \item Let $H'(N^j)$ be the \textit{anemic} subalgebra of $H(N^j)$; that is, $H'(N^j)$ is the $\Z_p$-subalgebra of $\End_{\Z_p}(M_2(\Gamma_0(N^j)))$ generated by $\{T_\ell \colon \ell \nmid N \text{ prime}\}$.  
\end{itemize}
\end{definition}
The containment $\Gamma_0(N) \subseteq \Gamma_0(N^2)$ induces surjections $H(N^2) \twoheadrightarrow H(N)$ and $H'(N^2) \twoheadrightarrow H'(N)$.

We recall the standard duality between modular forms and Hecke algebras, noting that it is well behaved in our setting.  For $j \in \{1,2\}$ define
\[
m_2(\Gamma_0(N^j)) \coloneqq \{f \in M_2(\Gamma_0(N^j); \Q_p) \colon a_n(f) \in \Z_p, \forall n \geq 1\}.
\]
There is a natural pairing
\begin{align*}
m_2(\Gamma_0(N^j)) \times H(N^j) &\to \Z_p\\
(f, t) &\mapsto a_1(f|t)
\end{align*}
that is perfect \cite[Theorem 5.3.1]{HidaBlueBook}.  The next lemma shows that for our purposes, $m_2(\Gamma_0(N^j))$ is just the usual space $M_2(\Gamma_0(N^j))$.

\begin{lemma}\label{lem:nodenominators}
If $p > 3$ then $m_2(\Gamma_0(N^j)) = M_2(\Gamma_0(N^j))$.
\end{lemma}

\begin{proof}
Certainly $m_2(\Gamma_0(N^j))$ contains $M_2(\Gamma_0(N^j))$ with finite index.  If the containment is proper then there is an $f \in m_2(\Gamma_0(N^j)) \setminus M_2(\Gamma_0(N^j))$ such that $p f = g \in M_2(\Gamma_0(N^j))$.  Then $g \bmod p$ is a mod-$p$ modular form whose $q$-expansion is constant, hence weight 0.  But $g \bmod p$ also has weight 2, and the space of mod-$p$ modular forms has a weight grading by $\Z/(p-1)\Z$ \cite[Theorem 2(iv)]{Swinnerton-Dyer}.  As $p>3$ we must have $g \equiv 0 \bmod p$.  Then $f = g/p \in M_2(\Gamma_0(N^j))$, a contradiction.  Thus the two spaces are equal.
\end{proof}

For convenience we state the duality theorem.

\begin{corollary}\label{cor:duality}
The pairing $M_2(\Gamma_0(N^j)) \times H(N^j) \to \Z_p$ given by $(f, t) \mapsto a_1(f | t)$ is perfect.
\end{corollary}

\begin{proof}
This is immediate from \cref{lem:nodenominators} and \cite[Theorem 5.3.1]{HidaBlueBook}.
\end{proof}

\subsubsection{Localizing at a maximal ideal in prime level}\label{subsubsec:localizeprimelevel}
In this section we show that, after localizing at a maximal ideal, the anemic and full Hecke algebras in prime-level are isomorphic.  In particular, \cref{lem:epsilonforTN} shows that a mod-$p$ anemic eigensystem determines the action of $T_N$ on any (characteristic $0$) modular form that gives rise to that mod-$p$ anemic eigensystem.
\begin{lemma}\label{lem:epsilonforTN}
Let $\n'$ be a maximal ideal in $H'(N)$.  
\begin{enumerate}[label=(\roman*)]
    \item There is an $\epsilon = \epsilon(\n') \in \{\pm 1\}$ such that $T_N$ acts as $\epsilon$ on $M_2(\Gamma_0(N))_{\n'}$.
    \item Let $\n$ be the ideal of $H(N)$ generated by $\n'$ and $T_N - \epsilon$.  Then $M_2(\Gamma_0(N))_{\n'} = M_2(\Gamma_0(N))_{\n}$, and the natural map $H'(N)_{\n'} \to H(N)_{\n}$ is an isomorphism.
\end{enumerate}
\end{lemma}

\begin{proof}
To prove the first statement, it suffices to calculate the action of $T_N$ on $M_2(\Gamma_0(N))_{\n'} \otimes_{\Z_p} \overline{\Q}_p$, which has a basis of $H(N)_{\n'}$-eigenforms.  Let $f \in M_2(\Gamma_0(N))_{\n'} \otimes_{\Z_p} \overline{\Q}_p$ be such an eigenform.  Assume that $f$ is normalized, and set $\epsilon \coloneqq a_N(f)$, which is the $T_N$-eigenvalue of $f$.  If $f$ is the unique normalized Eisenstein series, then a direct calculation gives $a_N(f) = 1$.  Otherwise $f$ is a newform and hence $a_N(f) \in \{\pm 1\}$ \cite[Theorem 4.6.17(2)]{Miyakebook}.  

We must show that $\epsilon$ does not depend on the choice of $f$.  Suppose that $g \in M_2(\Gamma_0(N))_{\n'} \otimes_{\Z_p} \overline{\Q}_p$ is another normalized $H(N)_{\n'}$-eigenform.  Then $f-g \bmod \p$ is a mod-$p$ modular form whose $q$-expansion is supported on powers of $q^N$.  Then there is a mod-$p$ modular form of weight 2 and level 1 whose $q$-expansion gives that of $f-g \bmod \p$ when evaluated at $q^N$ \cite[Lemma 5.9]{Mazur}.  Using Katz's definition of mod-$p$ modular forms, one sees that the space of mod-$p$ modular forms of weight 2 and level 1 is zero \cite{Katz73}.  Thus $f \equiv g \bmod \p$, so $\epsilon = a_N(f) \equiv a_N(g) \bmod \p$.  Since $g$ is a normalized eigenform, we have $a_N(g) \in \{\pm 1\}$ and hence $\epsilon = a_N(g)$.  Therefore $T_N$ acts by $\epsilon$ on all of $M_2(\Gamma_0(N))_{\n'}$.

The second statement follows directly from the first.
\end{proof}

\subsubsection{Localizing at a maximal ideal in prime-square level}\label{subsubsec:localizeprimesquarelevel}
Next we compare the localized anemic and full Hecke algebras in level $\Gamma_0(N^2)$.  Fix a maximal ideal $\n'$ of $H'(N)$.  Let $\epsilon \in \{\pm 1\}$ be the sign associated to $\n'$ by \cref{lem:epsilonforTN} and $\n \coloneqq \langle \n', T_N - \epsilon \rangle \subseteq H(N)$.  Let $\m'$ be the pullback of $\n'$ in $H'(N^2)$ under the natural map $H'(N^2) \twoheadrightarrow H'(N)$ given by restriction.  Viewing $H(N^2)$ and $H'(N^2)$ as $H'(N^2)$-modules, we localize at the multiplicative set $H'(N^2) \setminus \m'$.  This gives an injection $H'(N^2)_{\m'} \hookrightarrow H(N^2)_{\m'}$.  On the other hand, clearly $H(N^2)_{\m'}$ is generated as an $H'(N^2)_{\m'}$-algebra by $T_N$.  

We begin by considering the action of $T_N$ on the space of oldforms in $M_2(\Gamma_0(N^2))_{\m'}$.  By \cref{lem:epsilonforTN} we know that $T_N$ acts by $\epsilon = \epsilon(\n')$ on the copy of $M_2(\Gamma_0(N))_{\n'}$ in $M_2(\Gamma_0(N^2))_{\m'}$.  The following lemma essentially shows that $T_N$ acts by $0$ on a complement to $M_2(\Gamma_0(N))_{\n'}$.

\begin{lemma}\label{lem:stabilizationatN}
Let $f \in M_2(\Gamma_0(N); \overline{\Z}_p)$ be a normalized eigenform.  There is a unique normalized eigenform $f_0 \in M_2(\Gamma_0(N^2); \overline{\Z}_p)$ such that $T_\ell f_0 = a_\ell(f)f_0$ for all primes $\ell \neq N$ and $T_Nf_0 = 0$.  Explicitly, $f_0(z) = f(z) - a_N(f)f(Nz)$.
\end{lemma}

\begin{proof}
We know that $T_N f = a_N(f)f$.  From the definition of $T_N$ it sends $f(Nz)$ to $f$.  Thus it follows, for $f_0(z) \coloneqq f(z) - a_N(f)f(Nz)$, that $T_Nf_0 = 0$ and $T_\ell f_0 = a_\ell(f) f_0$ for all primes $\ell \neq N$.  The uniqueness statement follows from the fact that the dimension of the $H'(N^2)$-eigenspace with eigensystem $\{a_\ell(f) \colon \ell \neq N \text{ prime}\}$ in $M_2(\Gamma_0(N^2); \overline{\Q}_p)$ is 2.
\end{proof}

\cref{lem:stabilizationatN} shows that all of the anemic eigensystems of level $\Gamma_0(N)$ are seen by the $0$-eigenspace of $T_N$.  The following proposition shows that this eigenspace also sees the new eigensystems of level $\Gamma_0(N^2)$.

\begin{proposition}\label{prop:TN=0eigenspace}
Every $H'(N^2)_{\m'}$-eigensystem on $M_2(\Gamma_0(N^2); \overline{\Z}_p)$ admits an eigenvector in the $0$-eigenspace of $T_N$ on $M_2(\Gamma_0(N^2); \overline{\Z}_p)_{\m'}$.
\end{proposition}

\begin{proof}
\cref{lem:stabilizationatN} shows that the proposition holds for eigensystems of level $\Gamma_0(N)$.  To conclude, it suffices to show that if $f \in M_2(\Gamma_0(N^2); \overline{\Z}_p)_{\m'}$ is a normalized $H'(N^2)_{\m'}$-eigenform whose $H'(N^2)_{\m'}$-eigensystem does not appear in $M_2(\Gamma_0(N))_{\m'}$, then $a_N(f) = 0$.  If $f$ is a newform, then this follows from \cite[Theorem 4.6.17(3)]{Miyakebook}.  If $f$ is an Eisenstein series, then this can be seen directly from its construction \cite[Theorem 4.6.2]{DiamondShurman}.
\end{proof}

\begin{corollary}\label{cor:minpolyTN}
The operator $T_N$ acting on $M_2(\Gamma_0(N^2))_{\m'}$ satisfies the polynomial $x(x - \epsilon)$.
\end{corollary}

\begin{proof}
It suffices to compute the action of $T_N$ on a basis of $H(N^2)_{\m'}$-eigenforms of $M_2(\Gamma_0(N^2))_{\m'} \otimes_{\Z_p} \overline{\Q}_p$.  A basis for the space of old forms of this space consists of the union of $\{f, f_0\}$ as $f$ runs over normalized $H(N)_{\n'}$-eigenforms and $f_0$ is as in \cref{lem:stabilizationatN}.  Thus it follows from \cref{lem:epsilonforTN} and \cref{lem:stabilizationatN} that $T_N$ satisfies $x(x-\epsilon)$ on the space of old forms.  On the other hand, \cref{prop:TN=0eigenspace} shows that any anemic eigenform that is not in the old space is in the $0$-eigenspace of $T_N$.  Thus each eigenform has $T_N$-eigenvalue $0$ or $\epsilon$, and the result follows.  
\end{proof} 

From \cref{cor:minpolyTN} we see that there is a surjective morphism of $\Z_p$-algebras
\begin{equation}\label{eq:CRTquotient}
H'(N^2)_{\m'}[x]/(x(x-\epsilon)) \twoheadrightarrow H(N^2)_{\m'} 
\end{equation}
that sends $x$ to $T_N$.  One easily sees that $H'(N^2)_{\m'}[x]/(x(x-\epsilon))$ has two maximal ideals, from which it follows that $H(N^2)_{\m'}$ has two maximal ideals, namely $\m_0 \coloneqq \langle \m', T_N \rangle$ and $\m_\epsilon \coloneqq \langle \m', T_N - \epsilon \rangle$.  This immediately gives the following decomposition.

\begin{corollary}\label{cor:Heckealgproddecomp}
We have $H(N^2)_{\m'} \cong H(N^2)_{\m_0} \times H(N^2)_{\m_\epsilon}$.
\end{corollary}

\begin{proof}
This is immediate since $H(N^2)_{\m'}$ is a semilocal ring with maximal ideals $\m_0$ and $\m_\epsilon$.
\end{proof}
By the Chinese Remainder Theorem,
\begin{align*}
H'(N^2)_{\m'}[x]/(x(x - \epsilon)) &\cong H'(N^2)_{\m'}[x]/(x) \times H'(N^2)_{\m'}[x]/(x-\epsilon)\\ 
&\cong H'(N^2)_{\m'} \times H'(N^2)_{\m'},   
\end{align*}
where projection onto the first (respectively, second) factor is given by the idempotent $1 - \epsilon x$ (respectively, $\epsilon x$).  Applying these idempotents to \eqref{eq:CRTquotient}, we obtain surjections
\begin{equation}\label{eq:surjections}
\psi_0 \colon H'(N^2)_{\m'} \twoheadrightarrow H(N^2)_{\m_0} \text{ and } \psi_{\epsilon} \colon H'(N^2)_{\m'} \twoheadrightarrow H(N^2)_{\m_\epsilon}
\end{equation}
that send $T_\ell$ to $T_\ell$ for all primes $\ell \nmid N$.

\begin{proposition}\label{prop:Heckealgisos}
The map $\psi_0 \colon H'(N^2)_{\m'} \to H(N^2)_{\m_0}$ is an isomorphism.  The map $\psi_\epsilon \colon H'(N^2)_{\m'} \to H(N^2)_{\m_\epsilon}$ factors through the natural quotient $H'(N^2)_{\m'} \twoheadrightarrow H'(N)_{\n'}$ and induces an isomorphism $H'(N)_{\n'} \cong H(N^2)_{\m_\epsilon}$.  In particular, there are $\Z_p$-algebra isomorphisms
\begin{enumerate}[label=(\roman*)]
    \item $H'(N^2)_{\m'} \cong H(N^2)_{\m_0}$,
    \item $H(N^2)_{\m_\epsilon} \cong H(N)_{\n} \cong H'(N)_{\n'}$,
\end{enumerate}
 which send $T_\ell$ to $T_\ell$ for all primes $\ell \nmid N$.
\end{proposition}

\begin{proof}
That $T_\ell$ is sent to $T_\ell$ is immediate from the definitions of the maps.

For $\lambda \in \{0, \epsilon\}$, the surjection $\psi_\lambda$ is given by restricting Hecke operators to the $\lambda$-eigenspace of in $M_2(\Gamma_0(N^2))_{\m'}$.  To show that this is an isomorphism for $\lambda = 0$ it suffices to show that both rings have the same $\Z_p$-rank since neither ring has any $p$-torsion.  The $\Z_p$-rank of $H'(N^2)_{\m'}$ is the number of $H'(N^2)_{\m'}$-eigensystems in $M_2(\Gamma_0(N^2))_{\m'}$, while the $\Z_p$-rank of $H(N^2)_{\m_0}$ is the number of $H'(N^2)_{\m'}$-eigensystems in the $0$-eigenspace of $T_N$.  These numbers are the same by \cref{prop:TN=0eigenspace}.

On the other hand, the $\epsilon$-eigenspace of $T_N$ on $M_2(\Gamma_0(N^2))_{\m'}$ is $M_2(\Gamma_0(N))_{\n}$, and so the map $H'(N^2)_{\m'} \twoheadrightarrow H(N^2)_{\m_\epsilon}$ factors through the natural quotient $H'(N^2)_{\m'} \twoheadrightarrow H'(N)_{\n'}$.  Thus we have surjections
\[
H'(N)_{\n'} \twoheadrightarrow H(N^2)_{\m_\epsilon} \twoheadrightarrow H(N)_{\n},
\]
the composition of which is an isomorphism by \cref{lem:epsilonforTN}.  
\end{proof}

We deduce several corollaries from the isomorphisms in \cref{prop:Heckealgisos} that will be used in the paper.

\begin{corollary}\label{cor:reduced}
The ring $H(N^2)_{\m_0}$ is reduced, and $T_N = 0 \in H(N^2)_{\m_0}$.
\end{corollary}

\begin{proof}
To see that $T_N = 0 \in H(N^2)_{\m_0}$, recall that $x$ maps to $T_N$ in \eqref{eq:CRTquotient}.  On the other hand, under the identification $H'(N^2)_{\m'}[x]/(x(x-\epsilon)) \cong H'(N^2)_{\m'} \times H'(N^2)_{\m'}$ given by the Chinese Remainder Theorem following \cref{cor:minpolyTN}, we see that $x$ maps to $(0,\epsilon)$.  Since the isomorphism $H'(N^2)_{\m'} \cong H(N^2)_{\m_0}$ from \cref{prop:Heckealgisos} is obtained by projection onto the first component of this decomposition, we see that $T_N = 0 \in H(N^2)_{\m_0}$.

Thus $H(N^2)_{\m_0}$ is generated by the $T_\ell$ for $\ell \nmid N$ prime, which are well known to be semisimple operators.  Since $H(N^2)_{\m_0}$ has no $p$-torsion, its natural map to $H(N^2)_{\m_0} \otimes_{\Z_p} \overline{\Q}_p$ is an embedding.  Hence $H(N^2)_{\m_0} \otimes_{\Z_p} \overline{\Q}_p$ is a commutative semisimple ring and thus a product of fields, from which it follows that $H(N^2)_{\m_0}$ is reduced.
\end{proof}

\begin{corollary}\label{cor:m0decomp}
There is a direct sum decomposition as $H'(N^2)_{\m'}$-modules
\[
M_2(\Gamma_0(N^2))_{\m'} = M_2(\Gamma_0(N^2))_{\m_0} \oplus M_2(\Gamma_0(N))_{\m'}.
\]
\end{corollary}

\begin{proof}
The product decomposition of $H(N^2)_{\m'}$ in \cref{cor:Heckealgproddecomp} induces the decomposition of $H'(N^2)_{\m'}$-modules
\[
M_2(\Gamma_0(N^2))_{\m'} = M_2(\Gamma_0(N^2))_{\m_0} \oplus M_2(\Gamma_0(N^2))_{\m_\epsilon}.
\]
To see that $M_2(\Gamma_0(N^2))_{\m_\epsilon} = M_2(\Gamma_0(N))_{\m'}$, note that by \cref{lem:epsilonforTN} we have
\[
M_2(\Gamma_0(N))_{\m'} = M_2(\Gamma_0(N))_{\m_\epsilon} \subseteq M_2(\Gamma_0(N^2))_{\m_\epsilon}.
\]
This inclusion becomes an isomorphism after tensoring with $\overline{\Q}_p$ by \cref{prop:TN=0eigenspace}, so the inclusion has finite index.  Suppose $f \in M_2(\Gamma_0(N^2))_{\m_\epsilon}$ such that $pf \in M_2(\Gamma_0(N))_{\m_\epsilon}$.  Then $pf$ is $\Gamma_0(N)$-invariant, so $f$ is as well.  That is, $f \in M_2(\Gamma_0(N))_{\m_\epsilon}$, and hence the inclusion is an equality, as desired.
\end{proof}

\begin{corollary}\label{cor:TTtoTT(N)}
There is a surjective $\Z_p$-algebra homomorphism $H(N^2)_{\m_0} \twoheadrightarrow H'(N)_{\n'}$ that sends $T_\ell$ to $T_\ell$ for all primes $\ell \nmid N$.
\end{corollary}

\begin{proof}
This follows by combining the isomorphism $H(N^2)_{\m_0} \cong H'(N^2)_{\m'}$ from \cref{prop:Heckealgisos} with the surjection $H'(N^2)_{\m'} \twoheadrightarrow H'(N)_{\n'}$.
\end{proof}

\begin{remark}
We warn the reader that the map in \cref{cor:TTtoTT(N)} is not given by restricting Hecke operators.  Indeed, $M_2(\Gamma_0(N))_{\n'}$, which carries an action of $H'(N)_{\n'}$, is not a subset of $M_2(\Gamma_0(N^2))_{\m_0}$, which carries an action of $H(N^2)_{\m_0}$, since $T_N$ acts by $\epsilon$ on the first space and by $0$ on the second.

Rather, consider the map $\iota \colon M_2(\Gamma_0(N))_{\m'} \hookrightarrow M_2(\Gamma_0(N^2))_{\m'}$ given by $\iota(f)(z) \coloneqq -f(Nz)$.  It is easy to see that $\iota$ is $H'(N^2)_{\m'}$-equivariant.  By \cref{cor:m0decomp} there is a decomposition of $H'(N^2)_{\m'}$-modules
\[
M_2(\Gamma_0(N^2))_{\m'} = M_2(\Gamma_0(N^2))_{\m_0} \oplus M_2(\Gamma_0(N))_{\m'}.
\]
Composing $\iota$ with the projection onto $M_2(\Gamma_0(N^2))_{\m_0}$ gives an injection of $H'(N^2)_{\m'}$-modules
\[
M_2(\Gamma_0(N))_{\m'} \hookrightarrow M_2(\Gamma_0(N^2))_{\m_0},
\]
given explicitly by $f \mapsto f(z) - f(Nz)$.  Applying the duality from \cref{cor:duality} to this map gives the desired morphism $H(N^2)_{\m_0} \twoheadrightarrow H'(N)_{\n'}$.
\end{remark}

\subsubsection{The kernel of the trace map}\label{subsubsec:kertr}
The material in this section is only used in \cref{sec:modrepthy} to apply the results of \cite{LangPollackWake} to our setting.  We continue with a fixed maximal ideal $\m'$ of $H'(N^2)$.  Recall that there is a natural $H'(N^2)$-equivariant trace map
\begin{align*}
   \tr \colon M_2(\Gamma_0(N^2)) &\to M_2(\Gamma_0(N))\\
    f & \mapsto \sum_{\gamma \in \Gamma_0(N^2)\backslash\Gamma_0(N)} f|_2\gamma,
\end{align*}
where $f|_2\bigl(\begin{smallmatrix}
    a & b\\
    c & d
\end{smallmatrix}\bigr)(z) \coloneqq (cz+d)^{-2}f\bigl(\frac{az+b}{cz+d} \bigr)$.  Localizing at $\m'$ gives a morphism of $H'(N^2)_{\m'}$-modules $\tr_{\m'} \colon M_2(\Gamma_0(N^2))_{\m'} \to M_2(\Gamma_0(N))_{\m'}$.  This induces another direct sum decomposition of $M_2(\Gamma_0(N^2))_{\m'}$.

\begin{lemma}\label{lem:kertrdecomp}
There is a direct sum decomposition as $H'(N^2)_{\m'}$-modules 
\[
M_2(\Gamma_0(N^2))_{\m'} = \ker \tr_{\m'} \oplus M_2(\Gamma_0(N))_{\m'}.
\]
\end{lemma}

\begin{proof}
If $f \in M_2(\Gamma_0(N))_{\m'}$, then $\tr_{\m'} f = [\Gamma_0(N) \colon \Gamma_0(N^2)]f = Nf$.  This shows that $\ker \tr_{\m'} \cap M_2(\Gamma_0(N))_{\m'} = 0$ since $N \in \Z_p^\times$.  Any $f \in M_2(\Gamma_0(N^2))_{\m'}$ can be written as $f = (f - N^{-1}\tr_{\mm'} f) + N^{-1}\tr_{\mm'} f$ with $f - N^{-1}\tr_{\m'} f \in \ker \tr_{\m'}$ and $N^{-1}\tr_{\m'} f \in M_2(\Gamma_0(N))_{\m'}$, which proves the lemma.
\end{proof}

\begin{corollary}\label{cor:isomorphismofH'm'modules}
As $H'(N^2)_{\m'}$-modules, $\ker \tr_{\m'}$ and $M_2(\Gamma_0(N^2))_{\m_0}$ are isomorphic.
\end{corollary}

\begin{proof}
By \cref{lem:kertrdecomp} we see that $\ker \tr_{\m'}$ is isomorphic to the quotient of $M_2(\Gamma_0(N^2))_{\m'}$ by the subspace $M_2(\Gamma_0(N))_{\m'}$, and \cref{cor:m0decomp} implies that $M_2(\Gamma_0(N^2))_{\m_0}$ is isomorphic to the same quotient.  Thus we have isomorphisms of $H'(N^2)_{\m'}$-modules
\[
\ker \tr_{\m'} \cong M_2(\Gamma_0(N^2))_{\m'}/M_2(\Gamma_0(N))_{\m'} \cong M_2(\Gamma_0(N^2))_{\m_0}. \qedhere
\]
\end{proof}

\subsection{Eisenstein Hecke algebras}\label{subsec:EisHeckealgs}
We begin with a review of Eisenstein series in \cref{subsubsec:EisSeries} and then define the Hecke algebras that parametrize Eisenstein congruences in \cref{subsubsec:Heckealgs}.  

\subsubsection{Eisenstein series}\label{subsubsec:EisSeries}
Recall the function
\[
E_2(z) \coloneqq \frac{-1}{24} + \sum_{n \geq 1} \sigma(n)q^n,
\]
where $q = e^{2\pi iz}$ and $\sigma(n) \coloneqq \sum_{0 < d \mid n} d$.  It is a nearly holomorphic modular form of weight 2 and level 1.  One obtains a modular form of level $\Gamma_0(N)$ with $N$ prime via $N$-stabilization. Explicitly,
\[
E_{2,N}(z) \coloneqq E_2(z) - NE_2(Nz) \in M_2(\Gamma_0(N)).
\]
It has $T_\ell$-eigenvalue $\ell+1$ for all primes $\ell \neq N$ and $T_N$-eigenvalue 1.  It is well known that $E_{2,N}$ is the unique (up to scaling) Eisenstein series of weight 2 and level $\Gamma_0(N)$.  Mazur studied congruences between $E_{2,N}$ and cusp forms of weight $2$ and level $\Gamma_0(N)$ \cite{Mazur}.  

We recall the Eisenstein series of weight 2 and level $\Gamma_0(N^2)$ that are eigenforms and congruent to $E_{2,N}$ modulo $\pp$ away from $N$.  The description of Eisenstein series can be found in most textbooks on modular forms, for instance \cite[Theorem 4.6.2]{DiamondShurman}.

To define the others of level $\Gamma_0(N^2)$, we establish some notation.  For Dirichlet characters $\eta, \varphi$ and integer $n > 0$, define
\[
\sigma_{\eta,\varphi}(n) \coloneqq \sum_{0 < d \mid n} \eta(n/d)\varphi(d)d.
\]
When $\eta$ and $\varphi$ are both trivial, we recover the function $\sigma$ that describes the Fourier coefficients of $E_2$ above.  Let $\chi$ be a Dirichlet character of conductor $N$.  We have
\[
E_{\chi,\chi^{-1}}(z) \coloneqq  
\sum_{n \geq 1} \sigma_{\chi,\chi^{-1}}(n)q^n \in M_2(\Gamma_0(N^2)).
\]
Each $E_{\chi,\chi^{-1}}$ is a normalized eigenform with $T_\ell$-eigenvalue $\chi(\ell) + \chi^{-1}(\ell)\ell$ for all primes $\ell \nmid N$ and $T_N$-eigenvalue $0$.

\begin{lemma}\label{lem:conglevelN2}
Let $\chi$ be a Dirichlet character of conductor $N$.  The Eisenstein series $E_{\chi,\chi^{-1}}$ is congruent to $E_{2,N}$ modulo $\pp$ away from $N$ if and only if $\chi$ has $p$-power order.  In that case, $\chi$ factors through $\Delta$.
\end{lemma}

\begin{proof}
The final sentence follows directly from the definition of $\Delta$. 

If $\chi$ has $p$-power order, then the definition of $E_{\chi, \chi^{-1}}$ shows that its $\ell$-th Fourier coefficient is congruent to $1+\ell = a_\ell(E_{2,N})$ modulo $\pp$ for all primes $\ell \neq N$.  Since $E_{\chi,\chi^{-1}}$ is an eigenform, this suffices to show that $E_{\chi,\chi^{-1}}$ is congruent to $E_{2,N}$ modulo $\pp$ away from $N$.  Conversely, suppose that $a_\ell(E_{2,N}) \equiv a_\ell(E_{\chi,\chi^{-1}}) \bmod \pp$ for all primes $\ell \neq N$; that is, $1+\ell \equiv \chi(\ell) + \chi^{-1}(\ell)\ell \bmod \pp$.  By the Chebotarev density theorem, it follows that $\tr(1\oplus\omega) = \tr(\bar{\chi} \oplus \bar{\chi}^{-1}\omega)$.  By the Brauer-Nesbitt theorem, we have $\{1, \omega\} = \{\bar{\chi}, \bar{\chi}^{-1}\omega\}$.  Since $\chi$ is unramified at $p$ while $\omega$ is ramified at $p$, we must have $\bar{\chi} = 1$.  Thus $\chi$ has $p$-power order.
\end{proof}

For convenience of notation, we also define $E_{1,1}$ as follows.  Applying \cref{lem:stabilizationatN} to $E_{2,N}$, let
\begin{equation}\label{eq:E11defn}
E_{1,1}(z) \coloneqq E_{2,N}(z) - a_N(E_{2,N})E_{2,N}(Nz) = E_{2,N}(z) - E_{2,N}(Nz), 
\end{equation}
which has $T_\ell$-eigenvalue $\ell + 1$ for all primes $\ell \nmid N$ and $T_N$-eigenvalue $0$.

\subsubsection{Eisenstein Hecke algebras}\label{subsubsec:Heckealgs}
Here we define the Eisenstein Hecke algebras in levels $\Gamma_0(N)$ and $\Gamma_0(N^2)$, using the notation from \cref{subsec:genHeckealgs}.  

We begin with level $\Gamma_0(N)$.  Since $E_{2,N} \in M_2(\Gamma_0(N))$ is an eigenform with $T_\ell$-eigenvalue $\ell+1$ for primes $\ell \neq N$, there is a $\Z_p$-algebra homomorphism $\lambda'_{E_{2,N}} \colon H'(N) \to \Z_p$ that sends $T_\ell$ to $1+\ell$ for all primes $\ell \neq N$.  Composing $\lambda'_{E_{2,N}}$ with the natural reduction map $\Z_p \to \F_p$ gives a morphism $\bar{\lambda}'_{E_{2,N}} \colon H'(N^j) \twoheadrightarrow \F_p$ whose kernel we call $\n'$, which is a maximal ideal of $H'(N^2)$.  Explicitly, 
\[
\n' = \langle p, T_\ell - \ell - 1 \colon \ell \nmid N \text{ prime} \rangle \subseteq H'(N).
\]
Define
\[
\TT(N) \coloneqq H'(N)_{\n'},
\]
which is the $p$-adic Hecke algebra that appears in \cite[\S II.7]{Mazur} and \cite[\S 3.1.3]{WWE}.  We note that in this case, the $T_N$-eigenvalue $\epsilon$ associated to $\n'$ by \cref{lem:epsilonforTN} is $\epsilon = 1$ since $a_N(E_{2,N}) = 1$.

For level $\Gamma_0(N^2)$, let $\m'$ be the pullback of $\n'$ to $H'(N^2)$; that is, 
\[
\m' = \langle p, T_\ell - \ell - 1 \colon \ell \nmid N \text{ prime} \rangle \subseteq H'(N^2).
\]
By \cref{prop:TN=0eigenspace}, the $0$-eigenspace of $T_N$ on $M_2(\Gamma_0(N^2))_{\m'}$ sees all of the anemic eigensystems.  Therefore we let 
\[
\m_0 \coloneqq \langle \m', T_N \rangle = \langle p, T_N, T_\ell - \ell - 1 \colon \ell \nmid N \text{ prime} \rangle \subseteq H(N^2).
\]
Define
\[
\TT \coloneqq H(N^2)_{\m_0},
\]
which parametrizes mod-$p$ congruences with $E_{1,1}$ in $M_2(\Gamma_0(N^2))$.  We often write $\m$ for the maximal ideal of $\TT$.

Note that \cref{cor:TTtoTT(N)} gives a $\Z_p$-algebra epimorphism 
\begin{equation}\label{eq:TTtoTTN}
\TT \twoheadrightarrow \TT(N)  
\end{equation}
that sends $T_\ell$ to $T_\ell$ for all primes $\ell \nmid N$.

Since $\TT$ is a reduced ring by \cref{cor:reduced}, there is an injection $\TT \hookrightarrow \prod_\pp \TT/\pp$, where $\pp$ runs over the minimal prime ideals of $\TT$.  The set of minimal prime ideals is in bijection with the set 
\[
\mathcal{F} \coloneqq \{f \in M_2(\Gamma_0(N^2))_{\m} \otimes \overline{\Q}_p \colon f \text{ normalized } \TT\text{-eigenform}\}.  
\]
In particular, if $f \in \mathcal{F}$ then we have an injection
\begin{align*}
\TT &\hookrightarrow \prod_{f \in \mathcal{F}} \overline{\Z}_p\\
T_\ell &\mapsto (a_\ell(f))_{f \in \mathcal{F}}.
\end{align*}
We often identify $\TT$ with its image under this map.

\section{Freeness theorem and consequences}\label{sec:modrepthy}
We explain the \nameref{thm*:freeness} and some of its consequences.  In \cref{subsec:ZpDeltaPlus} we derive the $\Z_p[\Delta]^+$-algebra structure on $\TT$.  This is followed by a summary of work of the first author with Pollack and Wake in \cref{subsec:LPW}, which is the key ingredient needed to prove the \nameref{thm*:freeness}.  Finally, in \cref{subsec:ZpDeltaPlusFreeness} we prove the \nameref{thm*:freeness} followed by the \nameref{cor:equidistribution} in \cref{subsec:equidistribution}.

\subsection{The $\Z_p[\Delta]^+$-algebra structure on $\TT$}\label{subsec:ZpDeltaPlus}
In this section we construct a natural $\Z_p$-algebra homomorphism $\Z_p[\Delta]^+ \to \TT$.  The key idea is to identify $\Z_p[\Delta]^+$ as an appropriate inertia-at-$N$ pseudodeformation ring.  The map to $\TT$ is obtained by universality by restricting the usual global $\TT$-valued pseudorepresentation to $I_N$.  This is nearly identical to \cite[\S 4.2]{LangWake25}, though in that paper one assumes that $N \equiv -1 \bmod p$.

For a two-dimensional pseudorepresentation $\rho$, we write $\psi(\rho)$ for the associated pseudorepresentation.  Let $\tilde{R}_N$ be the universal pseudodeformation ring of the trivial 2-dimensional pseudorepresentation $\psi(1 \oplus 1) \colon I_N \to \F_p$ parametrizing deformations with trivial determinant.  Any such pseudodeformation factors through the maximal pro-$p$ quotient of $I_N$, which we denote by $I_N^{(p)} \cong \Z_p$ \cite[Lemma 7.64]{Chenevier}.  Write $\tilde{D}_N \colon I_N^{(p)} \to \tilde{R}_N$ for the corresponding universal pseudorepresentation.  Fix a topological generator $\tau \in I_N^{(p)}$.  Recall that
\[
\Frob_N \tau \Frob_N^{-1} = \tau^N,
\]
and hence any pseudorepresentation of $I_N^{(p)}$ that extends to a pseudorepresentation on $G_{\Q_N}$ necessarily has the same trace on $\tau$ and $\tau^N$.  Define 
\[
R_N \coloneqq \tilde{R}_N/(\tr(\tilde{D}_N)(\tau) - \tr(\tilde{D}_N)(\tau^N)).
\]
Let $D_N \colon I_N^{(p)} \to R_N$ be the corresponding pseudorepresentation obtained by following $\tilde{D}_N$ with the  quotient map $\tilde{R}_N \twoheadrightarrow R_N$.  Note that $\Delta$ is naturally a quotient of $I_N^{(p)}$ since it is the Galois group of the maximal $p$-power subextension of the totally ramified extension $\Q_N(\zeta_N)/\Q_N$.  Thus the choice of generator $\tau \in I_N^{(p)}$ projects to a  generator $\delta_\tau$ of $\Delta$.

\begin{proposition}\label{prop:ZpDeltaPlusStructure}
The $\Z_p$-algebra $R_N$ is isomorphic to $\Z_p[\Delta]^+$ via a map that identifies $\tr(D_N)(\tau) \in R_N$ with $[\delta_\tau] + [\delta_\tau^{-1}] \in \Z_p[\Delta]^+$.
\end{proposition}

\begin{proof}
The proof is nearly identical to \cite[Proposition 4.4]{LangWake25}, but there is a minor tweak that is needed since $N \equiv 1 \bmod p$ for us, while $N \equiv -1 \bmod p$ in that reference.  We summarize the strategy so that we can point out the needed change but leave the reader to consult \cite{LangWake25} for the details.

First one shows that $\tilde{R}_N \cong \Z_p\lb x \rb$ via an isomorphism that identifies $\tr(\tilde{D}_N)(\tau)$ and $2 + x$. To compute the quotient by the ideal generated by $\tr(\tilde{D}_N)(\tau) - \tr(\tilde{D}_N)(\tau^N)$, one adjoins the roots of the characteristic polynomial of $\tilde{D}_N(\tau)$ to $\Z_p\lb x \rb$.  That is, we embed $\Z_p\lb x \rb$ into $\frac{\Z_p\lb x \rb[\lambda]}{\lambda^2 - (2 + x)\lambda + 1} \cong \Z_p\lb \lambda - 1 \rb$.  Note that the element $x$ corresponds to $\frac{(\lambda - 1)^2}{\lambda}$ in $\Z_p\lb \lambda - 1 \rb$.  Let $A$ be the quotient of $\Z_p\lb \lambda - 1 \rb$ by the ideal generated by 
\[
\tr(\tilde{D}_N)(\tau) - \tr(\tilde{D}_N)(\tau^N) = \lambda + \lambda^{-1} - \lambda^N - \lambda^{-N} = -\lambda^{-N}(\lambda^{N+1} - 1)(\lambda^{N-1} - 1).
\]

Letting $s$ be the $p$-adic valuation of $N-1$ and using that $p \nmid N+1$, we find that 
\[
\frac{\lambda^{N+1} - 1}{\lambda - 1}, \frac{\lambda^{N-1} - 1}{\lambda^{p^s} - 1} \in \Z_p\lb \lambda - 1 \rb^\times.
\]
(The roles of $N+1$ and $N-1$ are reversed in \cite[Proposition 4.4]{LangWake25}.) Thus 
\[
A = \Z_p\lb \lambda - 1 \rb/(\lambda + \lambda^{-1} - \lambda^N - \lambda^{-N})  = \Z_p\lb \lambda - 1 \rb/(\lambda - 1)(\lambda^{p^s} - 1),
\]
and $R_N$ is isomorphic to the $\Z_p$-subalgebra of $A$ generated by the image of $x$, which is $\frac{(\lambda - 1)^2}{\lambda}$.  The surjection $A \twoheadrightarrow \Z_p[\Delta]$ given by sending $\lambda$ to $\delta_\tau$ sends $\frac{(\lambda - 1)^2}{\lambda}$ to $[\delta_\tau] + [\delta_\tau^{-1}] - 2$.  Thus we see that $R_N$ projects to $\Z_p[\Delta]^+$ under $A \twoheadrightarrow \Z_p[\Delta]$ since $\Z_p[\Delta]^+$ is the $\Z_p$-subalgebra of $\Z_p[\Delta]$ generated by $[\delta_\tau] + [\delta_\tau^{-1}] - 2$.  One concludes that this surjection is in fact an isomorphism $R_N \cong \Z_p[\Delta]^+$ as in \cite[Proposition 4.4]{LangWake25}.

Finally, note that the identification of $\tr(\tilde{D}_N)(\tau)$ with $2 + x$ becomes an identification of $\tr(D_N)(\tau)$ with $2 + \frac{(\lambda - 1)^2}{\lambda} = \lambda + \lambda^{-1}$ in $R_N \subseteq A$.  Finally, the quotient of $A$ given by sending $\lambda$ to $\delta_\tau$ then identifies $\lambda + \lambda^{-1}$ with $[\delta_\tau] + [\delta_\tau^{-1}] \in \Z_p[\Delta]^+$, as desired.
\end{proof}

\begin{lemma}\label{lem:Thaspseudodef}
There is a pseudorepresentation $D_\TT \colon G_{\Q,S} \to \TT$ that is unramified outside $Np$ and satisfies $\tr(D_\TT)(\Frob_\ell) = T_\ell$ for all primes $\ell \nmid Np$.
\end{lemma}

\begin{proof}
This is a standard argument gluing together the pseudorepresentations $D_f = \psi(\rho_f) \colon G_{\Q,S} \to \overline{\Z}_p$ for each eigenform $f \in \mathcal{F}$.  Each $D_f$ is unramified outside $Np$ and satisfies $\tr(D_f)(\Frob_\ell) = a_\ell(f)$ for all primes $\ell \nmid Np$.  Since $(a_\ell(f))_{f \in \mathcal{F}}   = T_\ell \in \TT$ for all primes $\ell \nmid Np$, it follows from the Chebotarev density theorem that the product $\prod_{f \in \mathcal{F}} D_f$ is a $\TT$-valued pseudorepresentation that satsfies the desired properties.
\end{proof}

\begin{corollary}\label{cor:TTaZpDeltaPlusAlgebra}
The ring $\TT$ is naturally a $\Z_p[\Delta]^+$-algebra.
\end{corollary}

\begin{proof}
The pseudorepresentation $D_\TT$ of \cref{lem:Thaspseudodef} is a pseudodeformation of $\psi(\omega \oplus 1)$ by the Chebotarev density theorem since $T_\ell \bmod \mm \equiv \ell + 1$.  Therefore $D_\TT$ restricts to a deformation of the trivial pseudorepresentation on $I_N$.  Since $\det(D_\TT) = \varepsilon$, it follows that $\det(D_\TT)|_{I_N} = 1$.  Moreover, this deformation on $I_N$ extends to one on $G_{\Q_N}$ since $D_\TT$ is a pseudorepresentation of $G_{\Q, S}$.  Therefore $D_\TT|_{I_N}$ gives rise to a homomorphism $R_N \to \TT$ by universality, which gives the desired structure by \cref{prop:ZpDeltaPlusStructure}.
\end{proof}

Since $\TT$ is a $\Z_p[\Delta]^+$-algebra by \cref{cor:TTaZpDeltaPlusAlgebra}, any $\TT$ eigenform is a $\Z_p[\Delta]^+$-eigenform.  For use in \cref{sec:ringpropertiesofTT,sec:rankcriteria}, we now make this action explicit.

\begin{corollary}\label{cor:ZpDeltaPlusoneigenforms}
Let $f \in M_2(\Gamma_0(N^2))_\m$ be a $\TT$-eigenform.  Then $[\delta_\tau] + [\delta_\tau^{-1}] \in \Z_p[\Delta]^+$ acts on $f$ by 
\[([\delta_\tau] + [\delta_\tau^{-1}])f = \tr(\rho_f(\tau))f,
\]
which determines the action of $\Z_p[\Delta]^+$.
\end{corollary}

\begin{proof}
For any $g \in G_{\Q, S}$, the element $\tr(D_\TT)(g) \in \TT$ acts on an eigenform $f$ by multiplication by $\tr \rho_f(g)$.  In particular this holds for $\tau \in I_N \subseteq G_{\Q,S}$ that projects to a generator of $I_N^{(p)}$.  By \cref{prop:ZpDeltaPlusStructure} it follows that this is exactly the action of $[\delta_\tau] + [\delta_\tau^{-1}]$.  Since $[\delta_\tau] + [\delta_\tau^{-1}]$ generates $\Z_p[\Delta]^+$ as a $\Z_p$-algebra, this determines the action of $\Z_p[\Delta]^+$ on $f$.
\end{proof}

The trivial $\Z_p$-valued pseudorepresentation gives rise to an algebra homomorphism $\Z_p[\Delta]^+ = R_N \to \Z_p$ by universality, which is the augmentation map.  Let $\II^+$ be its kernel, which is the augmentation ideal of $\Z_p[\Delta]^+$.  

\begin{lemma}\label{lem:factorthroughaugideal}
The map $\TT \to \TT(N)$ given in \eqref{eq:TTtoTTN} factors through $\TT/\II^+\TT$.
\end{lemma}

\begin{proof}
We can view
\[
\TT(N) \subseteq \prod_f \overline{\Z}_p,
\]
where the product runs over normalized $\TT(N)$-eigenforms $f$ in $M_2(\Gamma_0(N))_{\m'}$.  It suffices to show that the Galois pseudorepresentation associated to any such $f$ is trivial when restricted to $I_N$.  This is true for cusp forms since they are necessarily Steinberg at $N$.  For $f = E_{2,N}$ this is true since its associated pseudorepresentation is $\psi(\varepsilon \oplus 1)$, which is unramified at $N$.
\end{proof}

\subsection{Work of first author with Pollack and Wake}\label{subsec:LPW}
Although we said in the introduction that the work of the first author with Pollack and Wake gives that $\TT$ is a free $\Z_p[\Delta]^+$-module of rank $r \coloneqq \rk_{\Z_p} \TT(N)$, there is a bit of an argument that is required arrive at this statement from their work.  In this section we give a precise statement of their result, and in the next section we apply it to conclude the \nameref{thm*:freeness}.

The work of the first author with Pollack and Wake is in the setting of weight $2$ modular forms with principal level $N$.  More precisely, let $X(N)$ be the modular curve that parametrizes isomorphism classes of generalized elliptic curves together with an invariant differential and a trivialization of the $N$-torsion.  There is a natural map $X(N) \to X(1) = \mathbb{P}^1$ that is given by forgetting the trivialization of the $N$-torsion.  Define
\[
M_2(\Gamma(N)) \coloneqq H^0(X(N), \Omega^1(\log)),
\]
where $\Omega^1(\log)$ is the sheaf of regular differentials on $X(N)$ that may have log-poles at the cusps.  This space has a natural action of $G = \GL_2(\F_N)$ coming from the fact that the natural map $X(N) \to X(1)$ is a $\GL_2(\F_N)/\{\pm 1\}$-covering away from the cusps.  

We can recover the spaces of modular forms we care about --- those of levels $\Gamma_0(N)$ and $\Gamma_0(N^2)$ --- by taking invariants of $M_2(\Gamma(N))$ by certain subgroups of $G$.  Let $B$ be the Borel subgroup of $G$ consisting of upper triangular matrices and $T$ the diagonal torus in $G$.  Then as anemic Hecke modules, one has \cite[Proposition 7]{Gross18}
\[
M_2(\Gamma_0(N)) = M_2(\Gamma(N))^B \text{ and } M_2(\Gamma_0(N^2)) \cong M_2(\Gamma(N))^T.
\]
The isomorphism in the second case is given explicitly by sending $f \in M_2(\Gamma_0(N^2))$ to  $f|_2\bigl(\begin{smallmatrix}
    N^{-1} & 0\\ 
    0 & 1
\end{smallmatrix} \bigr)$.  (Equality makes sense in the first case since both sides are naturally subsets of $M_2(\Gamma(N))$.)
 
Let $\m'$ be the Eisenstein maximal ideal in the anemic Hecke algebra acting on $M_2(\Gamma(N))$.  That is, $\m'$ is generated by $\{\varpi, T_\ell - \ell - 1 \colon \ell \neq N \text{ prime}\}$.  Since the action of $G$ commutes with the anemic Hecke operators, it follows that the localization  
\[
M \coloneqq M_2(\Gamma(N))_{\m'},
\]
is still a $G$-module.  Moreover, as anemic Hecke modules we have $M^B = M_2(\Gamma_0(N))_{\m'}$ and $M^T \cong M_2(\Gamma_0(N^2))_{\m'}$.  There is a natural trace map
\begin{align*}
 \tr_{T/B} \colon M^T &\to M^B\\
 m &\mapsto \sum_{xT \in B/T} xm
\end{align*}
that is equivariant with respect to the anemic Hecke operators.  Thus $\ker \tr_{T/B}$ is an $H'(N^2)_{\m'}$-module.  By \cref{prop:Heckealgisos} we have $H'(N^2)_{\m'} \cong H(N^2)_{\m_0} = \TT$.  Since $\TT$ is a $\Z_p[\Delta]^+$-algebra, we may view $\ker \tr_{T/B}$ as a $\Z_p[\Delta]^+$-module.  The main freeness result that we use from the work of the first author with Pollack and Wake is the following.

\begin{theorem}\label{thm:LPW}\cite[Corollary 5.3, Proposition 7.2]{LangPollackWake}
The kernel of $\tr_{T/B}$ is a free $\Z_p[\Delta]^+$-module of rank equal to the $\Z_p$-rank of $M^B$.
\end{theorem}

\begin{remark}\label{rem:LLinfamilies}
We use two separate parts of the paper of the first author with Pollack and Wake in an essential way.  Namely, their Corollary 5.3 does say that $\ker \tr_{T/B}$ is a free $\Z_p[\Delta]^+$-module of the appropriate rank, but it uses an automorphic $\Z_p[\Delta]^+$-action that is a priori different than the Galois action described above.  The fact that these two actions agree is shown in \cite[Proposition 7.2]{LangPollackWake}. In fact, the agreement of the modular and Galois actions is an instance of local Langlands in families, as conjectured by Emerton and Helm \cite[Conjecture 1.4.1]{EmertonHelm} and reformulated by Helm \cite[Conjecture 7.6]{Helm2016}.  It has been proven for $\GL_n$ by Helm and Moss \cite{Helm2020, HelmMoss}, followed by another proof in greater generality by Li and Shotton \cite{LiShotton}.  This framework suggests that techniques similar to those of \cite{LangPollackWake} and this paper that use modular representation theory to study congruences between modular forms may be fruitful in settings beyond $\GL_2$.  
\end{remark}

\subsection{The freeness theorem}\label{subsec:ZpDeltaPlusFreeness}
In this section we prove that the \nameref{thm*:freeness} that $\TT$ is a free $\Z_p[\Delta]^+$-module of rank equal $r \coloneqq \rk_{\Z_p} \TT(N)$.  

The key input is \cref{thm:LPW}.  Recall that there is a natural $H'(N^2)$-equivariant trace map
\begin{align*}
   \tr \colon M_2(\Gamma_0(N^2)) &\to M_2(\Gamma_0(N))\\
    f & \mapsto \sum_{\gamma \in \Gamma_0(N^2)\backslash\Gamma_0(N)} f|_2\gamma.
\end{align*}
Localizing at $\m'$ gives a morphism of $H'(N^2)_{\m'}$-modules $\tr_{\m'} \colon M_2(\Gamma_0(N^2))_{\m'} \to M_2(\Gamma_0(N))_{\m'}$.  It is easy to check that $\ker \tr_{T/B}$ is the image of $\ker \tr_{\m'}$ under the isomorphism $M_2(\Gamma_0(N^2))_{\m'} \to M^T$ that sends $f$ to $f|_2\bigl(\begin{smallmatrix}
    N^{-1} & 0\\ 0 & 1
\end{smallmatrix} \bigr)$.  Thus $\ker \tr_{\m'}$ is a free $\Z_p[\Delta]^+$-module of rank equal to $\rk_{\Z_p} M^B$ by \cref{thm:LPW}.  We can now deduce the \nameref{thm*:freeness}.

\begin{theorem}\label{thm:TisFree}
The Hecke algebra $\TT$ is a free $\Z_p[\Delta]^+$-module of rank $r$.  
\end{theorem}

\begin{proof}
Recall that $\TT$ is an $\Z_p[\Delta]^+$-algebra by \cref{cor:TTaZpDeltaPlusAlgebra}.  Thus any $\TT$-module is also an $\Z_p[\Delta]^+$-module.  Moreover, by \cref{prop:Heckealgisos} we have that $\TT = H(N^2)_{\m_0} \cong H'(N^2)_{\m'}$, so any $H'(N^2)_{\m'}$-module is a $\Z_p[\Delta]^+$-module, and $H'(N^2)_{\m'}$-module homomorphisms are in particular $\Z_p[\Delta]^+$-module homomorphisms.  Combining \cref{cor:isomorphismofH'm'modules} with the duality in \cref{cor:duality}, we see that $\TT$ is isomorphic to $\Hom_{\Z_p}(\ker \tr_{\m'}, \Z_p)$ as an $\Z_p[\Delta]^+$-module.  Since $\Z_p[\Delta]^+$ is a complete intersection ring and hence self dual, it suffices to show that $\ker \tr_{\m'}$ is free over $\Z_p[\Delta]^+$ of the desired rank.  \cref{thm:LPW} and the commentary preceding this theorem statement show that $\ker \tr_{\m'}$ is a free $\Z_p[\Delta]^+$-module of rank equal to the $\Z_p$-rank of $M^B$.  Since $M^B = M_2(\Gamma_0(N))_{\m'}$, the $\Z_p$-rank of $M^B$ is equal to $\rk_{\Z_p} M_2(\Gamma_0(N))_{\m'}$.  By \cref{cor:duality} and \cref{lem:epsilonforTN}, the duality pairing between $M_2(\Gamma_0(N))_{\m'}$ and $\TT(N)$ is perfect.  Hence $\rk_{\Z_p} M_2(\Gamma_0(N))_{\m'} = \rk_{\Z_p} \TT(N) = r$. 
\end{proof}

We observe that the freeness result allows us to recover $\TT(N)$ from $\TT$ by taking the quotient by the augmentation ideal $\II^+$ of $\Z_p[\Delta]^+$.

\begin{corollary}\label{cor:descendtoTN}
The map \eqref{eq:TTtoTTN} induces an isomorphism $\TT/\II^+\TT \cong \TT(N)$.
\end{corollary}

\begin{proof} 
Using \cref{lem:factorthroughaugideal}, we see that there is a surjection $\TT/\II^+\TT \twoheadrightarrow \TT(N)$.  We know that $\rk_{\Z_p} \TT(N) = r$.  By \cref{thm:TisFree} we have $\TT \cong (\Z_p[\Delta]^+)^r$, and hence $\TT/\II^+\TT \cong \Z_p^r$.  Thus the map is an isomorphism.
\end{proof}

\subsection{Equidistribution result}\label{subsec:equidistribution}
In this section we establish the \nameref{cor:equidistribution}, which is a consequence of \cref{thm:TisFree}.  

Since $\Z_p[\Delta]^+$ is an inertia-at-$N$ pseudodeformation ring, it is useful to know the shape of $\psi(\rho_f|_{I_N})$ as $f$ runs through normalized eigenforms of $M_2(\Gamma_0(N))_\m$.  It is easy to see that
\begin{equation}\label{eq:EisensteinorSteinberg}
\psi(\rho_f|_{I_N}) = \begin{cases}\psi(\chi \oplus \chi^{-1}) & f = E_{\chi, \chi^{-1}}\\
\psi(1 \oplus 1) & f \in S_2(\Gamma_0(N))_{\m'},
\end{cases}
\end{equation}
where the second equality is because such $f$ are Steinberg at $N$.  \cref{lem:EisnotSteinberg} establishes the behavior of $\psi(\rho_f|_{I_N})$ for newforms $f$ in $S_2(\Gamma_0(N^2))_\m$.

\begin{lemma}\label{lem:EisnotSteinberg}
Let $f \in S_2(\Gamma_0(N^2))_\m$ be a primitive form.  Then $f$ is ramified principal series at $N$, so $\psi(\rho_f|_{I_N}) = \psi(\chi \oplus \chi^{-1})$ for some $1 \neq \chi \in \hat{\Delta}$.
\end{lemma}

\begin{proof}
The only other possibility is that $f$ is supercuspidal at $N$.  We show that if $f$ is supercuspidal then $f \not\equiv E_{1,1} \bmod \pp$, so $f \not\in S_2(\Gamma_0(N^2))_{\m}$.  If $f$ is supercuspidal, then local Langlands implies that there is a quadratic extension $F$ and character $\eta$ of $G_F$ such that $\rho_f|_{G_{\Q_N}} \cong \Ind_F^{\Q_N} \eta$, and this is irreducible.  The irreducibility forces $\eta^\sigma \neq \eta$, where $\sigma$ is the nontrivial element of $\Gal{F}{\Q_N}$ and $\eta^\sigma(g) \coloneqq \eta(\sigma g \sigma^{-1})$.  

If $f \equiv E_{1,1} \bmod \pp$, then $\Ind_F^{\Q_N} \eta$ is reducible modulo $\pp$. That is, $\eta \equiv \eta^\sigma \bmod \pp$.  It follows that $\bar{\eta} := \eta \bmod \pp$ admits an extension $\bar{\eta}'$ to $G_{\Q_N}$, and the only other extension is $\alpha_F\bar{\eta}'$, where $\alpha_F$ is the quadratic character associated to $F/\Q_N$.  Then we must have that
\[
\Ind_F^{\Q_N} \bar{\eta} = \bar{\eta}' \oplus  \alpha_F\bar{\eta}'.
\]
In particular, since $\alpha_F$ is nontrivial, it follows that $\tr \rho_f|_{G_{\Q_N}} = \tr \Ind_F^{\Q_N} \eta \equiv \bar{\eta}' + \alpha_F\bar{\eta}' \bmod \pp$ is not identically equal to $2 = \tr(1 \oplus \omega)|_{G_{\Q_N}}$. This contradicts the assumption that $f \equiv E_{1,1} \bmod \pp$.  Hence such an $f$ cannot be supercuspidal at $N$, so must be ramified principal series.  

The formula for $\psi(\rho_f|_{I_N})$ follows immediately from local Langlands.
\end{proof}

In light of the above descriptions of $\psi(\rho_f|_{I_N})$ as $f$ runs through the normalized eigenforms of $M_2(\Gamma_0(N^2))_\m$, let 
\[
\Xi \coloneqq \{1 \neq \chi \in \hat{\Delta}\}/(\chi \sim \chi^{-1}).
\]
We often abuse notation by conflating an equivalence class of $\Xi$ with its elements.  For instance, for $\chi \in \Xi$ we write $\chi \oplus \chi^{-1}$ for the representation whose constituents are the elements of the equivalence class $\chi$.  From context it should be clear whether we are viewing $\chi$ as an element of $\Xi$ or $\hat{\Delta}$.  
\begin{definition}\label{defn:chif}
For a normalized eigenform $f \in M_2(\Gamma_0(N^2))_\m$, let $\chi_f \in \Xi \sqcup \{1\}$ be defined by the property that 
\[
\psi(\rho_f|_{I_N}) = \psi(\chi_f \oplus \chi_f^{-1}),
\]
which is well defined in light of \eqref{eq:EisensteinorSteinberg} and \cref{lem:EisnotSteinberg}.
\end{definition}
 
The element $\chi_f$ determines the $\Z_p[\Delta]^+$-action on an eigenform $f$.  More precisely, fix an element $\tau \in I_N$ that projects to a topological generator of the tame inertia group.  Let $\delta \in \Delta$ be the image of $\tau$ under the surjection $I_N \twoheadrightarrow \Gal{\Q_N(\zeta_N)}{\Q_N} \twoheadrightarrow \Delta$, which is a generator of $\Delta$.  By \cref{cor:ZpDeltaPlusoneigenforms} we see that 
\[
([\delta] + [\delta^{-1}])f = \tr(\rho_f(\tau))f = (\chi_f(\tau) + \chi_f^{-1}(\tau))f = (\chi_f(\delta) + \chi_f(\delta^{-1}))f.
\]

We note that $\Xi \sqcup \{1\}$ is also an indexing set for the $\Z_p[\Delta]^+$-eigenspaces of $\Z_p[\Delta]^+ \otimes_{\Z_p} \overline{\Q}_p$.  That is, 
\[
\Z_p[\Delta]^+ \otimes_{\Z_p} \overline{\Q}_p \cong \bigoplus_{\chi \in \Xi \cup \{1\}} \overline{\Q}_p(\chi),
\]
where the element $[\delta] + [\delta^{-1}]$ of $\Z_p[\Delta]^+$ acts on $\overline{\Q}_p(\chi)$ by $\chi(\delta) + \chi(\delta^{-1})$.  We can now state the equidistribution result.

\begin{proposition}\label{prop:equidistributionofcharacters}
As $f$ runs over the normalized eigenforms in $M_2(\Gamma_0(N^2))_\m$, the set of characters $\{\chi_f\}_f$ is equidistributed in $\Xi \cup \{1\}$.  More precisely, for each $\chi \in \Xi \cup \{1\}$ there are $r$ normalized eigenforms $f \in M_2(\Gamma_0(N^2))_\m$ for which $\chi_f = \chi$, where forms that are Galois conjugate are counted as different forms.  For each $\chi \in \Xi$, there are $r-2$ normalized eigenforms $f \in S_2(\Gamma_0(N^2))_\m^{\mathrm{new}}$ for which $\chi_f = \chi$.
\end{proposition}

\begin{proof}
By \cref{thm:TisFree}, $\TT$ is a free $\Z_p[\Delta]^+$-module of rank $r$.  \cref{cor:duality} and the fact that $\Z_p[\Delta]^+$ is Gorenstein, hence self dual, imply that $M_2(\Gamma_0(N^2))_\m \cong (\Z_p[\Delta]^+)^r$ as $\Z_p[\Delta]^+$-modules.  Tensoring to $\overline{\Q}_p$ we have isomorphisms of $\Z_p[\Delta]^+ \otimes_{\Z_p} \overline{\Q}_p$-modules
\begin{equation}\label{eq:comparedirectsums}
\bigoplus_{f} \overline{\Q}_p(f) \cong M_2(\Gamma_0(N^2))_\m \otimes_{\Z_p} \overline{\Q}_p \cong (\Z_p[\Delta]^+ \otimes_{\Z_p} \overline{\Q}_p)^r \cong \bigoplus_{\chi \in \Xi \cup \{1\}} \overline{\Q}_p(\chi)^r,    
\end{equation}
where $\overline{\Q}_p(f)$ denotes the subspace of $M_2(\Gamma_0(N^2))_\m \otimes_{\Z_p} \overline{\Q}_p$ on which $\TT$ acts by the eigensystem of $f$, and the leftmost direct sum is over all normalized eigenforms in $M_2(\Gamma_0(N^2))_\m$.  By definition of $\chi_f$ we have that $\overline{\Q}_p(f) = \overline{\Q}_p(\chi_f)$ as a $\Z_p[\Delta]^+$-module.  The righthand side of \eqref{eq:comparedirectsums} shows that each $\chi$-eigenspace has dimension $r$.  Thus for any $\chi \in \Xi \cup \{1\}$ there must be exactly $r$ normalized eigenforms $f$ for which $\overline{\Q}_p(\chi_f) = \overline{\Q}_p(\chi)$.

For the last statement, note that for $\chi \in \Xi$, both $E_{\chi,\chi^{-1}}$ and $E_{\chi^{-1}, \chi}$ have $\chi$ as their associated character (in the sense of \cref{defn:chif}), and these are the only such Eisenstein series.  Thus the set of $r$ normalized eigenforms $f$ for which $\chi_f = \chi$ consist of two Eisenstein series and $r-2$ cusp forms, which are necessarily new at $N^2$ since $\chi \neq 1$.
\end{proof}

\section{Constructing $\Theta$ by interpolating residues}\label{sec:Heckeside}
In this section we define the Eisenstein-cuspidal congruence ideal and compute an explicit generator of it.  The definition is given in \cref{subsec:groupringEisSer}, which immediately gives a description of $\TT$ as the fiber product of its Eisenstein and cuspidal parts over the congruence module.  The strategy to compute a generator for the congruence ideal is to interpolate the residues of Eisenstein series.  We recall the residue map and its relation to the congruence ideal in \cref{subsec:congideal+resmap}.  The interpolation is completed in \cref{subsec:InterpolatingResidues}.

\subsection{The $\Z_p[\Delta]$-valued Eisenstein series and the congruence ideal}\label{subsec:groupringEisSer}
Recall from \cref{subsubsec:EisSeries} that for each character $\chi$ of $\Delta$, we have an Eisenstein series $E_{\chi, \chi^{-1}}$, where we interpret $E_{1,1}$ in the sense of \eqref{eq:E11defn}.  Each $E_{\chi, \chi^{-1}}$ is a $\TT$-eigenform and so gives rise to a $\Z_p$-algebra homomorphism $\TT \to \Z_p[\chi]$.  In this section we show that these can be interpolated to give a $\Z_p$-algebra homomorphism $\phi_\TT \colon \TT \to \Z_p[\Delta]$.  By duality, one could view this data as a $\Z_p[\Delta]$-valued Eisenstein series, but for our purposes the homomorphism $\phi_{\TT}$ will suffice.  Nevertheless,  for expositional convenience we denote this $\Z_p[\Delta]$-valued Eisenstein series by $\mathbb{E}$.  Our \textit{Eisenstein ideal} is given by $I \coloneqq \ker \phi_\TT$.  This allows us to define the ideal measuring congruences between $\mathbb{E}$ and cusp forms at the end of \cref{subsec:groupringEisSer}.  

Recall that $\Delta$ is a quotient of $\Gal{\Q(\zeta_N)}{\Q} = (\Z/N\Z)^\times$.  For any integer $n$ that is prime to $N$, we let $[n] \in \Delta \subset \Z_p[\Delta]$ denote the element that is the image of the class of $n$ in $(\Z/N\Z)^\times$ under the natural projection map to $\Delta$.  Let $\hat{\Delta}$ denote the character group of $\Delta$.  

\begin{proposition}\label{prop:GroupRingValuedEisSeries}
There is a surjective $\Z_p$-algebra homomorphism $\phi \colon \TT \to \Z_p[\Delta]$ that sends $T_\ell$ to $[\ell] + \ell[\ell^{-1}]$ for all primes $\ell \nmid N$.
\end{proposition}

\begin{proof}
It suffices to check that $\phi$ is a homomorphism after composing with the inclusion 
\[
\Z_p[\Delta] \hookrightarrow \prod_{\chi \in \hat{\Delta}} \Z_p[\chi]
\]
given by evaluating at each character of $\Delta$, and this can be done componentwise.  It therefore suffices to check that $\phi$ is a homomorphism when composed with the map $\Z_p[\Delta] \to \Z_p[\chi]$ induced by each $\chi \in \hat{\Delta}$.  For all primes $\ell \nmid N$ and $\chi \in \hat{\Delta}$,
\[
\chi \circ \phi (T_\ell) = \chi(\ell) + \ell\chi(\ell^{-1}) = a_\ell(E_{\chi, \chi^{-1}}).
\]
By \cref{cor:reduced}, the $\Z_p$-algebra $\TT$ is generated by the anemic Hecke operators.  Thus $\chi \circ \phi$ is the homomorphism given by the eigensystem of $E_{\chi, \chi^{-1}}$.

For surjectivity, choose primes $\ell_1$ and $\ell_2$ such that $\ell_1 \equiv \ell_2 \bmod N$ generates $(\Z/N\Z)^\times$ and such that $\ell_1 \not\equiv \ell_2 \bmod p$.  This can be done by applying the Chebotarev density theorem to the extension $\Q(\zeta_{Np})/\Q$, whose Galois group is naturally isomorphic to $(\Z/N\Z)^\times \times (\Z/p\Z)^\times$.  Then \[\phi(T_{\ell_1} - T_{\ell_2})=(\ell_1-\ell_2)[\ell_1]\] is a unit multiple of a generator of $\Delta$, so $\phi$ is surjective.
\end{proof}

Since $\TT$ is a $\Z_p[\Delta]^+$-algebra by \cref{cor:TTaZpDeltaPlusAlgebra}, it is natural to ask whether the morphism in \cref{prop:GroupRingValuedEisSeries} is a $\Z_p[\Delta]^+$-algebra homomorphism.  This is the case, as we show in \cref{lem:ZpDeltaPlusAlgebraHom} below.  

\begin{lemma}\label{lem:ZpDeltaPlusAlgebraHom}
The morphism $\phi \colon \TT \to \Z_p[\Delta]$ from \cref{prop:GroupRingValuedEisSeries} is a $\Z_p[\Delta]^+$-algebra homomorphism.
\end{lemma}

\begin{proof}
Fix an element $\tau \in I_N$ such that $\tau$ generates the tame quotient, and let $\delta_\tau$ be the image of $\tau$ under the natural map $I_N \twoheadrightarrow \Gal{\Q_N(\zeta_N)}{\Q_N} = (\Z/N\Z)^\times \twoheadrightarrow \Delta$.  We shall view elements of $\Z_p[\Delta]^+$ as elements of $\TT$ via the natural map $\Z_p[\Delta]^+ \to R \to \TT$.  Since $[\delta_\tau] + [\delta_\tau^{-1}]$ generates $\Z_p[\Delta]^+$ as a $\Z_p$-algebra, it suffices to show that
\[
\phi([\delta_\tau] + [\delta_\tau^{-1}]) = [\delta_\tau] + [\delta_\tau^{-1}].
\]
As in the proof of \cref{prop:GroupRingValuedEisSeries}, it suffices to check that the above equality holds after applying $\chi$ for each $\chi \in \hat{\Delta}$.  By construction of $\phi$ we see that $\chi \circ \phi \colon \TT \to \Z_p[\chi]$ is the homomorphism that gives the eigensystem of $E_{\chi,\chi^{-1}}$.  Thus by \cref{cor:ZpDeltaPlusoneigenforms} it follows that
\[
\chi \circ \phi([\delta_\tau] + [\delta_\tau^{-1}]) = \tr(\rho_{E_{\chi,\chi^{-1}}}(\tau)) = \chi(\tau) + \chi^{-1}(\tau) = \chi(\delta_\tau) + \chi^{-1}(\delta_\tau),
\]
as desired.
\end{proof}

The homomorphism from \cref{prop:GroupRingValuedEisSeries} can be postcomposed with any automorphism of $\Z_p[\Delta]$.  Any one of these maps can be used to describe a congruence ideal, though of course they are no longer $\Z_p[\Delta]^+$-algebra homomorphisms.  In order to ensure that the specialization of the congruence ideal at $\chi$ is the ideal generated by $L(-1, \chi)$ (rather than $L(-1,\chi^{-2})$), we define $\phi_\TT \colon \TT \to \Z_p[\Delta]$ to be the composition of the homomorphism from \cref{prop:GroupRingValuedEisSeries} with the automorphism given by raising elements of $\Delta$ to the $u$-th power, where $u$ is chosen such that $2u \equiv -1 \bmod |\Delta|$.  That is, for any prime $\ell \nmid N$ we have
\begin{equation}\label{eq:phiT}
\phi_\TT(T_\ell) = [\ell^u] + \ell[\ell^{-u}]. 
\end{equation}

Define $I \coloneqq \ker \phi_\TT$.  Note that $I$ is also the kernel of the map from \cref{prop:GroupRingValuedEisSeries}; the automorphism does not change the kernel.

We can now define the congruence ideal that measures congruences between the cuspforms and Eisenstein series parametrized by $\TT$.  Write $\TT^0$ for the maximal quotient of $\TT$ that acts faithfully on $S_2(\Gamma_0(N^2))_{\m}$, and let $\Aa \coloneqq \ker (\TT \to \TT^0)$.  Write $I^0$ for the image of the Eisenstein ideal $I$ in $\TT^0$.

\begin{definition}\label{def:congideal}
The \textit{(Eisenstein--cuspidal) congruence ideal} is the $\Z_p[\Delta]$-ideal $\cc$ satisfying \mbox{$\Z_p[\Delta]/\cc \cong \TT^0/I^0$}, where the isomorphism is induced by $\phi_\TT$.  That is,  $\cc \coloneqq \phi_\TT(\Aa)$.
\end{definition} 

\begin{remark}\label{rem:fiberproduct}
It is useful to note that the definition of the congruence ideal immediately implies that $\TT$ admits a description as a fiber product, namely
\[
\TT \cong \TT^0 \times_{\TT/I^0 \cong \Z_p[\Delta]/\cc} \Z_p[\Delta],
\]
where $t \in \TT$ is identified with $(t \bmod \Aa, \phi_\TT(t))$.
\end{remark}

\subsection{The residue sequence}\label{subsec:congideal+resmap}
The tool we use to calculate a generator of the Eisenstein--cuspidal congruence ideal $\cc$ is the residue exact sequence, which we recall here and explain how it relates to the congruence ideal.  In particular, the key conclusion of this section is \cref{prop:InterpolationPropertySuffices}, which shows that one can find a generator for $\cc \otimes \Z_p[\zeta_N]$ by interpolating the residues of the Eisenstein series.  The derivation of the residue exact sequence in prime-square level is worked out in detail in \cite[\S 2]{LangWake22} along with the relevant Hecke actions.  We freely use the results derived there, but we recall the notation for the convenience of the reader.

Let $\OK = \Z_p[\zeta_N]$ with residue field $k$.  Recall that the modular curve $X_0(N^2)$ has a model over $\OK$ with $N+1$ cusps.  The cusps, viewed as elements of $\mathbb{P}^1(\Q) = \Q \cup \{\infty\}$, are $\{\infty\}, \{0\}, \{1/N\}, \{2/N\}, \ldots, \{(N-1)/N\}$.  Let $\Div^0(C; \OK)$ denote the free $\OK$-module on the degree-0 divisors of $X_0(N^2)$ that are supported at the cusps.  There is a Hecke-equivariant exact sequence
\[
0 \to S_2(\Gamma_0(N^2); \OK) \to M_2(\Gamma_0(N^2); \OK) \xrightarrow{\Res} \Div^0(C; \OK) \to 0
\]
that remains exact after completing at the maximal ideal $\mm$ \cite[Proposition 2.2]{LangWake22}.  

Taking $\OK$-linear duals yields  the short exact sequence
\[
0 \to \Hom_{\OK}(\Div^0(C; \OK)_\mm, \OK) \xrightarrow{\Res^*} \TT_\OK \to \TT^0_\OK \to 0.
\]
thanks to \cref{cor:duality} and \cite[Theorem 5.3.1]{HidaBlueBook}.  A priori this is just a short exact sequence of $\OK$-modules, but it is easy to check that the map \mbox{$\TT_\OK \to \TT^0_\OK$} is the natural projection map.  In particular, it is a ring map, and hence its kernel $\Hom_{\OK}(\Div^0(C; \OK)_\mm; \OK)$ is actually a $\TT_\OK$-ideal, which we previously denoted by $\Aa_\OK$.  By \cref{def:congideal}, the congruence ideal $\cc_\OK$ is given by the image of $\Hom_\OK(\Div^0(C; \OK)_\mm, \OK)$ in $\OK[\Delta]$ under the map $\phi_{\TT_\OK}$ given in \eqref{eq:phiT}, as illustrated by the following diagram
\[
\xymatrix{
0 \ar@{->}[r] & \Hom_{\OK}(\Div^0(C; \OK)_\mm, \OK)\ar@{->}[r]\ar@{->>}^{\phi_{\TT_\OK}}[d] & \TT_\OK\ar@{->}[r]\ar@{->>}^{\phi_{\TT_\OK}}[d] & \TT_\OK^0\ar@{->}[r]\ar@{->>}[d] & 0 \\
0 \ar@{->}[r] & \cc_\OK\ar@{->}[r] & \OK[\Delta]\ar@{->}[r] & \OK[\Delta]/\cc_\OK \cong \TT_\OK^0/I^0_\OK\ar@{->}[r] & 0.
}
\]
Moreover, the congruence ideal $\cc_\OK$ can be computed by embedding $\OK[\Delta]$ into the product $\prod_{\chi \in \hat{\Delta}} \OK[\chi]$ via the projection maps $\chi \colon \OK[\Delta] \to \OK[\chi]$ induced by each character $\chi$.  Note that the composition $\chi \circ \phi_{\TT_\OK}$ is given by evaluation at the Eisenstein series $E_{\chi^u, \chi^{-u}}$.  

The residue map $\Res \colon M_2(\Gamma_0(N^2); \OK) \to \Div^0(C; \OK)$, which induces the inclusion $\Hom_\OK(\Div^0(C; \OK)_\mm, \OK) \hookrightarrow \TT_\OK$, can be computed explicitly, as we now recall.  For a cusp $c$ of $X_0(N^2)$ and $f \in M_2(\Gamma_0(N^2); \OK)$, let $a_0^c(f)$ be the constant term in the $q$-expansion of $f$ around the cusp $c$, which is well defined since the weight is even.  In particular, $a_0^\infty(f)$ is what we have been calling $a_0(f)$.  The residue map is given explicitly by 
\begin{equation}\label{eq:ResFormula}
\Res(f) =  a_0^{\{\infty\}}(f)\cdot \{\infty\} + N^2a_0^{\{0\}}(f) \cdot \{0\} + N\sum_{i = 1}^{N-1} a_0^{\{i/N\}}(f)\cdot \{i/N\};
\end{equation}
see \cite[\S 2.2]{LangWake22} or \cite[\S 4.5]{Ohta99} for more details.  In particular, the factor next to each constant term comes from the width of the corresponding cusp.  The fact that the residue of a modular form always gives a divisor of degree zero follows from the Residue Theorem.

The residue exact sequence shows that $\Div^0(C; \OK)_\mm$ is a free $\OK$-module of rank $|\Delta|$.  Suppose that $\{b_i \in \Div^0(C; \OK)_\mm \colon 1 \leq i \leq |\Delta|\}$ is an $\OK$-basis for $\Div^0(C; \OK)_\mm$.  Let $\{b_i^* \in \Hom_\OK(\Div^0(C; \OK)_\mm, \OK) \colon 1 \leq i \leq |\Delta|\}$ be its dual basis; that is, $b_i^*(b_j) = \delta_{ij}$.  Thus the congruence ideal $\cc_{\OK}$ is the $\OK$-span of the elements $\phi_\TT \circ \Res^*(b_i^*) \in \OK[\Delta]$ for $1 \leq i \leq |\Delta|$.  For any character $\chi \in \hat{\Delta}$, the $\chi$-projection of $\phi_\TT \circ \Res^*(b_i^*)$ is given by
\[
b_i^*(\Res(E_{\chi^u,\chi^{-u}})).
\]
Since $\Res$ is $\TT_\OK$-equivariant, we know that $\Res(E_{\chi, \chi^{-1}})$ must be a $\TT_\OK$-eigenvector with the same eigensystem as $E_{\chi, \chi^{-1}}$.  With that in mind, for each $\chi \in \hat{\Delta}$ we define
\[
u_\chi \coloneqq \sum_{j = 1}^{N-1} \chi(j)(\{j/N\} - \{0\}) \in \Div^0(C; \OK[\chi]).
\]

\begin{proposition}\label{lem:uchievals}
For each $\chi \in \hat{\Delta}$, the vector $u_\chi$ is a $\TT_\OK$-eigenvector whose eigensystem matches that of $E_{\chi, \chi^{-1}}$.  Moreover, $u_\chi$ is a basis for the subspace
\[
\{v \in \Div^0(C; \OK[\chi])_\mm \colon T_\ell v = (\chi(\ell) + \ell\chi^{-1}(\ell))v \text{ for all primes } \ell \nmid N \text{ and } T_Nv = 0\}.
\]
\end{proposition}

\begin{proof}
This is an easy calculation using the formulae for the action of the Hecke operators on cusps in \cite[Proposition 2.2]{LangWake22}.  The last sentence follows from the fact that $u_\chi$ remains nonzero over the residue field.
\end{proof}

By \cref{lem:uchievals} and the $\TT_\OK$-equivariance of $\Res$, there exist constants $a_\chi \in \OK[\chi]$ such that 
\begin{equation}\label{eq:achis}
    \Res(E_{\chi,\chi^{-1}}) = a_\chi u_\chi. 
\end{equation}
The constants $a_\chi$ are key to computing the congruence ideal, but to see this we need an explicit integral basis for $\Div^0(C; \OK)_\mm$.  One might hope that the $u_\chi$'s are such a basis; indeed, they are a basis over $\overline{\Q}_p$. 

\begin{corollary}\label{cor:uchisQpbasis}
The set $\{u_\chi \colon \chi \in \hat{\Delta}\}$ is a basis for the vectors space $\Div^0(C; \OK)_\mm \otimes_{\OK} \overline{\Q}_p$.
\end{corollary}

\begin{proof}
Localizing the residue exact sequence at $\mm$ and then tensoring to $\overline{\Q}_p$ shows that $\Div^0(C;\OK)_\mm \otimes \overline{\Q}_p$ has a basis given by the images of the Eisenstein series $E_{\chi, \chi^{-1}}$ for $\chi \in \hat{\Delta}$.  Each of these is in turn a nonzero multiple of the corresponding $u_\chi$, which proves the claim.
\end{proof}

However, the $u_\chi$'s are not an integral basis for $\Div^0(C; \OK)_\mm$ because they are all congruent modulo $p$.  Therefore we introduce, for each $g \in \Delta$,
\[
b_g \coloneqq \sum_{\substack{j = 1\\ [j] = g}}^{N-1} (\{j/N\} - \{0\}) \in \Div^0(C; \OK).
\]
Note that we can write
\[
u_\chi = \sum_{g \in \Delta} \chi(g)b_g.
\]

\begin{proposition}\label{prop:integralbasisforDiv0}
The set $\{b_g \colon g \in \Delta\}$ is an $\OK$-basis for $\Div^0(C; \OK)_\mm$.
\end{proposition}

\begin{proof}
We begin by showing that $b_g \in \Div^0(C; \OK)_\mm$ for all $g \in \Delta$.  Since $\Div^0(C; \OK)$ is a free $\OK$-module, it follows that $\Div^0(C; \OK)_\mm$ is a direct summand and that $\Div^0(C; \OK)_\mm = (\Div^0(C; \OK)_\m \otimes \overline{\Q}_p) \cap \Div^0(C; \OK)$.  Hence it suffices to show that $b_g \in \Div^0(C; \OK)_\mm \otimes \overline{\Q}_p$.

Let $V$ be the span of $\{b_g \colon g \in \Delta\}$ in $\Div^0(C; \OK) \otimes \overline{\Q}_p$.  Clearly the dimension of $V$ is at most $|\Delta|$.  On the other hand, the equation preceding this proposition shows that $u_\chi \in V$ for all $\chi \in \hat{\Delta}$.  Thus by \cref{cor:uchisQpbasis} we have $\Div^0(C; \OK)_\mm \otimes \overline{\Q}_p \subseteq V$.  But the dimension of $\Div^0(C; \OK)_\mm \otimes \overline{\Q}_p$ is $|\Delta|$ and so we must have equality.  Therefore each $b_g$ is in $\Div^0(C; \OK)_\mm \otimes \overline{\Q}_p$, hence in $\Div^0(C; \OK)_\mm$.

To show that the $b_g$'s form an $\OK$-basis for $\Div^0(C; \OK)_\mm$, it suffices to show that the images of the $b_g$'s in $\Div^0(C; \OK)_\mm \otimes k$ are linearly independent.  This follows easily from the fact that $\{\{j/N\} - \{0\} \colon 1 \leq j \leq N-1\}$ are linearly independent in $\Div^0(C; \OK)\otimes k$.
\end{proof}

Let $\{b_g^* \in \Hom_\OK(\Div^0(C; \OK)_\mm, \OK) \colon g \in \Delta\}$ be the dual basis of $\{b_g \colon g \in \Delta\}$.  We now have 
\[
b_g^*(\Res(E_{\chi^u,\chi^{-u}})) = b_g^*(a_\chi u_\chi) =  a_{\chi^u}\chi^u(g)
\]
for all $g \in \Delta$ and $\chi \in \hat{\Delta}$.  We will show in \cref{subsec:InterpolatingResidues} that there is an element $\Theta_{\Res} \in \OK[\Delta]$ with the interpolation property that $\chi(\Theta_{\Res}) = a_{\chi^u}$ for all $\chi \in \hat{\Delta}$.  Morally, $\Theta_{\Res}$ should be thought of as the residue of the $\OK[\Delta]$-valued Eisenstein series $\mathbb{E}$, but rather than making this precise we simply show that there is an element with the desired interpolation property.  The following proposition shows that the interpolation property suffices to show that $\Theta_{\Res}$ generates $\cc_\OK$.

\begin{proposition}\label{prop:InterpolationPropertySuffices}
Suppose that $\Theta_{\Res} \in \OK[\Delta]$ is an element that satisfies $\chi(\Theta_\Res) = a_{\chi^u}$ for all $\chi \in \hat{\Delta}$.  Then $\Theta_\Res$ generates the congruence ideal $\cc_\OK$.
\end{proposition}

\begin{proof}
Under the inclusion $\prod_\chi \chi \colon \OK[\Delta] \hookrightarrow \prod_\chi \OK[\chi]$ we have
\[
(\prod_\chi \chi)(\Theta_{\Res} \cdot g^u) = (a_{\chi^u}\chi^u(g))_\chi = (b_g^*(\Res(E_{\chi^u, \chi^{-u}})))_\chi = (\prod_\chi \chi)(\phi_\TT(b_g^* \circ \Res)),
\]
and hence
\[
\Theta_{\Res} \cdot g^u = \phi_\TT(b_g^* \circ \Res)
\]
for all $g \in \Delta$.  Recall that $\cc_\OK$ is the $\OK$-span of the elements $\{\phi_\TT(b_g^* \circ \Res) \colon g \in \Delta\}$.  Since $g^u \in \OK[\Delta]^\times$, the above equation shows both that $\Theta_{\Res} \in \cc_\OK$ and that $\cc_\OK$ is generated by $\Theta_{\Res}$.
\end{proof}

\subsection{Interpolating residues}\label{subsec:InterpolatingResidues}
We now calculate the constants $a_\chi$ such that $\Res(E_{\chi,\chi^{-1}}) = a_\chi u_\chi$, as in \eqref{eq:achis}, and interpolate them into the desired element $\Theta_\Res \in \OK[\Delta]$, as in \cref{prop:InterpolationPropertySuffices}.  We end this section by modifying $\Theta_{\Res}$ to obtain an explicit generator $\Theta \in \Z_p[\Delta]$ of the congruence ideal $\cc$. 

From \eqref{eq:achis} we see that for $1 \leq j \leq N-1$, the coefficient of the cusp $\{j/N\}$ in $a_\chi u_\chi$ is $a_\chi\chi(j)$.  On the other hand, \eqref{eq:ResFormula} shows that the coefficient of $\{j/N\}$ in $\Res(E_{\chi,\chi^{-1}})$ is $N\cdot a_0^{\{j/N\}}(E_{\chi,\chi^{-1}})$.  Thus for any such $j$ we have
\begin{equation}\label{eq:achiswithconstantterms}
    a_\chi = N\chi^{-1}(j)a_0^{\{j/N\}}(E_{\chi, \chi^{-1}}).
\end{equation}

When $\chi \neq 1$, Ohta computed the constant term of $E_{\chi,\chi^{-1}}$ at each cusp \cite[Proposition 2.5.5]{Ohta03}.  For the convenience of the reader we note that for $\chi \neq 1$, our $E_{\chi,\chi^{-1}}$ is equal to Ohta's $E_2(\theta,\psi)$ for $\theta = \chi^{-1}$ and $\psi = \chi$.  Using his calculations we arrive at the following explicit formula for $a_\chi$.

\begin{lemma}\label{lem:achiexplicit}
For $\chi \in \Delta$ we have
\[
a_\chi = \begin{cases}\frac{N\tau(\chi^2)}{2\tau(\chi)}L(-1,\chi^{-2}) & \chi \neq 1\\
\frac{1-N^2}{2N}L(-1,1) & \chi = 1,
\end{cases}
\]
where $\tau(\chi) \coloneqq \sum_{a = 1}^{N-1} \chi(a)\zeta_N^a \in \OK$ is the usual Gauss sum.
\end{lemma}

\begin{proof}
The equation for $\chi \neq 1$ follows from Ohta's formula for $a_0^{\{j/N\}}(E_{\chi,\chi^{-1}})$.  We take Ohta's $\theta$ to be our $\chi^{-1}$ and his $\psi$ to be our $\chi$.  This means that his $u$ and $c$ are both equal to $N$, and his $a$ is our $j$.  Since $\chi$ has odd order we have $\chi(-c/u) = \chi(-1) = 1$, and both $\theta = \chi^{-1}$ and $\psi\overline{\theta} = \chi^2$ have conductor $N$.  Moreover, the product over primes in Ohta's formula is empty in our case.  Thus we see that
\[
a_0^{\{j/N\}}(E_{\chi,\chi^{-1}}) = \frac{\tau(\chi^2)}{2\tau(\chi)}\chi(j)L(-1,\chi^{-2}).
\]
The result now follows from \eqref{eq:achiswithconstantterms}.

The proof that $\Res(E_{1,1}) = \frac{N^2-1}{24N}\cdot u_1$ can be done using the Atkin--Lehner operator as in the proof of \cite[Corollary 2.6]{LangWake22}.  Indeed, from the definition of $a_1$ and $u_1$, we see that $a_1(1-N)$ is the coefficient of $\{0\}$ in $\Res(E_{1,1})$.  This can be computed as the constant term of the $q$-expansion at $\infty$ of the level-$N^2$ Atkin--Lehner operator applied to $E_{1,1}$.  The details are carried out in \cite[Corollary 2.6]{LangWake22}.  Note that in the notation of that paper, $\mathfrak{c}$ is equal to our $u_1$ and their $E$ is equal to $NE_{1,1}$.  Although \cite{LangWake22} often assume that $N \not\equiv 1 \bmod p$, that assumption plays no role in the proof of \cite[Corollary 2.6]{LangWake22}.  To get the formula in the statement of the lemma, note that $L(-1,1) = \frac{-1}{12}$.
\end{proof}

We can now interpolate the values $a_\chi$ found in \cref{lem:achiexplicit} into the desired element $\Theta_{\Res} \in \OK[\Delta]$.  Recall from the introduction that 
\begin{equation}\label{eq:zetaelement}
\zeta \coloneqq \frac{-N}{2}\sum_{i = 1}^{N-1} B_2(i/N)[i + N\Z] \in \Z_p[(\Z/N\Z)^\times].
\end{equation}
Let $\zetadelta$ be the image of $\zeta$ in $\Z_p[\Delta]$ under the natural map $\Z_p[(\Z/N\Z)^\times] \to \Z_p[\Delta]$.  Fix a primitive $N$-th root of unity $\zeta_N$.  For $a \in (\Z/|\Delta|\Z)^\times$, define 
\[
\tau_a \coloneqq \sum_{i=1}^{N-1} \zeta_N^i[i^a] \in \OK[\Delta].
\]
Note that $\tau_a$ is a unit in $\OK[\Delta]$ since under the augmentation map to $\Z_p$, it maps to $\sum_{i=1}^{N-1} \zeta_N^i = -1 \not\equiv 0 \bmod p$.  Finally, define
\[
\eta \coloneqq \sum_{i=1}^{N-1} [i] \in \Z_p[\Delta].
\]
Using these, we define
\[
\Theta_{\Res} \coloneqq \tau_{-1}(2\tau_u)^{-1}\left(\zetadelta - NL(-1,1)\eta\right)\left(N - \frac{N+1}{N}\eta\right) \in \OK[\Delta].
\]

\begin{theorem}\label{thm:Theta_resInterpolation}
The element $\Theta_{\Res}$ satisfies $\chi(\Theta_{\Res}) = a_{\chi^u}$ for all $\chi \in \hat{\Delta}$, so $\Theta_{\Res}$ generates the congruence ideal $\cc_\OK$.
\end{theorem}

\begin{proof}
The second statement follows from the first by \cref{prop:InterpolationPropertySuffices}.  To see that $\Theta_{\Res}$ has the claimed interpolation property, use \cref{lem:achiexplicit} and \cite[Theorem 4.2]{washington} to write 
\[
a_{\chi^u} = \begin{cases}\frac{\tau(\chi^{-1})}{2\tau(\chi^u)}L(-1,\chi)N = \frac{-N^2\tau(\chi^{-1})}{4\tau(\chi^u)}\sum_{i = 1}^{N-1}\chi(i)B_2(i/N) & \chi \neq 1\\
\frac{1-N^2}{2N}L(-1,1) & \chi = 1.
\end{cases}
\]
On the other hand, $\zetadelta, \tau_a$, and $\eta$ have the interpolation properties given in \cref{table:interpolation} \cite[Theorem 4.2]{washington}.

\begin{table}[h]
\begin{tabular}{|c|c|c|}
\hline
 & $\chi \neq 1$  & $\chi = 1$ \\ 
\hline
$\chi(\zetadelta)$ & $L(-1,\chi)$ & $(1-N)L(-1,1)$ \\  
\hline
$\chi(\tau_a)$ & $\tau(\chi^a)$ & $-1$\\
\hline
$\chi(\eta)$ & $0$ & $N-1$\\
\hline
\end{tabular}
\caption{}
\label{table:interpolation}
\end{table}
The interpolation property is now a straightforward calculation.
\end{proof}

Finally we descend from $\OK[\Delta]$ to $\Z_p[\Delta]$.  Define
\begin{equation}\label{eq:Thetaandefn}
\Theta \coloneqq \zetadelta - NL(-1,1)\eta \in \Z_p[\Delta].   
\end{equation}

\begin{corollary}\label{cor:Theta_an}
The element $\Theta$ generates the congruence ideal $\cc$ defined in \cref{def:congideal}.
\end{corollary}

\begin{proof}
First note that $\Theta$ generates $\cc_\OK$ as an $\OK[\Delta]$-ideal.  Indeed, by \cref{thm:Theta_resInterpolation}, $\Theta$ differs from the generator $\Theta_{\Res}$ of $\cc_\OK$ by $\tau_{-1}(2\tau_u)^{-1}(N - \frac{N+1}{N}\eta)$, which we claim is a unit.  We have seen that each $\tau_a$ is a unit in $\OK[\Delta]$, as is $2$.  The augmentation map sends $N - \frac{N+1}{N}\eta$ to $1/N$, which is congruent to $1$ modulo $p$, hence a unit.  Thus $\cc$ and $\Theta$ generate the same $\OK[\Delta]$-ideal.

To conclude that $\Theta$ generates $\cc$, it suffices to show that $\OK[\Delta]$ is faithfully flat as a $\Z_p[\Delta]$-module \cite[Proposition 3.5.9]{BourbakiCommAlg}.  In fact, since the inclusion $\Z_p[\Delta] \hookrightarrow \OK[\Delta]$ is an inclusion of local rings, it suffices to show that $\OK[\Delta]$ is a flat $\Z_p[\Delta]$-module \cite[Proposition 3.5.9]{BourbakiCommAlg}.  This follows from the fact that $\OK$ is free, and hence flat, over $\Z_p$.  Indeed, since $\Z_p[\Delta]$ is flat over itself, the $\Z_p[\Delta]$-module $\OK \otimes_{\Z_p} \Z_p[\Delta] = \OK[\Delta]$ is flat \cite[Proposition 2.7.8]{BourbakiCommAlg}.
\end{proof}

In \cref{sec:rankcriteria} it is useful to have a multiplicative relationship between $\Theta$ and $\zeta_\Delta$.  Let $e \coloneqq \frac{1}{|\Delta|} \sum_{g \in \Delta} [g] \in \Q_p[\Delta]$ denote the central idempotent attached to the trivial character.  Define
\begin{equation}\label{eq:diffdefn}
\diff \coloneqq 1 + Ne \in \Q_p[\Delta].  
\end{equation}

\begin{lemma}\label{lem:comparingThetaanWithZetaDelta}
We have $\Theta = \zeta_\Delta \cdot \diff$, and $\chi(\diff) \in \Z_p[\chi]^\times$ for all $\chi \in \hat{\Delta}$.
\end{lemma}

\begin{proof}
As usual, it suffices to check that the equality holds after applying every character $\chi \in \hat{\Delta}$.  Note that 
\[
\chi(\diff) = \begin{cases}
    1 & \chi \neq 1\\
    1 + N & \chi = 1.
\end{cases}
\]
Using \cref{table:interpolation} we see that if $\chi \neq 1$, then 
\[
\chi(\zeta_\Delta \cdot \diff) = L(-1,\chi^{-1}) \cdot 1 = \chi(\Theta).
\]
Applying the trivial character to both $\Theta$ and $\zeta_\Delta \cdot \diff$ gives $(1-N^2)L(-1,1)$, which proves the lemma.
\end{proof}

\section{Ring theoretic properties of $\TT$ coming from $\Z_p[\Delta]^+$-freeness}\label{sec:ringpropertiesofTT}
In this section we use the freeness theorem established in \cref{thm:TisFree} to give a presentation for $\TT$.  This allows us to deduce in \cref{subsec:presentationforT} that $\TT$ is a local complete intersection ring.  In \cref{subsec:monogenicity} we prove \cref{thmA:ringprop}, showing that $\TT$ is monogenic if and only if $r \coloneqq \rk_{\Z_p} \TT$ is equal to $2$.  The presentation plays an essential role in \cref{sec:rankcriteria}. 

\subsection{A presentation for $\TT$}\label{subsec:presentationforT}
We begin by giving a presentation of $\TT$ as a $\Z_p[\Delta]^+$-algebra.  Recall that $\TT(N)$ has a presentation of the form
\begin{equation}\label{eq:levelNpresentation}
\TT(N) = \Z_p[x]/(xf(x)),
\end{equation}
where $f$ is a monic distinguished polynomial of degree $r-1$, and $x$ is a generator of the Eisenstein ideal \cite[\S II.19]{Mazur}.  We fix once and for all a generator $x$ of Mazur's Eisenstein ideal as well as a lift $\tilde{x} \in \TT$ of $x$.

\begin{lemma}\label{lem:presentingToverZpDeltaPlus}
There is a monic polynomial $c(t) \in \Z_p[\Delta]^+[t]$ of degree $r$ such that there is an isomorphism of $\Z_p[\Delta]^+$-algebras $\Z_p[\Delta]^+[t]/(c(t)) \cong \TT$ that sends $t$ to $\tilde{x}$.    
\end{lemma}

\begin{proof}
By \cref{cor:descendtoTN}, the kernel of the map $\TT \twoheadrightarrow \TT(N)$ is $\II^+\TT$, where $\II^+$ is the augmentation ideal of $\Z_p[\Delta]^+$.  By Nakayama's Lemma $\tilde{x}$ generates $\TT$ as a $\Z_p[\Delta]^+$-algebra.  By \cref{thm:TisFree} we see that $\TT \cong (\Z_p[\Delta]^+)^r$ as a $\Z_p[\Delta]^+$-module, and multiplication by $\tilde{x}$ acts as a $\Z_p[\Delta]^+$-linear endomorphism of $\TT$.  Let $c(t) \in \Z_p[\Delta]^+$ be the minimal polynomial of $\tilde{x}$, which is monic and has degree at most $r$.  On the other hand, if $d$ is the degree of $c(t)$ then $\TT$ is generated as a $\Z_p[\Delta]^+$-module by $1, \tilde{x}, \ldots, \tilde{x}^{d-1}$, which means that $\TT(N)$ is generated by $1, x, \ldots, x^{d-1}$ as a $\Z_p$-module.  Since $\TT(N)$ has rank $r$ we have $r \leq d$, hence $r = d$.
\end{proof}

We give an explicit presentation of $\Z_p[\Delta]^+$ over $\Z_p$.  Fix a generator $\delta$ of $\Delta$.  Let $v_p$ be the $p$-adic valuation on $\Z_p$ normalized so that $v_p(p) = 1$.

\begin{lemma}\label{lem:presentationofZpDeltaPlus}
There is an isomorphism of $\Z_p$-algebras $\Z_p[y]/(y\Psi(y)) \cong \Z_p[\Delta]^+$ given by sending $y$ to $[\delta] + [\delta^{-1}] - 2$, where $\Psi(y)$ is a monic distinguished polynomial of degree $\frac{|\Delta| - 1}{2}$ with $v_p(\Psi(0)) = v_p(|\Delta|)$.
\end{lemma}

\begin{proof}
This is \cite[Lemma 4.1]{LangWake25}.
\end{proof}

Let $G(Y) = Y\Psi(Y)$, which has degree $\frac{|\Delta| + 1}{2}$.  For each $a \in \Z_p[\Delta]^+$ there is a unique polynomial $g_a(Y) \in \Z_p[Y]$ of degree at most $\frac{|\Delta| - 1}{2}$ such that $g_a(Y)$ maps to $a$ under the map $\Z_p[Y] \twoheadrightarrow \Z_p[\Delta]^+$ that sends $Y$ to $[\delta] + [\delta^{-1}] - 2$.  In particular, write 
\[
c(t) = \sum_{i = 0}^r a_it^i
\]
with $c$ as in \cref{lem:presentingToverZpDeltaPlus} and $a_i \in \Z_p[\Delta]^+$.  Define
\[
F(X, Y) \coloneqq \sum_{i = 0}^r g_{a_i}(Y)X^i \in \Z_p[X, Y].
\]

\begin{proposition}\label{prop:presentationT}
The map $\Z_p[X, Y] \to \TT$ given by sending $X$ to $\tilde{x}$ and $Y$ to $[\delta] + [\delta^{-1}] - 2$ is surjective. Its kernel is generated by $G(Y)$ and $F(X,Y)$.
\end{proposition}

\begin{proof}
The surjectivity follows from the fact that $[\delta] + [\delta^{-1}] - 2$ generates $\Z_p[\Delta]^+$ as a $\Z_p$-algebra by \cref{lem:presentationofZpDeltaPlus} and $\tilde{x}$ generates $\TT$ as a $\Z_p[\Delta]^+$-algebra by \cref{lem:presentingToverZpDeltaPlus}.  The definition of $G(Y)$ and $F(X,Y)$ make it clear that they are in the kernel of the map.  Thus we have a surjection
\[
\Z_p[X,Y]/(G(Y), F(X, Y)) \twoheadrightarrow \TT.
\]
Examining degrees, we see that $\Z_p[X,Y]/(G(Y), F(X, Y))$ can be generated as a $\Z_p$-module by $\{X^iY^j \colon 0 \leq i \leq r-1, 0 \leq j \leq \frac{|\Delta| - 1 }{2}\}$.  Let $m \coloneqq (\frac{|\Delta| + 1}{2})r$.  Thus there is a surjection of $\Z_p$-modules 
\[
\Z_p^m \twoheadrightarrow \Z_p[X,Y]/(G(Y), F(X,Y)) \twoheadrightarrow \TT.
\]
Since $m = \rk_{\Z_p} \TT$ by \cref{thm:TisFree}, the map is  an isomorphism.  
\end{proof}

\begin{corollary}\label{cor:TisCI}
The ring $\TT$ is a complete intersection ring.
\end{corollary}

\begin{proof}
We apply \cite[\href{https://stacks.math.columbia.edu/tag/09Q1}{Lemma 09Q1}]{stacks-project} with $R = \Z_p[X,Y], \p = (X, Y)$, and $I = (G(Y), F(X,Y))$.  Note that $R_\p$ has dimension $3$, and the quotient $R_\p/I_\p$ has dimension $1$ since it is a finitely generated $\Z_p$-module.  Since $I$ is generated by $\dim(R_\p) - \dim(R_\p/I_\p) = 2$ elements, it follows that the completion of $R_\p/I_\p$ is a complete intersection.  By \cref{prop:presentationT} we have that $R_\p/I_\p \cong \TT$, which is already complete, so the result follows. 
\end{proof}

In what follows, we identify $\TT$ with the presentation given in \cref{prop:presentationT} and $\TT(N)$ with its presentation $\Z_p[x]/(xf(x))$.  Thus it is useful to understand the map $\TT \twoheadrightarrow \TT(N)$ from \eqref{eq:TTtoTTN} in terms of these presentations.

\begin{proposition}\label{prop:quotientviapresentations}
The map $\Z_p[X,Y]/(G(Y), F(X,Y)) \twoheadrightarrow \Z_p[x]/(xf(x))$ given by $X \mapsto x$ and $Y \mapsto 0$ makes the following diagram commute:
\[
\xymatrix{
\TT \ar@{->}[d]_{\cong} \ar@{->>}[r] & \TT(N)\ar@{->}[d]_{\cong}\\
\Z_p[X,Y]/(G(Y), F(X,Y))\ar@{->>}[r] & \Z_p[x]/(xf(x)),
}
\]
where the top horizontal map comes from \eqref{eq:TTtoTTN}, the left vertical map is the isomorphism of \cref{prop:presentationT}, and the right vertical map comes from \eqref{eq:levelNpresentation}.
\end{proposition}

\begin{proof}
By \cref{cor:descendtoTN} the kernel of the map $\TT \to \TT(N)$ is the ideal of $\TT$ generated by the augmentation ideal $\II^+$ of $\Z_p[\Delta]^+$, which is generated by $Y$ by construction.  Also, $X$ corresponds to the chosen lift $\tilde{x} \in \TT$ of $x$, so by definition this gets mapped back to $x$ under the quotient.  
\end{proof}

\begin{corollary}\label{cor:basicpropertiesofF}
We have 
\begin{enumerate}[label=(\roman*)]
    \item $F(x,0) = xf(x)$ in $\Z_p[x]$;
    \item $g_{a_0}(0) = 0$, and hence $g_{a_0}(Y) = Yh(Y)$ for some $h(Y) \in \Z_p[Y]$;
    \item $g_{a_i}(0) \in p\Z_p$ for all $1 \leq i \leq r-1$.
\end{enumerate}
\end{corollary}

\begin{proof}
By \cref{prop:quotientviapresentations} we see that $\Z_p[x]/(G(0), F(x,0)) = \Z_p[x]/(xf(x))$.  Since $G(0) = 0$ it follows that $F(x,0)$ and $xf(x)$ generate the same ideal in $\Z_p[x]$.  As both are monic polynomials of degree $r$, we must have $F(x,0) = xf(x)$.  Expanding this equation gives
\[
xf(x) = F(x,0) = \sum_{i=0}^r g_{a_i}(0)x^i.
\]
Thus the constant term $g_{a_0}(0)$ vanishes, which immediately gives that $g_{a_0}(Y)$ is divisible by $Y$ in $\Z_p[Y]$.  Since $f$ is a distinguished polynomial, all of the nonleading coefficients of $xf(x)$ are divisible by $p$.  That is, $g_{a_i}(0) \in p\Z_p$ for all $1 \leq i \leq r-1$.
\end{proof}

\subsection{Criterion for monogenicity of $\TT$}\label{subsec:monogenicity}
\cref{prop:presentationT} shows that $\TT$ can always be generated by two elements as a $\Z_p$-algebra.  In this section we show that the question of whether one can make do with fewer generators is governed by $r$.  In particular, $\TT$ is monogenic if and only if $r = 2$. 

We begin by recasting the monogenicity of $\TT$ in terms of the presentation for $\TT$ found in \cref{prop:presentationT}.  Recall from \cref{cor:basicpropertiesofF} that $g_{a_0}(Y) = Yh(Y)$ for some $h(Y) \in \Z_p[Y]$.  In particular, \cref{lem:monogenicityrecast} shows that the monogenicity of $\TT$ is equivalent to $h(0) \in \Z_p^\times$.  The rest of the subsection is devoted to showing that $h(0) \in \Z_p^\times$ if and only if $r = 2$.

\begin{lemma}\label{lem:monogenicityrecast}
The ring $\TT$ is a monogenic $\Z_p$-algebra if and only if $h(0) \in \Z_p^\times$.
\end{lemma}

\begin{proof}
The minimal number of generators for $\TT$ as a $\Z_p$-algebra is given by the dimension of the reduced tangent space, which is equal to $\dim_{\F_p} \m/(\m^2,p)$.  Viewing $\TT$ as the image of $\Z_p[X,Y]$ under the map from \cref{prop:presentationT}, we see that $\m$ is generated by the images of $p, X, Y$, since $G(Y)$ and $F(X,Y)$ are contained in the ideal generated by $X$ and $Y$ by \cref{lem:presentationofZpDeltaPlus} and \cref{cor:basicpropertiesofF}.  To compute $\m/(\m^2, p)$ we need to compute $G(Y)$ and $F(X,Y)$ modulo $(\m^2, p)$.  Note that $(\m^2, p)$ is generated by the images of $p, X^2, XY,$ and $Y^2$.  By definition of $G$ and \cref{lem:presentationofZpDeltaPlus}, we have $G(Y) = Y\Psi(Y) \equiv \Psi(0)Y \equiv 0 \bmod (p, X^2, XY, Y^2)$.  On the other hand,
\begin{align*}
F(X, Y) = \sum_{i = 0}^r g_{a_i}(Y)X^i &\equiv g_{a_0}(0) + g_{a_1}(0)X + h(0)Y \bmod (X^2, XY, Y^2)\\
&\equiv h(0)Y \bmod (X^2, XY, Y^2, p),
\end{align*}
where the last equivalence follows from \cref{cor:basicpropertiesofF}.  Thus we have
\begin{align*}
\m/(\m^2,p) &= (p, X, Y)/(G(Y), F(X,Y), p, X^2, XY, Y^2)\\
&= (p,X,Y)/(h(0)Y, p, X^2, XY, Y^2).
\end{align*}

If $h(0) \in \Z_p^\times$, then this presentation shows that $\m/(\m^2,p)$ has $\F_p$-dimension $1$ with basis $X$, and so $\TT$ is monogenic.  Otherwise $\m/(\m^2, p)$ has $\F_p$-dimension $2$ with basis elements $X$ and $Y$, and hence $\TT$ is not monogenic.
\end{proof}

From \eqref{eq:phiT} we have the $\Z_p$-algebra homomorphism 
\[
\phi_\TT \colon \TT \twoheadrightarrow \Z_p[\Delta]
\]
that interpolates the homomorphisms coming from the various Eisenstein series.  Recall that $\phi_\TT$ is obtained by composing the $\Z_p[\Delta]^+$-algebra homomorphism $\phi$ from \cref{lem:ZpDeltaPlusAlgebraHom} with the automorphism of $\Z_p[\Delta]$ given by raising elements of $\Delta$ to the $u$-th power, where $2u \equiv -1 \bmod |\Delta|$.  In particular, if we fix a generator $\delta$ of $\Delta$ and view $[\delta] + [\delta^{-1}]$ as an element of $\TT$ via the natural map $\Z_p[\Delta]^+ \to \TT$, then 
\[
\phi_\TT([\delta] + [\delta^{-1}]) = [\delta^u] + [\delta^{-u}].
\]
Postcomposing $\phi_\TT$ with a character $\chi \in \hat{\Delta}$ gives a surjective $\Z_p$-algebra homomorphism 
\[
\phi_{\TT,\chi} \colon \TT \twoheadrightarrow \Z_p[\chi].
\]
Note that $F(X,Y)$ sees both $r$ (as the degree of $X$) and $h(0)$ (as the coefficient of $Y$).  Our strategy is to push forward $F(X,Y)$ via $\phi_{\TT,\chi}$ for a nontrivial character $\chi$.  Using the fact that $F(X,Y)$ maps to $0$ in $\TT$, we deduce a relationship between $r$ and $h(0)$ by exploiting the fact that $\Z_p[\chi]$ is a discrete valuation ring.

Fix a character $\chi \in \hat{\Delta}$ of order $|\Delta|$, and fix a generator $\delta$ of $\Delta$.  Let $\zeta_{p^s} \coloneqq \chi(\delta^u)$, which is a primitive $p^s$-th root of unity.  Since $\zeta_{p^s} - 1$ is a uniformizer in $\Z_p[\chi]$, we let $v$ be the valuation on $\Z_p[\chi]$ normalized so that
\[
v(\zeta_{p^s} - 1) = 1.
\]
In particular, $v(\Z_p[\chi]) = \Z_{\geq 0}$.

For the rest of this section and in \cref{sec:rankcriteria}, we freely view any ring homomorphism out of $\TT$ --- for instance $\phi_{\TT}$ or $\phi_{\TT, \chi}$ --- as a homomorphism from $\Z_p[X,Y]$ by composing with the map $\Z_p[X,Y] \to \TT$ that sends $X$ to $\tilde{x}$ and $Y$ to $[\delta] + [\delta^{-1}] - 2$ as in \cref{prop:presentationT}.  Since $Y$ corresponds to $[\delta] + [\delta^{-1}] - 2$ under the isomorphism $\Z_p[Y]/(G(Y)) \cong \Z_p[\Delta]^+$ of \cref{lem:presentationofZpDeltaPlus}, it follows that 
\begin{equation}\label{eq:phi(Y)}
\phi_{\TT,\chi}(Y) = \phi_{\TT, \chi}([\delta] + [\delta^{-1}] - 2) = \zeta_{p^s} + \zeta_{p^s}^{-1} - 2.
\end{equation}
In particular
\begin{equation}\label{eq:vofphi(Y)}
v(\phi_{\TT, \chi}(Y)) = 2,
\end{equation}
since $\zeta_{p^s} + \zeta_{p^s}^{-1} - 2$ is a uniformizer for the degree 2 subfield $\Q_p(\zeta_{p^s} + \zeta_{p^s}^{-1})$ of the totally ramified field $\Q_p(\zeta_{p^s}) = \Q_p(\chi)$.

\begin{proposition}\label{prop:boundonphi(X)}
We have $v(\phi_{\TT, \chi}(X)) \geq \frac{2}{r}$, and  equality holds if $h(0) \in \Z_p^\times$.
\end{proposition}

\begin{proof}
Applying $\phi_{\TT,\chi}$ to $F(X,Y)$, which maps to $0$ in $\TT$, we see that
\begin{equation}\label{eq:phiXr}
(\phi_{\TT,\chi}(X))^r = -((\zeta_{p^s} + \zeta_{p^s}^{-1} - 2)h(\zeta_{p^s} + \zeta_{p^s}^{-1} - 2) + \sum_{i = 1}^{r-1} g_{a_i}(\zeta_{p^s} + \zeta_{p^s}^{-1} - 2)\phi_{\TT,\chi}(X)^i).
\end{equation}
Since $g_{a_i}(0) \in p\Z_p$ for all $1 \leq i \leq r - 1$ by \cref{cor:basicpropertiesofF}, it follows that $\zeta_{p^s} + \zeta_{p^s}^{-1} - 2$ divides $g_{a_i}(\zeta_{p^s} + \zeta_{p^s}^{-1} - 2)$ in $\Z_p[\zeta_{p^s} + \zeta_{p^s}^{-1}]$.  Therefore $2 = v(\zeta_{p^s} + \zeta_{p^s}^{-1} - 2) \leq v((\phi_{\TT,\chi}(X))^r) = rv(\phi_{\TT,\chi}(X))$, which proves the inequality.

Now suppose that $h(0) \in \Z_p^\times$.  To show that equality holds, it suffices to show that exactly one of the terms in the sum on the righthand side of \eqref{eq:phiXr} has valuation equal to $2$.  If $h(0)$ is a unit then so is $h(\zeta_{p^s} + \zeta_{p^s}^{-1} - 2)$, and hence $v((\zeta_{p^s} + \zeta_{p^s}^{-1} - 2)h(\zeta_{p^s} + \zeta_{p^s}^{-1} - 2)) = 2$.  On the other hand, for any $1 \leq i \leq r - 1$ we have 
\[
v(g_{a_i}(\zeta_{p^s} + \zeta_{p^s}^{-1} - 2)\phi_{\TT,\chi}(X)^i) \geq 2 + \frac{2i}{r} > 2,
\]
which completes the proof.
\end{proof}

\begin{corollary}\label{cor:h(0)unitimpliesr=2}
If $h(0) \in \Z_p^\times$, then $r = 2$.
\end{corollary}

\begin{proof}
Since $h(0) \in \Z_p^\times$, \cref{prop:boundonphi(X)} implies that $v(\phi_{\TT,\chi}(X)) = \frac{2}{r}$.  But $v(\phi_{\TT,\chi}(X)) \in \Z_{\geq 0}$ and $r \geq 2$, so we must have $r = 2$.
\end{proof}

We now prove a sharper result on $v(\phi_{\TT,\chi}(X))$ than \cref{prop:boundonphi(X)}.

\begin{proposition}\label{prop:vofphi(X)}
We have $v(\phi_{\TT,\chi}(X)) = 1$.
\end{proposition}

\begin{proof}
First note that \cref{prop:boundonphi(X)} implies that $v(\phi_{\TT,\chi}(X)) > 0$, from which it follows that $v(\phi_{\TT, \chi}(X)) \geq 1$ since $v(\phi_{\TT,\chi}(X)) \in \Z$.

Since $\phi_{\TT,\chi}$ is surjective, every element of $\Z_p[\chi]$ can be written as a polynomial in $\phi_{\TT,\chi}(X)$ and $\phi_{\TT,\chi}(Y)$ with $\Z_p$-coefficients.  In particular, write 
\[
\zeta_{p^s} - 1 = \sum_{i,j \geq 0} b_{i,j}\phi_{\TT,\chi}(X)^i\phi_{\TT,\chi}(Y)^j,
\]
with $b_{i,j} \in \Z_p$.  Thus we have
\begin{equation}\label{eq:minimum}
1 = v(\zeta_{p^s} - 1) \geq \min_{i,j \geq 0} \{v(b_{i,j}) + iv(\phi_{\TT,\chi}(X)) + jv(\phi_{\TT,\chi}(Y))\}.
\end{equation}
For any pair $(i,j)$ we have
\[
v(b_{i,j}) + iv(\phi_{\TT,\chi}(X)) + jv(\phi_{\TT,\chi}(Y)) \geq v(b_{i,j}) + i + 2j.
\]
Note that since $b_{i,j} \in \Z_p$ and $v(\Z_p) = [\Q_p(\zeta_{p^s}) \colon \Q_p]\Z$, it follows that the minimum in \eqref{eq:minimum} must be achieved at some $(i_0,j_0)$ where $b_{i_0,j_0} \in \Z_p^\times$.  In that case we have $1 \geq i_0 + 2j_0$ from which it follows that $j_0 = 0$ and $i_0 \in \{0, 1\}$.

Suppose for contradiction that the minimum in \eqref{eq:minimum} is achieved at $i_0 = j_0 = 0$.  Then $b_{0,0} \in \Z_p^\times$, and we have
\[
\zeta_{p^s} - 1 = b_{0,0} + \sum_{(0,0) \neq (i,j)} b_{i,j}\phi_{\TT,\chi}(X)^i\phi_{\TT,\chi}(Y)^j.
\]
Since $v(\phi_{\TT,\chi}(X)), v(\phi_{\TT,\chi}(Y)) \geq 1$, this implies that $\zeta_{p^s} - 1$ is a unit in $\Z_p[\chi]$, which is a contradiction.  

Therefore the minimum must be achieved when $i_0 = 1$ and $j_0 = 0$ and $b_{1,0} \in \Z_p^\times$.  In fact, this is the only index where the minimum of \eqref{eq:minimum} is achieved.  Indeed, the argument in the previous paragraph shows that $b_{0,0} \in p\Z_p$, and for any other pair $(i,j)$ we have $i + 2j \geq 2$, which is bigger than 1 and hence would violate \eqref{eq:minimum} if it were the minimum.  Thus the inequality in \eqref{eq:minimum} is an equality, and we have
\[
1 = v(\zeta_{p^s} - 1) = v(b_{1,0}) + v(\phi_{\TT,\chi}(X)) = v(\phi_{\TT,\chi}(X)),
\]
as desired.
\end{proof}

\begin{theorem}\label{thm:monogeniciffr=2}
The ring $\TT$ is monogenic if and only if $r = 2$.
\end{theorem}

\begin{proof}
If $\TT$ is monogenic, then $r = 2$ by \cref{lem:monogenicityrecast} and \cref{cor:h(0)unitimpliesr=2}.  Conversely, suppose that $r = 2$.  Applying $\phi_{\TT,\chi}$ to $F(X,Y) = 0 \in \TT$ gives
\[
(\phi_{\TT,\chi}(X))^2 = -(g_{a_1}(\zeta_{p^s} + \zeta_{p^s}^{-1} - 2)\phi_{\TT,\chi}(X) + (\zeta_{p^s} + \zeta_{p^s}^{-1} - 2)h(\zeta_{p^s} + \zeta_{p^s}^{-1} - 2)).
\]
Taking valuations and applying \cref{prop:vofphi(X)} we obtain
\[
2 \geq \min\{v(g_{a_1}(\zeta_{p^s} + \zeta_{p^s}^{-1} - 2)) + 1, 2 + v(h(\zeta_{p^s} + \zeta_{p^s}^{-1} - 2)))\}.
\]
Since $v(g_{a_1}(\zeta_{p^s} + \zeta_{p^s}^{-1} - 2)) \geq 2$ by \cref{cor:basicpropertiesofF}, the minimum must be equal to $2 + v(h(\zeta_{p^s} + \zeta_{p^s}^{-1} - 2)))$.  But this is at most $2$ if and only if $0 = v(h(\zeta_{p^s} + \zeta_{p^s}^{-1} - 2)))$, which is equivalent to $h(0) \in \Z_p^\times$.  Thus $r = 2$ implies that $h(0) \in \Z_p^\times$, which gives that $\TT$ is monogenic by \cref{lem:monogenicityrecast}.
\end{proof}

\section{Deducing rank criteria}\label{sec:rankcriteria}
In this section we prove \cref{thmA:merel,thmA:lecouturier,thmA:higherrank}, which relate $\ord(\bar{\zeta})$ to $r-1$ when these quantities are small.  In \cref{subsubsec:ordandThetaan} we show that, for our purposes, we may replace $\ord(\bar{\zeta})$ with the valuation of $\chi(\Theta)$ for a generator $\chi$ of $\hat{\Delta}$ and $\Theta$ the generator of the congruence ideal defined in \eqref{eq:Thetaandefn}.  On the other hand, in \cref{subsubsec:randcoeffsofH} we combine the presentation of $\TT$ from \cref{subsec:presentationforT} with the fiber product description of $\TT$ from \cref{rem:fiberproduct} to create a polynomial $H(X,Y) \in \Z_p[X,Y]$ whose coefficients encode $r-1$.  This information suffices to prove \cref{thmA:merel}; see \cref{subsubsec:recoveringMerel}.  After deducing additional information about $H$, we prove \cref{thmA:lecouturier} in \cref{subsubsec:recoveringLecouturier}.  Finally, we investigate the first consequences of spoiler coefficients in \cref{subsubsec:higherrank}, proving \cref{thmA:higherrank}. 

\subsection{Relating $\ord(\bar{\zeta})$ to the valuation of $\Theta$}\label{subsubsec:ordandThetaan}
Choose a generator $g$ of $(\Z/N\Z)^\times$.  Recall that $g-1$ generates the augmentation ideal $\mathcal{I}$ of $\F_p[(\Z/N\Z)^\times]$. Expand $\zeta$, defined in \eqref{eq:zetaelement},  in terms of powers of $g-1$:
\begin{equation}\label{eq:zetaexpansion}
\zeta = \frac{N-1}{12} + \sum_{i = 1}^{N-2} b_i(g-1)^i,  
\end{equation}
for some $b_i \in \Z_p$.  Since $p \mid N - 1$, we see that $\bar{\zeta} \in \mathcal{I}$.  Recall that $\ord(\bar{\zeta})$ is the largest $m$ such that $\bar{\zeta} \in \mathcal{I}^m$.

Fix $\chi \in \hat{\Delta}$ of order $p^s$, which we also view as a character of $(\Z/N\Z)^\times$.  Set $\zeta_{p^s} \coloneqq \chi(g) \in \Z_p[\chi]$.  Let $v$ be the $p$-adic valuation on $\Z_p[\chi]$ normalized such that $v(\zeta_{p^s} - 1) = 1$.  This is the same valuation on $\Z_p[\chi]$ as in \cref{subsec:monogenicity}, even though we may have changed which $p^s$-th root of unity is called $\zeta_{p^s}$.

\begin{proposition}\label{prop:calculateordzeta}
With notation as in \eqref{eq:zetaexpansion}, we have $\ord(\bar{\zeta}) = \min\{1 \leq i \leq N-2 \colon b_i \in \Z_p^\times\}$.  Moreover, 
\begin{enumerate}[label=(\roman*)]
    \item\label{item:ordzetasmall} if $\ord(\bar{\zeta}) < v(p) = (p-1)p^{s-1}$, then $\ord(\bar{\zeta}) = v(\chi(\Theta))$;
    \item\label{item:ordzetabig} if $\ord(\bar{\zeta}) \geq v(p)$, then $v(\chi(\Theta)) \geq v(p)$.
\end{enumerate}
\end{proposition}

\begin{proof}
The first statement is immediate from the definition of $\ord$ and the fact that $g-1$ generates $\mathcal{I}$.

Recall from \cref{lem:comparingThetaanWithZetaDelta} that $\chi(\zeta) (= \chi(\zeta_\Delta))$ and $\chi(\Theta)$ differ by $\chi(\diff)$, which is a unit in $\Z_p[\chi]$.  Therefore $v(\chi(\Theta)) = v(\chi(\zeta_\Delta))$; we work with $\zeta_\Delta$ in place of $\Theta$ for the rest of the proof.  Applying $\chi$ to \eqref{eq:zetaexpansion}, we see that 
\begin{equation}\label{eq:chiofzetaDelta}
\chi(\zeta_\Delta) = \frac{N-1}{12} + \sum_{i=1}^{N-2} b_i(\zeta_{p^s} - 1)^i.  
\end{equation}
For each $1\le i<\ord(\overline{\zeta})$ we have $v(b_i(\zeta_{p^s}-1)^i)= v(b_i)+i\ge v(p)$ by the definition of $v$ and the first claim of the proposition. Furthermore, if $i \geq \ord(\bar{\zeta})$, then $v(b_i(\zeta_{p^s}-1)^i)\ge i\geq \ord(\overline{\zeta})$ for all $i>\ord(\overline{\zeta})$. Therefore
\[v(\chi(\zeta_\Delta))\ge \min_{1 \leq i \leq N-1} \{v((N-1)/12), v(b_i(\zeta_{p^s} - 1)^i)\} \geq \min\{v(p),\ord(\overline{\zeta})\}, \]
from which \ref{item:ordzetabig} is immediate. 

If $i_0:=\ord(\overline{\zeta})<v(p)$, then $b_{i_0} \in \Z_p^\times$ and hence $v(b_{i_0}(\zeta_{p^s}-1)^{i_0})=i_0 = \ord(\bar{\zeta})$.  For $i > i_0$ we have $v(b_i(\zeta_{p^s} - 1)^i) \geq i > i_0$.  For $i < i_0$ we have $v(b_i(\zeta_{p^s} - 1)) \geq v(p) > i_0$.  Therefore the minimum is achieved at a unique index, giving $v(\chi(\zeta_\Delta))=i_0 = \ord(\bar{\zeta})$ and proving \ref{item:ordzetasmall}.
\end{proof}

\subsection{Relating $r$ to the coefficients of a polynomial representing $\Theta$}\label{subsubsec:randcoeffsofH}
The element $\Theta$, which generates the congruence ideal by \cref{cor:Theta_an}, can be viewed as an element of $\TT$ using \cref{rem:fiberproduct}.  Namely, we can consider the element $(0,\Theta) \in \TT^0 \times_{\TT^0/I^0 \cong \Z_p[\Delta]/(\Theta)} \Z_p[\Delta] \cong \TT$, which generates the kernel of the natural map $\TT \twoheadrightarrow \TT^0$.  Since $\TT \cong \Z_p[X,Y]/(G(Y),F(X,Y))$ we can choose 
\begin{equation}\label{eq:Hcoeffs}
H(X,Y) = \sum_{i = 0}^{r-1} \sum_{j = 0}^{\frac{|\Delta|-1}{2}} c_{i,j}X^iY^j \in \Z_p[X,Y]
\end{equation}
such that $H(X,Y)$ maps to $(0,\Theta)$ under the identifications above.  The second step in our proofs of \cref{thmA:merel,thmA:lecouturier} is to show that the rank $r$ of $\TT(N)$ is related to the coefficients of $H$, specifically the coefficients $c_{i,0}$ that have no $Y$-term in their monomial; see \cref{cor:relatingranktoHcoeffs} below.  Let $\TT^0(N)$ denote the cuspidal quotient of $\TT(N)$.

\begin{proposition}\label{prop:Hequalsunittimesf}
Both $H(x,0)$ and $f(x)$ generate the same ideal in $\Z_p[x]$.
\end{proposition}

\begin{proof}
By \cref{rem:fiberproduct}, \cref{cor:Theta_an}, and \cite[Lemma 3.2.2]{WWE} both $\TT$ and $\TT(N)$ have a fiber product structure arising from the Eisenstein--cuspidal congruence ideal, namely
\[
\TT = \TT^0 \times_{\Z_p[\Delta]/(\Theta)} \Z_p[\Delta] \text{ and } \TT(N) = \TT^0(N) \times_{\Z_p/(N-1)\Z_p} \Z_p.
\]
The natural surjection $\TT \twoheadrightarrow \TT(N)$ respects these decompositions.  That is, the natural map $\TT^0 \twoheadrightarrow \TT^0(N)$, given by sending $T_\ell$ to $T_\ell$, together with the augmentation map $\Z_p[\Delta] \twoheadrightarrow \Z_p$ induce a surjection $\TT^0 \times_{\Z_p[\Delta]/(\Theta)} \Z_p[\Delta] \twoheadrightarrow \TT^0(N) \times_{\Z_p/(N-1)\Z_p} \Z_p$.  This is the natural surjection $\TT \to \TT^0(N)$ of \eqref{eq:TTtoTTN} that sends $T_\ell$ to $T_\ell$ for all primes $\ell$.  In summary, we have the following commutative diagram:
\begin{equation}\label{eq:fiberproductcommdiag}
\begin{tikzcd}[row sep=1.2cm]
    \Z_p[X,Y] \arrow[d, two heads] \arrow[r, two heads] &[0.8cm] \mathbb{T} \arrow[r, "\sim"] \arrow[d, two heads] &[0.8cm]
    \mathbb{T}^0 \times_{\Z_p[\Delta]/(\Theta)} \Z_p[\Delta] \ar[d, two heads, shift right=1.5cm] \ar[d, two heads, shift left=1.2cm] \ar[d, shorten=2mm,dashed,two heads] \\
    \Z_p[x] \arrow[r, two heads] &\TT(N) \arrow[r, "\sim"] & \TT^0(N) \times_{\Z_p/(N-1)}  
    \mathbb{Z}_p.
\end{tikzcd}
\end{equation}

The map $\TT \twoheadrightarrow \TT(N)$ can be described in terms of presentations as the map that sends the image of $X$ to $x$ and the image of $Y$ to $0$ as in \cref{prop:quotientviapresentations}.  Thus the image of $H(X,Y)$ maps to $H(x,0)$.  By definition of $H(X,Y)$ this must match up with the image of $(0, \Theta) \in \TT^0 \times_{\Z_p[\Delta]/(\Theta)} \Z_p[\Delta]$ under the right vertical maps of \eqref{eq:fiberproductcommdiag}, namely $(0, \frac{N^2-1}{12})$.  Under the identification $\TT(N) = \Z_p[x]/(xf(x))$ we have $\TT^0(N) = \Z_p[x]/(f(x))$ and $\Z_p = \Z_p[x]/(x)$.  In particular, the ideal generated by $f(x)$ in $\Z_p[x]/(xf(x))$ corresponds to the ideal generated by $(0,N-1)$ in $\TT^0(N) \times_{\Z_p/(N-1)\Z_p} \Z_p$.  Since $(0, \frac{N^2-1}{12})$ is a unit multiple of $(0, N-1)$ it follows that $H(x,0)$ must be a unit multiple of $f(x)$; that is, they generate the same ideal in $\Z_p[x]$.
\end{proof}

\begin{corollary}\label{cor:relatingranktoHcoeffs}
For $0 \leq i \leq r-1$ we have $c_{i,0} \in \Z_p^\times$ if and only if $i = r-1$.
\end{corollary}

\begin{proof}
Recall that $f$ is a monic distinguished polynomial of degree $r-1$.  By \cref{prop:Hequalsunittimesf} there is $u \in \Z_p^\times$ such that 
\[
uf(x) = H(x,0) = \sum_{i=0}^{r-1} c_{i,0}x^i.
\]
Since $f$ is distinguished, $c_{i,0} \in \Z_p^\times$ if and only if $i = r-1$, as desired.
\end{proof}

\cref{prop:Hequalsunittimesf} is sufficient to prove \cref{thmA:merel,thmA:lecouturier}. Nevertheless, the following corollary establishes that one recovers the cuspidal quotient $\TT^0(N)$ of Mazur's Hecke algebra by taking the quotient of $\TT^0$ by the augmentation ideal $\II^+$ of $\Z_p[\Delta]^+$, as claimed in equation \eqref{eq:cuspidalmoduloaug} of \cref{subsub:cuspquot}. 

\begin{corollary}\label{cor:Tcuspaugideal} The map \eqref{eq:TTtoTTN} induces an isomorphism
$\TT^0 / \II^+\TT^0  \cong \TT^0(N)$.
\end{corollary}

\begin{proof}
By construction $H(X,Y)$ generates the kernel of the natural map $\TT \twoheadrightarrow \TT^0$.  Hence we get an induced presentation  $\TT^0 \cong \frac{\Z_p[X,Y]}{(G(Y),F(X,Y),H(X,Y))}$. The corollary follows from the following sequence of isomorphisms:
\begin{align*}
\TT^0 /   \II^+\TT^0 & \cong \frac{\Z_p[X,Y]}{(G(Y),F(X,Y),H(X,Y))} \Bigg/ \frac{(Y)+(G(Y),F(X,Y),H(X,Y))}{(G(Y),F(X,Y),H(X,Y))}  \\
&  \cong \frac{\Z_p[X,Y]}{(G(Y),F(X,Y),H(X,Y),Y)} \\ & \cong \frac{\Z_p[x]}{(F(x,0),H(x,0))} \cong \frac{\Z_p[x]}{(f(x))} \cong \TT^0(N).
\end{align*}
That $Y$ generates $\II^+$ gives the first isomorphism.  Standard properties of quotients give the second isomorphism. Since $G(0)=0$, the third isomorphism follows by considering the evaluation  map $\Z_p[X,Y] \rightarrow \Z_p[x]$ sending $X$ to $x$ and $Y$ to $0$. The fourth isomorphism follows since $F(x,0)=xf(x)$ by \cref{cor:basicpropertiesofF} and $(H(x,0))=(f(x))$  by \cref{prop:Hequalsunittimesf}. The last isomorphism corresponds to Mazur's presentation for $\TT^0(N)$.
\end{proof}

\subsection{Recovering Merel's criterion}\label{subsubsec:recoveringMerel}
We give a new proof of \cref{thmA:merel}.

\begin{theorem}\label{thm:mereltake2}
\cref{thmA:merel} is true.  That is, $r = 2$ if and only if $\ord(\bar{\zeta}) = 1$.
\end{theorem}

\begin{proof}
\cref{prop:calculateordzeta} shows that $\ord(\bar{\zeta}) = 1$ if and only if $v(\chi(\Theta)) = 1$, and \cref{cor:relatingranktoHcoeffs} shows that $r = 2$ is equivalent to $c_{1,0} \in \Z_p^\times$.  Thus it suffices to show that $v(\chi(\Theta)) = 1$ if and only if $c_{1,0} \in \Z_p^\times$.

By definition of $H(X,Y)$ we have $\phi_\TT(H(X,Y)) = \Theta$.  Thus
\[
\chi(\Theta) = \phi_{\TT,\chi}(H(X,Y)) = H(\phi_{\TT,\chi}(X), \phi_{\TT,\chi}(Y)) = \sum_{i,j \geq 0} c_{i,j}\phi_{\TT,\chi}(X)^i\phi_{\TT,\chi}(Y)^j.
\]
From \cref{prop:vofphi(X)} and \eqref{eq:vofphi(Y)} we have that $v(\phi_{\TT,\chi}(X)) = 1$ and $v(\phi_{\TT,\chi}(Y)) = 2$.  Therefore
\[
v(\chi(\Theta)) \geq \min\{v(c_{i,j}) + i + 2j \colon i,j \geq 0\},
\]
with equality holding if one of the terms is strictly minimal.  Note that $v(c_{0,0}) \geq v(p) = (p-1)p^{s-1} > 1$, and if $i \geq 2$ or $j \geq 1$, then $v(c_{i,j}) + i + 2j > 1$.

If $v(\chi(\Theta)) = 1$ then the inequalities in the previous sentence force the minimum to occur at $(i,j) = (1,0)$ with $v(c_{1,0}) = 0$.  Thus $c_{1,0} \in \Z_p^\times$.  Conversely, suppose that $c_{1,0} \in \Z_p^\times$.  Then $v(c_{1,0}) + 1 + 2\cdot0 = 1$, which is strictly less than the valuations of all of the other terms by the last sentence in the previous paragraph.  Thus $v(\chi(\Theta)) = 1$.  This completes the proof.
\end{proof}

\subsection{Recovering Lecouturier's criterion}\label{subsubsec:recoveringLecouturier}
To recover \cref{thmA:lecouturier} by a similar method to that used in \cref{thm:mereltake2}, we need information about the coefficient $c_{0,1}$ of the $Y$-term in $H$.  

\begin{lemma}\label{lem:MakeHEisenstein}
Fix $1 \leq t \leq s$.  The polynomial $H(X, \zeta_{p^t} + \zeta_{p^t}^{-1} - 2)$ is a distinguished polynomial of degree $r-1$ in $\Z_p[\zeta_{p^t} + \zeta_{p^t}^{-1}][X]$.  Moreover, if $c_{0,1} \in \Z_p^\times$, then $H(X, \zeta_{p^t} + \zeta_{p^t}^{-1} - 2)$ is an Eisenstein polynomial.
\end{lemma}

\begin{proof}
We have
\begin{align}\label{eq:Hevallambdafcoeff}
H(X, \zeta_{p^t} + \zeta_{p^t}^{-1} - 2) = \sum_{i=0}^{r-1} (\sum_{j=0}^{\frac{|\Delta|-1}{2}} c_{i,j}(\zeta_{p^t} + \zeta_{p^t}^{-1} - 2)^j)X^i.
\end{align}
By \cref{cor:relatingranktoHcoeffs} we see that $c_{i,0} \in p\Z_p$ for all $1 \leq i \leq r-2$ and $c_{r-1,0} \in \Z_p^\times$.  This shows that the leading coefficient is a unit, and all of the other coefficients are nonunits.  The constant term is $c_{0,0} + \sum_{j=1}^{(|\Delta| - 1)/2} c_{0,j}(\zeta_{p^t} + \zeta_{p^t}^{-1}-2)^j$.  Since $c_{0,0} \in p\Z_p$, we see that this term has minimal positive valuation if and only if $c_{0,1}$ is a unit, as desired.
\end{proof}

We now use the equidistribution result, \cref{prop:equidistributionofcharacters}, to deduce that $r$ is related to the reducibility of the polynomial studied in \cref{lem:MakeHEisenstein}.  Recall from \cref{sec:modrepthy} that we have fixed $\tau \in I_N$ whose projection to the tame inertia group is a generator, and we denote by $\delta \in \Delta$ the image of $\tau$ under the natural map $I_N \twoheadrightarrow \Gal{\Q_N(\zeta_N)}{\Q_N} \twoheadrightarrow \Delta$.  Moreover, for a $\TT$-eigenform $f$, we write $\lambda_f \colon \TT \to \overline{\Z}_p$ for the algebra homomorphism given by its Hecke eigenvalues.

\begin{proposition}\label{prop:Hreducible}
Fix $1 \leq t \leq s$.  If $r > 2$, then $H(X, \zeta_{p^t} + \zeta_{p^t}^{-1} - 2)$ is reducible in $\Z_p[\zeta_{p^t} + \zeta_{p^t}^{-1}][X]$.
\end{proposition}

\begin{proof}
Fix $\chi \in \Xi$ such that $\chi(\delta) = \zeta_{p^t}$.  Since $r > 2$, by \cref{prop:equidistributionofcharacters} there is a normalized eigenform $f \in S_2(\Gamma_0(N^2))_\m \otimes \overline{\Z}_p$ such that $\chi_f = \chi$.  Since $f$ is a cusp form, $\lambda_f$ must factor through the quotient map $\TT \twoheadrightarrow \TT^0$.  Viewing $\TT = \TT^0 \times_{\Z_p[\Delta]/(\Theta)} \Z_p[\Delta]$, this is just the projection onto the first coordinate, whose kernel is generated by $H(X,Y)$ by construction.  Thus 
\begin{align}\label{eq:evalHcuspform}
0 = \lambda_f(H(X,Y)) = H(\lambda_f(X),\lambda_f(Y)) = H(\lambda_f(X), \zeta_{p^t} + \zeta_{p^t}^{-1} - 2).
\end{align}

Note that $\Q_p(f)$ is generated over $\Q_p$ by $\lambda_f(X)$ and $\lambda_f(Y) = \zeta_{p^t} + \zeta_{p^t}^{-1} - 2$.  Thus $\lambda_f(X)$ generates $\Q_p(f)$ over $\Q_p(\zeta_{p^t} + \zeta_{p^t}^{-1})$.  Let $\sigma$ be an automorphism of the Galois closure of $\Q_p(f)$ over $\Q_p(\zeta_{p^t} + \zeta_{p^t}^{-1})$.  We claim that $\sigma(f) \in S_2(\Gamma_0(N^2))_\m \otimes \overline{\Z}_p$ with $\chi_{\sigma(f)} = \sigma \circ \chi_f = \chi_f \in \Xi$.  The latter claim follows from the fact that $\rho_{\sigma(f)} = \sigma \circ \rho_f$.  Indeed, applying both sides to $\tau$ and taking traces gives $\chi_{\sigma(f)} = \sigma \circ \chi_f$.  It remains to show that $\sigma(f) \equiv E_{1,1} \bmod \p$.  For all primes $\ell$ we have $a_\ell(f) - \ell - 1 \in \p$.  Applying $\sigma$ yields $\sigma(a_\ell(f)) - \ell - 1 \in \sigma(\p) = \p$, which gives the desired congruence.

Suppose for contradiction that $H(X, \zeta_{p^t} + \zeta_{p^t}^{-1} - 2)$ is irreducible in $\Z_p[\zeta_{p^t} + \zeta_{p^t}^{-1}][X]$.  Since $H(X, \zeta_{p^t} + \zeta_{p^t}^{-1} - 2)$ has degree $r-1$, there are $r-1$ distinct roots for where $\sigma$ can send $\lambda_f(X)$.  This gives $r-1$ Galois conjugate forms of $f$ that are Eisenstein modulo $\p$ and that have $\chi$ as their associated character.  These forms are all cuspidal since $f$ is, which contradicts the last sentence of \cref{prop:equidistributionofcharacters}.  Therefore $H(X, \zeta_{p^t} + \zeta_{p^t}^{-1} - 2)$ is reducible, as claimed.
\end{proof}

\begin{corollary}\label{cor:c01}
If $c_{0,1} \in \Z_p^\times$, then $r = 2$.
\end{corollary}

\begin{proof}
Since $c_{0,1} \in \Z_p^\times$ by assumption, it follows that for any $1 \leq t \leq s$ we have that $H(X, \zeta_{p^t} + \zeta_{p^t}^{-1} - 2)$ is an Eisenstein polynomial by \cref{lem:MakeHEisenstein}.  In particular, it is irreducible.  By \cref{prop:Hreducible}, it follows that $r \leq 2$, which finishes the proof since $r \geq 2$ by Mazur's work.
\end{proof}

\begin{theorem}\label{thm:lecouturiertake2}
\cref{thmA:lecouturier} is true.  That is, $r = 3$ if and only if $\ord(\bar{\zeta}) = 2$.
\end{theorem}

\begin{proof}
\cref{prop:calculateordzeta} shows that $\ord(\bar{\zeta}) = 2$ if and only if $v(\chi(\Theta)) = 2$, and \cref{cor:relatingranktoHcoeffs} shows that $r = 3$ is equivalent to $c_{2,0} \in \Z_p^\times$.  Thus it suffices to show that $v(\chi(\Theta)) = 2$ if and only if $c_{2,0} \in \Z_p^\times$.

As in the proof of \cref{thm:mereltake2} we have
\begin{equation}\label{eq:keyinequality}
v(\chi(\Theta)) \geq \min\{v(c_{i,j}) + i + 2j \colon i,j \geq 0\},   
\end{equation}
with equality holding if one of the terms is strictly minimal.  Note that $v(c_{0,0}) \geq v(p) = (p-1)p^{s-1} > 2$, and $v(c_{i,j}) + i + 2j > 2$ if $(i,j) \not\in \{(1,0), (0,1), (2,0)\}$.  Moreover, if $v(c_{i,j}) + i + 2j \leq 2$ then we must have $v(c_{i,j}) = 0$ since $c_{i,j} \in \Z_p$ and $v(\Z_p) \subseteq (p-1)p^{s-1}\Z$.

First suppose that $v(\chi(\Theta)) = 2$.  If the minimum in \eqref{eq:keyinequality} occurs at $(1,0)$ then $v(c_{1,0}) \leq 2$ and hence $v(c_{1,0}) = 0$.  Thus $v(c_{1,0}) + 1 + 0 = 1$.  For all other $(i,j)$ we have $v(c_{i,j}) + i + 2j > 1$.  Thus there is a unique term that achieves the minimum in \eqref{eq:keyinequality}, and we must have $v(\chi(\Theta)) = 1$, a contradiction.  If the minimum in \eqref{eq:keyinequality} occurs at $(0,1)$ then $v(c_{0,1}) + 2 \leq 2$ and hence $v(c_{0,1}) = 0$.  By \cref{cor:c01} it follows that $r = 2$, and hence $c_{1,0} \in \Z_p^\times$ by \cref{cor:relatingranktoHcoeffs}.  But we have just seen that this is impossible.  Therefore the minimum is achieved at $(2,0)$.  We have $v(c_{2,0}) + 2 \leq 2$ and hence $v(c_{2,0}) = 0$.  Thus $c_{2,0} \in \Z_p^\times$, as desired.

Conversely, suppose that $c_{2,0} \in \Z_p^\times$.  By \cref{cor:relatingranktoHcoeffs} it follows that $c_{1,0} \not\in \Z_p^\times$ and hence $v(c_{1,0}) \geq v(p) = (p-1)p^{s-1}$.  Therefore the minimum in \eqref{eq:keyinequality} cannot occur at $(1,0)$.  If $v(c_{0,1}) + 0 + 2\cdot 1 \leq 2$ then $c_{0,1} \in \Z_p^\times$.  By \cref{cor:c01} we have that $r = 2$, and hence $c_{1,0} \in \Z_p^\times$ by \cref{cor:relatingranktoHcoeffs}, a contradiction.  Therefore the minimum occurs at $(2,0)$ and equality holds.  That is, $v(\chi(\Theta)) = v(c_{2,0}) + 2 + 0 = 2$.
\end{proof}

\subsection{Higher rank}\label{subsubsec:higherrank}
We now use our methods to investigate the first degenerate case when exactly one of $\ord(\bar{\zeta})$ and $r-1$ is equal to $3$.  In particular, we prove \cref{prop:c11sufficientcondition,thmA:higherrank} in this section.

Assume henceforth that exactly one of $\ord(\bar{\zeta})$ and $r-1$ is equal to $3$.  One can find examples where both of the inequalities $\ord(\bar{\zeta}) > r-1 = 3$ and $r-1 > \ord(\bar{\zeta}) = 3$ occur in \cite[Table 1]{WWE}.  When $N = 4229$ and $p = 7$, we have $\ord(\bar{\zeta}) = 4 > 3 = r-1$.  When $N = 3671$ and $p = 5$, we have $r-1 = 5 > 3 = \ord(\bar{\zeta})$.  We begin with a lemma that collects the information about the $c_{i,j}$ with small indices that has been gathered thus far and then a proposition that shows that if exactly one of $\ord(\bar{\zeta})$ and $r-1$ is equal to $3$, then $c_{1,1} \in \Z_p^\times$.

\begin{lemma}\label{lem:gathercijinfo}
Assume that $\ord(\bar{\zeta}) = 3$ or $r-1 = 3$.  Then $c_{i,j} \in p\Z_p$ whenever either $j = 0$ and $0 \leq i \leq r-2$ or $i = 0$ and $0 \leq j \leq 1$.
\end{lemma}

\begin{proof}
Note that under either assumption we have $r \geq 4$.  Indeed, if $r < 4$ then \cref{thm:mereltake2} or \cref{thm:lecouturiertake2} implies that $\ord(\bar{\zeta}) = r-1 < 3$, contradicting the first assumption.  Therefore $c_{i,0} \in p\Z_p$ for $0 \leq i \leq r-2$ by \cref{cor:relatingranktoHcoeffs}.  We have $c_{0,1} \in p\Z_p$ by \cref{cor:c01}.
\end{proof}

\begin{proposition}\label{prop:c11criterion}
Assume that $\ord(\bar{\zeta}) = 3$ or $r-1 = 3$.  If $c_{1,1} \in p\Z_p$, then $\ord(\bar{\zeta}) = r-1 = 3$.
\end{proposition}

\begin{proof}
First suppose that $\ord(\bar{\zeta}) = 3$.  Since $p \geq 5$ it follows that $\ord(\bar{\zeta}) < v(p)$ and hence $3 = \ord(\bar{\zeta}) = v(\chi(\Theta))$ by \cref{prop:calculateordzeta}.  As in the proof of \cref{thm:mereltake2}, we have that 
\begin{equation}\label{eq:3min}
3 = v(\chi(\Theta)) \geq \min_{i, j} \{v(c_{i,j}) + i + 2j\},
\end{equation}
with equality holding if one of the terms is strictly minimal.  Recall that $v(p) \geq p-1$.  By \cref{lem:gathercijinfo} and the hypothesis on $c_{1,1}$, we see that the only index $(i,j)$ where $3 \geq v(c_{i,j}) + i + 2j$ is $(3,0)$, and that requires $c_{3,0} \in \Z_p^\times$.  By \cref{cor:relatingranktoHcoeffs} it follows that $3 = r-1$, as desired.

Now suppose that $r-1 = 3$.  By \cref{cor:relatingranktoHcoeffs} it follows that $c_{3,0} \in \Z_p^\times$.  Again by \cref{lem:gathercijinfo} and the hypothesis on $c_{1,1}$, we see that the unique minimal term in \eqref{eq:3min} occurs at the index $(3,0)$ and is equal to $3$.  Thus $v(\chi(\Theta)) = 3 < v(p)$.  By \cref{prop:calculateordzeta}\ref{item:ordzetabig} we see that $\ord(\bar{\zeta}) < v(p)$ and thus $\ord(\bar{\zeta}) = v(\chi(\Theta)) = 3$ by \cref{prop:calculateordzeta}\ref{item:ordzetasmall}, as desired.
\end{proof}

Although \cref{prop:c11criterion} provides a sufficient criterion for showing that $\ord(\bar{\zeta}) = r-1$ in a higher rank case, it is somewhat unsatisfying since $c_{1,1}$ seems difficult to compute in practice.  If there is a way to make $c_{1,1}$ more explicit, it would greatly strengthen \cref{prop:c11criterion}.  In contrast, \cref{thmA:merel,thmA:lecouturier} are easily computable.  Nevertheless, we can use our techniques to provide some information about the shape of the $p$-adic Hecke fields of new-at-level-$N^2$ cuspidal eigenforms when exactly one of $\ord(\bar{\zeta})$ and $r - 1$ is equal to $3$.  We believe that this type of information is not accessible via the techniques of Merel \cite{Merel} and Lecouturier \cite{Lecouturier}.  We begin with a proposition that explains how $H(X, \zeta_{p^t} + \zeta_{p^t}^{-1} - 2)$ factors in our setting.

\begin{proposition}\label{prop:Hfactorization}
Suppose that exactly one of $\ord(\bar{\zeta})$ and $r-1$ is equal to $3$.  Fix $1 \leq t \leq s$.  The polynomial $H(X, \zeta_{p^t} + \zeta_{p^t}^{-1} - 2)$ factors in $\Z_p[\zeta_{p^t} + \zeta_{p^t}^{-1}][X]$ as 
\[
H(X, \zeta_{p^t} + \zeta_{p^t}^{-1} - 2) = h_1(X)h_2(X)
\]
with $h_1$ of degree $1$ and $h_2$ an Eisenstein polynomial of degree $r-2$.
\end{proposition}

\begin{proof}
We being by decomposing $H(X, \zeta_{p^t} + \zeta_{p^t}^{-1} - 2)$ into irreducible factors in $\Z_p[\zeta_{p^t} + \zeta_{p^t}^{-1}][X]$, say
\begin{equation}\label{eq:Hlambdaffactorization}
H(X, \zeta_{p^t} + \zeta_{p^t}^{-1} - 2) = \prod_{k = 1}^m h_k(X) 
\end{equation}
with each $h_k(X)$ irreducible (not necessarily distinct).  Note that $H(X, \zeta_{p^t} + \zeta_{p^t}^{-1} - 2)$ is a distinguished polynomial by \cref{lem:MakeHEisenstein}.  By reducing modulo a uniformizer we see that each $h_k(X)$ is also a distinguished polynomial.  

We claim that $m = 2$.  As in the proof of \cref{lem:gathercijinfo}, note that $r \geq 4$.  Thus $H(X, \zeta_{p^t} + \zeta_{p^t}^{-1} - 2)$ is reducible by \cref{prop:Hreducible} and hence $m \geq 2$.

For the other inequality, we compute the coefficient of $X$ in $H(X, \zeta_{p^t} + \zeta_{p^t}^{-1} - 2)$ in two different ways.  Write $d_k$ for the coefficient of $X$ in $h_k(X)$.  Using \eqref{eq:Hevallambdafcoeff} and \eqref{eq:Hlambdaffactorization},  the coefficient of $X$ in $H(X, \zeta_{p^t} + \zeta_{p^t}^{-1} - 2)$ is equal to
\begin{align}\label{eq:linearterm}
\sum_{j = 0}^{\frac{|\Delta| - 1}{2}} c_{1,j}(\zeta_{p^t} + \zeta_{p^t}^{-1} - 2)^j = \sum_{k = 1}^m d_k \prod_{\substack{l = 1\\ l \neq k}}^m h_l(0).
\end{align}
Let $v_t$ be the valuation on $\Z_p[\zeta_{p^t} + \zeta_{p^t}^{-1}]$ with normalization $v_t(\zeta_{p^t} + \zeta_{p^t}^{-1} - 2) = 1$.  We can calculate that the $v_t$-valuation of the lefthand side of \eqref{eq:linearterm} is 1.  Indeed, since $r \geq 4$ it follows that $c_{1,0} \not\in \Z_p^\times$ by \cref{cor:c01} and hence $v_t(c_{1,0}) \geq v_t(p) > 2$.  By inspection we have $v_t(c_{1,j}(\zeta_{p^t} + \zeta_{p^t}^{-1} - 2)) \geq 2$ for all $j \geq 2$.  Finally we use the hypothesis that exactly one of $\ord(\bar{\zeta})$ and $r-1$ is equal to $3$ to conclude that $c_{1,1} \in \Z_p^\times$ by \cref{prop:c11criterion}.  It follows that the lefthand side of \eqref{eq:linearterm} has $v_t$-valuation $1$.  On the other hand, each $h_l(0)$ is the constant term of a distinguished polynomial and hence has $v_t$-valuation at least 1.  Thus the $v_t$-valuation of $\prod_{\substack{\ell = 1\\ \ell \neq k}}^m h_l(0)$ is at least $m-1$.  Thus the righthand side of \eqref{eq:linearterm} has $v_t$-valuation $1$ only if $m \leq 2$, so $m = 2$ as desired.

The coefficient of $X$ in $H(X,\zeta_{p^t} + \zeta_{p^t}^{-1} -2)$ equals $d_1h_2(0) +d_2h_1(0)$. Since this expression has $v_t$-valuation equal to $1$ and both $h_2(0)$ and $h_1(0)$ have positive valuation, at least one of $d_1$ and $d_2$ must be a unit. Observe that the coefficient of $X$ in the distinguished polynomial $h_k(X)$ is a $p$-adic unit if and only if the degree of $h_k(X)$ equals $1$. However, since $\mathrm{deg}(H(X,\zeta_{p^t} + \zeta_{p^t}^{-1} -2)) = r-1 \geq 3$, exactly one of $h_1(X)$ and $h_2(X)$ can have degree equal to $1$. Without loss of generality, assume that degree of $h_1(X)$ equals $1$ and the degree of $h_2(X)$ equals $r-2$.  Thus, $d_1$ is a unit, while $d_2$ has positive $v_t$-valuation. Using once more that $v_t(d_1h_2(0) + d_2h_1(0)) = 1$ allows us to conclude that $v_t(h_2(0)) = 1$. Therefore, the distinguished polynomial $h_2(X)$ is an Eisenstein polynomial in $\Z_p[\zeta_{p^t} + \zeta_{p^t}^{-1}][X]$ with degree $r-2$.
\end{proof}

\cref{thmA:higherrank} is a corollary of the following theorem.

\begin{theorem}\label{thm:higherrank}
Suppose that exactly one of $\ord(\bar{\zeta})$ and $r-1$ is equal to $3$.  Let $f \in S_2(\Gamma_0(N^2))_\m^{\mathrm{new}}$ be a normalized eigenform.
The $p$-adic Hecke ring $\Z_p[f]$ of $f$ is generated over $\Z_p[\chi_f(\delta) + \chi_f^{-1}(\delta)]$ by $\lambda_f(X)$, which is a root of the Eisenstein polynomial of degree $r-2$ that is a factor of $H(X, \chi_f(\delta) + \chi_f(\delta)^{-1} - 2)$ in $\Z_p[\zeta_{p^t} + \zeta_{p^t}^{-1}][X]$.
\end{theorem}

\begin{proof}
Recall that we view $\lambda_f$ as a ring homomorphism out of $\Z_p[X,Y] \to \TT \to \Z_p[f]$ via \cref{prop:presentationT}.  Thus $\Z_p[f]$ is generated by $\lambda_f(X)$ and $\lambda_f(Y)$ as a $\Z_p$-algebra.  Set $\zeta_{p^t} = \chi_f(\delta)$ so that $\lambda_f(Y) = \chi_f(\delta) + \chi_f(\delta)^{-1} - 2 = \zeta_{p^t} + \zeta_{p^t}^{-1} - 2$.  Thus $\Z_p[f]$ is generated by $\lambda_f(X)$ over $\Z_p[\zeta_{p^t} + \zeta_{p^t}^{-1}]$.  Since $f$ is a cusp form it follows that $\lambda_f(X)$ is a root of $H(X, \zeta_{p^t} + \zeta_{p^t}^{-1} - 2)$ just as in \eqref{eq:evalHcuspform}.  By \cref{prop:Hfactorization} there are $h_1, h_2 \in \Z_p[\zeta_{p^t} + \zeta_{p^t}^{-1}][X]$ of degrees $1$ and $r - 2$, respectively, such that 
\[
H(X, \zeta_{p^t} + \zeta_{p^t}^{-1} - 2) = h_1(X)h_2(X).
\]
Moreover, $h_2(X)$ is an Eisenstein polynomial.
 
 To complete the proof, it suffices to conclude that $h_2(\lambda_f(X))=0$. Suppose not. Then $h_1(\lambda_f(X))=0$. Since $r-2 \geq 2$, the equidistribution result, \cref{prop:equidistributionofcharacters}, implies that there is an eigenform $g \neq f$ in $S_2(\Gamma_0(N^2))_\m^{\mathrm{new}}$, with $\chi_g=\chi_f$. Thus, $H(\lambda_g(X),\zeta_{p^t} + \zeta_{p^t}^{-1} -2)=0$.  Since $h_1$ has degree $1$, we must have $h_1(\lambda_g(X)) \neq 0$ as $g \neq f$. This implies $h_2(\lambda_g(X))=0$. Since $h_2(X)$ is irreducible in $\Z_p[\zeta_{p^t} + \zeta_{p^t}^{-1}][X]$, the $r-2$ Galois conjugates $g$ over $\Q_p(\zeta_{p^t} + \zeta_{p^t}^{-1})$ are also cuspidal eigenforms in $S_2(\Gamma_0(N^2))_\m^{\mathrm{new}}$, with the same associated character as $f$.  We have produced $r-1$ cuspidal eigenforms in $S_2(\Gamma_0(N^2))_\m^{\mathrm{new}}$ whose associated character is $\chi_f$, namely $f$ and the $r-2$ Galois conjugates of $g$. This contradicts \cref{prop:equidistributionofcharacters}. Thus we must have $h_2(\lambda_f(X))=0$, and the theorem follows. 
\end{proof}

\begin{corollary}\label{cor:Galoisorbits}
Suppose that exactly one of $\ord(\bar{\zeta})$ and $r-1$ is equal to $3$.  There are $s$ Galois orbits of normalized eigenforms in $S_2(\Gamma_0(N^2))_\m^{\mathrm{new}}$, indexed by the orders of elements in $\hat{\Delta} \setminus \{1\}$.  That is, if $f, g \in S_2(\Gamma_0(N^2))_\m^{\mathrm{new}}$ are normalized eigenforms, then $f$ and $g$ are in the same Galois orbit if and only if $\chi_f$ and $\chi_g$ have the same order.  If $\chi_f$ has order $p^t$, then the Hecke field of $f$ is totally ramified of degree $\frac{1}{2}(p-1)p^{t-1}(r-2)$ over $\Q_p$.
\end{corollary}

\begin{proof}
The orders of elements in $\hat{\Delta} \setminus \{1\}$ are $p^t$ for $1 \leq t \leq s$.  Since $r \geq 4$, by \cref{prop:equidistributionofcharacters} for any $\chi \in \Xi$ there exists $f \in S_2(\Gamma_0(N^2))_\m^{\mathrm{new}}$ such that $\chi_f = \chi$.  We show that $f$ and $g$ are in the same Galois orbit if and only if the orders $\chi_f$ and $\chi_g$ are equal, which implies that there are $s$ Galois orbits.

Let $f, g \in S_2(\Gamma_0(N^2))_\m^{\mathrm{new}}$ be normalized eigenforms.  Note that if there exists $\sigma \in G_{\Q_p}$ such that $g = \sigma(f)$, then $\rho_g = \sigma \circ \rho_f$.  By the definition of $\chi_f$ and $\chi_g$, it follows that $\chi_g = \sigma \circ \chi_f$, so in particular $\chi_f$ and $\chi_g$ have the same order.  Conversely, suppose that $\chi_f$ and $\chi_g$ have the same order, say $p^t$ for some $1 \leq t \leq s$.  Then there exists $\sigma \in \Gal{\Q_p(\zeta_{p^t} + \zeta_{p^t}^{-1})}{\Q_p}$ such that $\chi_g = \sigma \circ \chi_f$.  Note that $\lambda_g$, and hence $g$, is determined by the image of $X$ and $Y$.  We have that $\lambda_g(Y) = \chi_g(\delta) + \chi_g^{-1}(\delta) - 2 = \sigma(\lambda_f(Y))$.  Moreover, by \cref{thm:higherrank} we see that $\lambda_g(X)$ is a root of the unique irreducible factor --- call it $h_g(X)$ --- of $H(X, \chi_g(\delta) + \chi_g^{-1}(\delta) - 2)$ with degree $r-2$.  Since $\chi_g = \sigma \circ \chi_f$, it follows that $h_g = \sigma \circ h_f$.  Therefore we can extend $\sigma$ to an embedding $\tilde{\sigma}$ of $\Q_p(f)$ that sends the root $\lambda_f(X)$ of $h_f(X)$ to the root $\lambda_g(X)$ of $h_g(X)$.  It follows that $\tilde{\sigma} \circ \lambda_f = \lambda_g$ and hence $\tilde{\sigma}(f) = g$, as desired.

The last statement follows from 
\cref{thm:higherrank} since $[\Q_p(\zeta_{p^t} + \zeta_{p^t}^{-1}) \colon \Q_p] = \frac{1}{2}(p-1)p^{t-1}$.
\end{proof}

\section*{Acknowledgements}
This project started as part of the ``Pair of Automorphic Workshops" at the University of Oregon in 2022, supported by the National Science Foundation grant DMS-1751281 and the National Security Agency MSP conference grant H98230-21-1-0029.  The authors thank Ellen Eischen and the other workshop organizers for facilitating the workshop; this project would not have started without it.  They also thank Romyar Sharifi, who contributed to the early stages of this project through the aforementioned workshop.  The authors thank the following people for helpful conversations: Shaunak Deo, David Helm, Pedro Lemos, Robert Pollack, Preston Wake, and Carl Wang-Erickson.  Their thoughts have improved this paper.  

J.L.~gratefully acknowledges support from the National Science Foundation through the grant DMS-2301738 and from the Simons Foundation through the award MP-TSM-00002260.

B.P.'s research is partially supported by the Infosys Young Investigator Award, from the Infosys Foundation Bangalore, the SERB-MATRICS grant MTR/2022/000244, and DST FIST program 2021 [TPN - 700661].

\bibliographystyle{amsalpha}
\bibliography{biblio}

\end{document}